\pdfoutput = 1

\documentclass[twoside,11pt]{article}

\usepackage{jmlr2e}

\usepackage{newtxtext}

\usepackage{amssymb}
\usepackage{amsmath}

\usepackage{amsfonts}
\usepackage{algorithm}
\usepackage{algorithmic}

\usepackage{array}
\usepackage{multicol}
\usepackage{multirow}

\allowdisplaybreaks

\usepackage[bottom]{footmisc}

\usepackage{mathrsfs}  

\usepackage{enumitem}

\usepackage{makecell} 

\usepackage{color}
\newcounter{ToDo}
\newcounter{gaocomm}

\newcounter{Note}
\definecolor{blue-violet}{rgb}{0.54, 0.17, 0.89}
\definecolor{mygreen}{rgb}{0.0, 0.5, 0.0}
\definecolor{awesome}{rgb}{1.0, 0.13, 0.32}
\definecolor{bostonuniversityred}{rgb}{1.0, 0.0, 0.0}

\makeatletter
\newcommand{\thickhline}{%
    \noalign {\ifnum 0=`}\fi \hrule height 1pt
    \futurelet \reserved@a \@xhline
}
\newcolumntype{"}{@{\hskip\tabcolsep\vrule width 1pt\hskip\tabcolsep}}
\makeatother

\firstpageno{1}

\usepackage{eqparbox}

\usepackage{etoolbox}  
\makeatletter
\patchcmd{\algorithmic}{\addtolength{\ALC@tlm}{\leftmargin} }{\addtolength{\ALC@tlm}{\leftmargin}}{}{}
\makeatother

\begin{document}

\title{Escaping Saddle Points Faster on Manifolds via Perturbed Riemannian Stochastic Recursive Gradient}

\author{\name Andi Han \email andi.han@sydney.edu.au 
       \AND
       \name Junbin Gao \email junbin.gao@sydney.edu.au \\
       \addr Discipline of Business Analytics\\
       University of Sydney}
\maketitle

\begin{abstract}
In this paper, we propose a variant of Riemannian stochastic recursive gradient method that can achieve second-order convergence guarantee and escape saddle points using simple perturbation. The idea is to perturb the iterates when gradient is small and carry out stochastic recursive gradient updates over tangent space. This avoids the complication of exploiting Riemannian geometry. We show that under finite-sum setting, our algorithm requires $\widetilde{\mathcal{O}}\big( \frac{ \sqrt{n}}{\epsilon^2} + \frac{  \sqrt{n} }{\delta^4} + \frac{  n}{\delta^3}\big)$ stochastic gradient queries to find a $(\epsilon, \delta)$-second-order critical point. This strictly improves the complexity of perturbed Riemannian gradient descent and is superior to perturbed Riemannian accelerated gradient descent under large-sample settings. We also provide a complexity of $\widetilde{\mathcal{O}} \big(  \frac{ 1}{\epsilon^3} + \frac{ 1}{\delta^3 \epsilon^2} + \frac{1 }{\delta^4 \epsilon} \big)$ for online optimization, which is novel on Riemannian manifold in terms of second-order convergence using only first-order information.
\end{abstract}

\section{Introduction}
Consider the following finite-sum and online optimization problem defined on Riemannian manifold $\mathcal{M}$:
\begin{equation}
    \min_{x \in \mathcal{M}} F(x) := \begin{cases} \mathbb{E}_{\omega}[f(x; \omega)], & \text{ online}\\
    \frac{1}{n} \sum_{i=1}^n f_i(x), & \text{ finite-sum } \end{cases} \label{problem_formulation_begin}
\end{equation}
where $F: \mathcal{M} \xrightarrow[]{} \mathbb{R}$ is a non-convex function. Finite-sum optimization setting is a special case of online setting where $\omega$ can be finitely sampled. For simplicity of analysis, we only consider finite-sum formulation $\frac{1}{n} \sum_{i=1}^n f_i(x)$ and refer to it as online optimization when $n$ approaches infinity. Solving problem \eqref{problem_formulation_begin} globally on Euclidean space (i.e when $\mathcal{M} \equiv \mathbb{R}^d$) is NP-hard in general, letting alone for any Riemannian manifold. Thus many algorithms set out to find only approximate first-order critical points with small gradients. But this is usually insufficient because for non-convex problems, a point with small gradient can be either around a local minima, a local maxima or a saddle point. Therefore, to avoid being trapped by saddle points (and possibly local maxima), we need to find approximate second-order critical points (see Definition \ref{epsilon_second_order_point}) with small gradients as well as nearly positive semi-definite Hessians. 

Second-order algorithms can explore curvature information via Hessian and thus escape saddle points by construction. However, it is always desired that simpler algorithms using only gradient information are designed for this purpose, because access to Hessian can be costly and sometimes unavailable in real applications. It turns out that gradient descent (GD) can inherently escape saddle points, yet requiring exponential time \citep{DuGDEscape2017}. Instead, \cite{JinPGD2017} showed that by injecting isotropic Gaussian noise to iterates, the perturbed gradient descent (PGD) can reach approximate $\epsilon$-second-order stationarity within $\widetilde{\mathcal{O}}(n\epsilon^{-2})$ stochastic gradient queries with high probability. Similarly, a perturbed variant of stochastic gradient algorithm (PSGD) requires a complexity of $\widetilde{\mathcal{O}}(\epsilon^{-4})$ \citep{JinPGD_PSGD2019}. The complexities match that of vanilla GD and SGD to find approximate first-order stationary points up to some 
poly-logarithmic factors. 

When considering problem \eqref{problem_formulation_begin} on any arbitrary manifold, we require Riemannian gradient $\text{grad}F(x) \in T_x\mathcal{M}$ as well as a retraction $R_x: T_x\mathcal{M} \xrightarrow[]{} \mathbb{R}$ that allows iterates to be updated on the manifold with direction determined by Riemannian gradient. Defining a `pullback' function $\hat{F}_x:= F \circ R_x: T_x\mathcal{M} \xrightarrow{} \mathbb{R}$, \cite{CriscitielloEESPManifold2019} extended the idea of perturbed gradient descent (referred to as PRGD) by executing perturbations and the following gradient steps on the tangent space. Given that $\hat{F}_x$ is naturally defined on a vector space, the analysis can be largely drawn from \citep{JinPGD2017}, with similar convergence guarantees. Recently, \cite{CriscitielloAFOM2020} also proposed  perturbed Riemannian accelerated gradient descent (PRAGD), a direct generalization of its Euclidean counterpart originally proposed in \citep{JinPAGD2018}. It achieves an even lower complexity of $\widetilde{\mathcal{O}}(n \epsilon^{-7/4})$ compared to perturbed GD. However, these algorithms are sub-optimal under finite-sum settings and even fail to work for online problems where full gradient is inaccessible. 

When objective can be decomposed into component functions as in problem \eqref{problem_formulation_begin}, variance reduction techniques can improve gradient complexity of both GD and SGD \citep{ReddiSVRG2016,NguyenSRG2017,FangSPIDER2018}. The main idea is to exploit past gradient information to correct for deviation of current stochastic gradients, by comparing either to a fixed snapshot point (SVRG, \cite{ReddiSVRG2016}) or recursively to previous iterates (SRG/SPIDER, \cite{NguyenSARAH2017,FangSPIDER2018}). Notably, the gradient complexities $\mathcal{O}(\sqrt{n}\epsilon^{-2})$ and $\mathcal{O}(\epsilon^{-3})$ achieved by SRG and SPIDER are optimal to achieve first-order guarantees under finite-sum and online settings respectively \citep{FangSPIDER2018,ArjevaniLBOnline2019}. Motivated by the success of variance reduction and the simplicity of adding perturbation to ensure second-order convergence, \cite{LiSSRGD2019} proposed perturbed stochastic recursive gradient method (PSRG), which can escape saddle points faster than perturbed GD and SGD. In this paper, we generalize PSRG to optimization problems defined on Riemannian manifolds with the following contributions:
\begin{itemize}
    \item Our proposed method PRSRG is the first simple stochastic algorithm that achieves second-order guarantees on Riemannian manifold using only first-order information. That is, the algorithm only requires simple perturbations and does not involve any negative curvature exploitation as in PRAGD. Our algorithm adopts the idea of tangent space steps as in \citep{CriscitielloEESPManifold2019}, which results in simple analysis for convergence. 
    \item Our complexity is measured against $(\epsilon,\delta)$-second-order stationarity, which is more general than $\epsilon$-second-order stationarity where $\delta = \mathcal{O}(\sqrt{\epsilon})$. Under finite-sum setting, PRSRG strictly improves the complexity of PRGD by $\widetilde{\mathcal{O}}(\min\{ n/\epsilon^{1/2},$ $\sqrt{n}/\epsilon^2 \})$ and is superior to PRAGD when $n \geq \epsilon^{-1/2}$. We also provide a complexity of PRSRG under online setting, which is novel for Riemannian optimization. 
\end{itemize}

\section{Other related work}
Following the work by \citep{JinPGD2017}, except for the perturbed SGD \citep{JinPGD_PSGD2019} and perturbed accelerated gradient method \citep{JinPAGD2018}, \cite{GePSVRG2019} showed that SVRG with perturbation (PSVRG) is sufficient to escape saddle points and a stabilized version is also developed to further improve the dependency on the Hessian Lipschitz constant. However, its complexity is strictly worse than PSRG \citep{LiSSRGD2019}. We suspect that, with similar tangent space trick, PSVRG can be generalized to Riemannian manifolds with little efforts. 

Another line of research is to incorporate negative curvature search subroutines \citep{XuNEON2018,AllenNEON22018} in some classic first-order algorithms, such as GD and SGD and also variance-reduction algorithms including SVRG \citep{AllenNEON22018}, SNVRG \citep{ZhouSNVRG_NEON2018} and SPIDER \citep{FangSPIDER2018}. It usually requires access to Hessian-vector products when searching for the smallest eigenvector direction \citep{XuNEON2018}. Even though \cite{AllenNEON22018} managed to only use first-order information for the same goal, it still involves an iterative process that can be difficult to implement in practice. 

To solve general optimization problems on Riemannian manifold, gradient descent and stochastic gradient descent have been generalized with provable guarantees \citep{BoumalRGD2019,BonnabelSGD2013,HosseiniSGD2020}. Some acceleration and variance reduction techniques have also been considered for speeding up the convergence of GD and SGD \citep{ZhangRAGD2018,AhnRAGD2020,ZhangRSVRG2016, KasaiRSRG2018}. These first-order methods however can only guarantee first-order convergence. To find second-order critical points, Newton's methods, particularly the famous globalization variants, trust region and cubic regularization have been extended to manifold optimization, achieving similar second-order guarantees \citep{BoumalRGD2019, AgarwalARC2018}. 

Finally, we note that \cite{SunPRGD2019} also proposed perturbed gradient descent on manifold. The perturbations and the following updates are performed directly on manifold, which is contrastly different to tangent space steps in this paper. The preference of tangent space steps over manifold steps is mainly due to the complexity of analysis. That is, \cite{SunPRGD2019} requires more complicated geometric results as well as the use of exponential map given its natural connection to geodesics and Riemannian distance. By executing all the steps on tangent space, we can bypass these results while some regularity conditions should be carefully managed.

\begin{table}[!t]
\setlength{\tabcolsep}{8pt}
\renewcommand{\arraystretch}{1.8}
\centering
\caption{Comparison of stochastic gradient complexity to achieve second-order guarantees}
\begin{tabular}{c | c |c c}
\thickhline
\textit{}                   &       & \makecell{$\epsilon$-second-order \\ stationarity} & \makecell{$(\epsilon, \delta)$-second-order \\ stationarity} \\\hline
\multirow{3}{*}{Finite-sum} & \makecell{PRGD \\ \citep{CriscitielloEESPManifold2019}}  &  $\widetilde{\mathcal{O}}\big(\frac{n}{\epsilon^2} \big)$             &  ---                      \\
                            & \makecell{PRAGD \\ \citep{CriscitielloAFOM2020}}&  $\widetilde{\mathcal{O}}\big( \frac{n}{\epsilon^{7/4}} \big)$                &  ---                   \\
                            & PRSRG (this work) & $\widetilde{\mathcal{O}}\big( \frac{n}{\epsilon^{3/2}} + \frac{\sqrt{n}}{\epsilon^2} \big)$                & $\widetilde{\mathcal{O}}\big( \frac{ \sqrt{n}}{\epsilon^2} + \frac{  \sqrt{n} }{\delta^4} + \frac{  n}{\delta^3}\big)$                       \\\hline
Online                      & PRSRG (this work) & $\widetilde{\mathcal{O}}\big( \frac{ 1}{\epsilon^{7/2}} \big)$               & $\widetilde{\mathcal{O}}\big( \frac{ 1}{\epsilon^3} + \frac{ 1}{\delta^3 \epsilon^2} + \frac{1 }{\delta^4 \epsilon}\big)$                  \\
\thickhline
\end{tabular}
\end{table}

\section{Preliminaries}
Here we start with a short review of some preliminary definitions. Readers can refer to \citep{AbsilOPTManifold2009} for more detailed discussions on manifold geometry. We consider $\mathcal{M}$ to be a $d$-dimensional Riemannian manifold. The tangent space $T_x\mathcal{M}$ for $x \in \mathcal{M}$ is a $d$-dimensional vector space. $\mathcal{M}$ is equipped with an inner product $\langle \cdot, \cdot \rangle_x$ and a corresponding norm $\| \cdot\|_x$ on each tangent space $T_x\mathcal{M}$, $\forall \, x \in \mathcal{M}$. Tangent bundle is the union of tangent spaces defined as $T\mathcal{M}:= \{ (x, u) : x \in \mathcal{M}, u \in T_x\mathcal{M} \}$. Riemannian gradient of a function $F:\mathcal{M}\xrightarrow[]{} \mathbb{R}$ is the vector field $\text{grad}F$ that uniquely satisfies $\text{D} F(x)[u] = \langle \text{grad}F(x), u \rangle_x$, for all $(x, u) \in T\mathcal{M}$ where D$F(x)[u]$ is the directional derivative of $F(x)$ along $u$. Riemannian Hessian of $F$ is the covariant derivative of grad$F$ defined as that for all $(x, u) \in T\mathcal{M}$, it satisfies $\text{Hess}F(x)[u] = \nabla_u \text{grad}F(x)$, where $\nabla$ is the Riemannian connection (also known as Levi-Civita connection). Note that we also use $\nabla$ to represent differentiation on vector space, which is a special case of covariant derivative. 

Retraction $R_x: T_x\mathcal{M} \xrightarrow{} \mathcal{M}$ maps a tangent vector to manifold surface by satisfying (i) $R_x(0) = x$, and (ii) $\text{D} R_x(0)$ is the identity map. $\text{D} R_x(u)$ represents the differential of retraction, denoted as $T_{u}: T_x\mathcal{M} \xrightarrow{} T_{R_x(u)}\mathcal{M}$. Exponential map is a special instance of retraction and hence our results in this paper can be trivially generalized to this particular retraction. Define the pullback function $\hat{F}_x := F \circ R_x: T_x\mathcal{M} \xrightarrow{} \mathbb{R}$. That is, $\hat{F}_x(u) = \frac{1}{n} \sum_{i=1}^n \hat{f}_{i, x}(u)$ $= \frac{1}{n}\sum_{i=1}^n {f}_i (R_x(u))$, where we also define the pullback components as $\hat{f}_{i, x} := f_i \circ R_x$. Given that the domain of pullback function is a vector space, we can represent its gradient $\nabla \hat{F}_x$ and Hessian $\nabla^2 \hat{F}_x$ in terms of Riemannian gradient and Hessian as well as the differentiated retraction $T_{u}$ in the following Lemma (see Lemma 2.5 in \cite{CriscitielloAFOM2020}).
\begin{lemma}[Gradient and Hessian of the pullback]
\label{define_grad_hess_pullback}
    For a twice continuously differentiable function $f: \mathcal{M} \xrightarrow{} \mathbb{R}$ and $(x, u) \in T\mathcal{M}$, we have
    \begin{equation*}
        \nabla \hat{F}_x(u) = T_{u}^* \emph{grad}F(R_x(u)) \quad \emph{ and } \quad \nabla^2 \hat{F}_{x} (u) = T_{u}^* \circ \emph{Hess} f(R_x(u)) \circ T_{u} + W_u,
    \end{equation*}
    where $T_{u}^*$ denotes the adjoint operator of $T_{u}$ and $W_u$ is a symmetric linear operator defined by $\langle W_u[\Dot{u}], \Dot{u} \rangle_x$ $= \langle \emph{grad}f(R_x(u)), c''(0) \rangle_{R_x(u)}$ with $c(t) = R_x(u + t \Dot{u})$, $\Dot{u}$ a perturbation to $u$. 
\end{lemma}

If the retraction $R_x$ is a second-order retraction, $\nabla \hat{F}_x(0)$ and $\nabla^2 \hat{F}_x(0)$ match the Riemannian gradient and Hessian at $x$. See Lemma \ref{grad_hess_pullback_second_order} (Appendix \ref{app_regularity_condition}) for more details. Vector transport $\mathcal{T}$ with respect to retraction $R_x$ is a linear map that $\mathcal{T}: T\mathcal{M} \,\oplus\, T\mathcal{M} \xrightarrow[]{} T\mathcal{M}$ satisfies (i) $\mathcal{T}_{u}[\xi] \in T_{R_x(u)}\mathcal{M}$; (ii) $\mathcal{T}_{0_x} [\xi] = \xi$. In this paper, we only consider isometric vector transport, which satisfies $\langle \xi, v \rangle_x = \langle \mathcal{T}_{u} \xi , \mathcal{T}_{u} v \rangle_{R_x(u)}$. Below are some common notations used throughout this paper.

\textbf{Notation.} Denote $\lambda_{\min}(H)$ as the minimum eigenvalue of a symmetric operator $H$ and use $\| \cdot \|$ to represent either the vector norm or the spectral norm for a matrix. We claim $h(t) = \mathcal{O}(g(t))$ if there exists a positive constant $M$ and $t_0$ such that $h(t) \leq Mg(t)$ for all $t \geq t_0$ and use $\widetilde{\mathcal{O}}(\cdot)$ to hide poly-logarithmic factors. Let $\mathbb{B}_x(r)$ be the Euclidean ball with radius $r$ to the origin of the tangent space $T_x\mathcal{M}$. That is, $\mathbb{B}_x(r) := \{ u \in T_x\mathcal{M} : \| u\|_x \leq r \}$. We use $\text{Unif}(\mathbb{B}_x(r))$ to denote the uniform distribution on the set $\mathbb{B}_x(r)$. Denote $[n] := \{1,2,..,n\}$ and let $\nabla \hat{f}_{\mathcal{I}, x}(u) = \frac{1}{|\mathcal{I}|} \sum_{i \in \mathcal{I}} \nabla \hat{f}_{i, x}(u)$ be a mini-batch gradient of the pullback component function where $\mathcal{I} \subseteq [n]$ with cardinality $|\mathcal{I}|$. Similarly, $\text{grad}f_{\mathcal{I}}(x)$ represents a mini-batch Riemannian gradient. Finally, we denote $\log(x)$ as the natural logarithm and $\log_\alpha(x)$ as logarithm with base $\alpha$.

\subsection{Assumptions}
Now we state the main assumptions as follows. The first assumption is required to bound function decrease and to ensure existence of stationary points on the manifold.

\begin{assumption}[Lower bounded function]
\label{bounded_function_assump}
    $F$ is lower bounded by $F^*$ with $F(x) \geq F^*$ for all $x \in \mathcal{M}$.  
\end{assumption}

Following \citep{CriscitielloEESPManifold2019}, we need to impose some Lipschitzness conditions on the pullback $\hat{F}_x$ because the main gradient steps (see Algorithm \ref{TSSRG_algorithm}) are performed on tangent space with respect to the pullback function. The next assumption requires both gradient and Hessian Lipschitzness of the pullback component functions $\hat{f}_{i,x}$. Note that we only require the condition to satisfy with respect to the origin of $T_x\mathcal{M}$ within a constraint ball $\mathbb{B}_x(D)$. This assumption is much weaker compared to requiring Lipschitz continuity over entire tangent space. 

\begin{assumption}[Gradient and Hessian Lipschitz]
\label{grad_hess_lip_asssump}
    Gradient and Hessian of the pullback component function is Lipschitz to the origin. That is, for all $x \in \mathcal{M}, u \in T_x\mathcal{M}$ with $\|u \| \leq D$, there exist $\ell, \rho \geq 0$ such that
    \begin{align*}
        \| \nabla \hat{f}_{i, x}(u) -  \nabla \hat{f}_{i, x}(0)  \| &\leq \ell\| u\|,\\
        \| \nabla^2 \hat{f}_{i, x}(u) -  \nabla^2 \hat{f}_{i, x}(0) \| &\leq \rho \| u \|.
    \end{align*}
    This also suggests the gradient and Hessian of the pullback function $\hat{F}_x$ satisfy the same Lipschitzness condition to the origin.
\end{assumption}

Immediately based on this assumption, we can show (in Lemma \ref{L_lipschit_assump}) that gradient of the pullback function is Lipschitz continuous in the constraint ball $\mathbb{B}_x(D)$. This result implies smoothness of the pullback function and is fundamental for analysing first-order optimization algorithms.

\begin{lemma}[$L$-Lipschitz continuous]
\label{L_lipschit_assump}
    Under Assumption \ref{grad_hess_lip_asssump}, for all $x \in \mathcal{M}$, there exists $L = \ell + \rho D$ such that $\nabla \hat{f}_{i, x}$ is $L$-Lipschitz continuous inside the ball $\mathbb{B}_x(D)$. This also implies that $\nabla \hat{F}_{x}$ is $L$-Lipschitz continuous. That is, for any $u, v \in \mathbb{B}_x(D)$, we have 
    \begin{equation*}
        \| \nabla \hat{f}_{i, x}(u) - \nabla \hat{f}_{i, x} (v)  \| \leq L \| u - v \| \quad \emph{ and } \quad \|\nabla \hat{F}_{x}(u) - \nabla \hat{F}_{ x} (v) \| \leq L \| u - v \|.
    \end{equation*}
\end{lemma}

The next assumption is to ensure the correspondence between gradient and Hessian of the pullback $\hat{F}_x$ at origin, and those of original function $F$. Note similar to \citep{CriscitielloEESPManifold2019}, we can relax this assumption to bounded initial acceleration, i.e. $\| c''(t)\| \leq \beta$ for some $\beta \geq 0$. In that case, results only differ by some constants.

\begin{assumption}[Second-order retraction]
\label{second_order_retraction_assump}
    $R_x$ is a second-order retraction such that for all $(x, u) \in T\mathcal{M}$, the retraction curve $c(t) = R_x(tu) $ has zero initial acceleration. That is $c''(t) = 0$. 
\end{assumption}

The following assumption is needed to bound the difference between differentiated retraction and vector transport. Although our algorithm does not require vector transport as all updates are completed on tangent space, the main purpose of this assumption is to establish a bound on singular value of differentiated retraction (Lemma \ref{singular_value_bound}). The Lemma can then be used to bound the difference between $\nabla \hat{F}_x(u)$ and $\text{grad}F(R_x(u))$ for \textit{any} $u \in \mathbb{B}_x(D)$. 

\begin{assumption}[Difference between differentiated retraction and vector transport]
\label{diff_vector_diff_assump}
    For any $\Bar{x} \in \mathcal{M}$, there exists a neighbourhood $\mathcal{X}$ of $\Bar{x}$ such that for all $x, y = R_x(u) \in \mathcal{X}$, there exists a constant $c_0 \geq 0$ uniformly,
    \begin{equation*}
        \| T_u - \mathcal{T}_u \| \leq c_0 \| u \| \quad \emph{ and } \quad \| T_u^{-1} - \mathcal{T}_u^{-1}\| \leq c_0 \|u\|.
    \end{equation*}
\end{assumption} 

\begin{lemma}[Singular value bound of differentiated retraction]
\label{singular_value_bound}
    For all $x, y = R_x(u) \in \mathcal{X}$ where $ u \in \mathbb{B}_x(D)$ with $D \leq \frac{1}{2 c_0}$, we have $\sigma_{\min}(T_u) \geq \frac{1}{2}$.
\end{lemma}

It is not difficult to satisfy these assumptions. For compact Riemannian manifolds with a second-order retraction and a three-times continnously differentiable function $F$, Assumptions \ref{bounded_function_assump}, \ref{grad_hess_lip_asssump} and \ref{second_order_retraction_assump} are easily satisfied (see Lemma 3.2 in \citep{CriscitielloEESPManifold2019}). Assumption \ref{diff_vector_diff_assump} can be guaranteed by requiring vector transport to be $C^0$ based on Taylor expansion \citep{HuangBFGS2015}. 

Apart from the main assumptions, one additional assumption of bounded variance is particularly important for solving online problems. 
\begin{assumption}[Uniform bounded variance]
\label{bounded_variance_assump}
Variance of gradient of the pullback component function is bounded uniformly by $\sigma^2$. That is, for all $x \in \mathcal{M}$ and $\hat{f}_{i,x}$, 
\begin{equation*}
    \| \nabla \hat{f}_{i,x}(u) - \nabla \hat{F}_x(u) \|^2 \leq \sigma^2
\end{equation*}
holds for any $u \in T_x\mathcal{M}$ such that $\| u\| \leq D$. 
\end{assumption}
This assumption is more stringent than the standard variance bound, which is in expectation. However, we can relax this assumption by requiring sub-Gaussian tails, which is sufficient to achieve a high probability bound \citep{LiSSRGD2019}. Lastly, we conclude this section by defining second-order critical points as follows. 

\begin{definition}[$(\epsilon, \delta)$-second-order and $\epsilon$-second-order critical points]
\label{epsilon_second_order_point}
    $x \in \mathcal{M}$ is an $(\epsilon, \delta)$-second-order critical point if 
    $$\|\emph{grad}F(x) \| \leq \epsilon, \quad \emph{ and } \quad \lambda_{\min} (\emph{Hess} F(x)) \geq -\delta. $$
    It is an $\epsilon$-second-order critical point if $\| \emph{grad}F(x) \| \leq \epsilon$ and $\lambda_{\min} (\emph{Hess}F(x)) \geq -\sqrt{\rho\epsilon}$. The second definition is a special case of the first one with $\delta = \sqrt{\rho \epsilon}$. 
\end{definition}

\section{Algorithm}

In this section, we introduce perturbed Riemannian stochastic recursive gradient method (\textsf{PRSRG}) in Algorithm \ref{PRSRG_algorithm} where the main updates are performed by tangent space stochastic recursive gradient (\textsf{TSSRG}) in Algorithm \ref{TSSRG_algorithm}. The key idea is simple: when gradient is large, we repeatedly execute $m$ standard recursive gradient iterations (i.e. an epoch) on tangent space before retracting back to the manifold. Essentially, we are minimizing the pullback function $\hat{F}_x$ within the constraint ball by SRG updates, which translates to minimizing $F$ within a neighbourhood of $x$. 

This process repeats until gradient is small. Then we perform the same SRG updates but at most $\mathscr{T}$ times, from a perturbed iterate with isotropic noise added on the tangent space. Notice that the small gradient condition in Line \ref{line_check_smallgrad} in Algorithm \ref{PRSRG_algorithm} is examined against $\text{grad}f_{\mathcal{B}}(x)$, where $\mathcal{B}$ contains $B$ samples drawn without replacement from $[n]$. This is to ensure that full batch gradient is computed under finite-sum setting where we choose $B = n$. Under online setting when $n$ approaches infinity, access to full gradient is unavailable and therefore we can only rely on the large-batch gradient.

\textsf{TSSRG} is mainly based on stochastic recursive gradient algorithm in \citep{NguyenSRG2017}. The algorithm adopts a double-loop structure. At the start of each outer loop (called an epoch), a full gradient (or a large-batch gradient for online optimization) is evaluated. Within each inner loop, mini-batch gradients are computed for current iterate $u_t$ and its last iterate $u_{t-1}$. Then a modified gradient $v_t$ at $u_t$ is constructed recursively based on the difference between $v_{t-1}$ and mini-batch gradient at $u_{t-1}$. That is, 
\begin{equation}
    v_t \xleftarrow{} \nabla \hat{f}_{\mathcal{I}, x}(u_{t}) - \big( \nabla \hat{f}_{\mathcal{I}, x}(u_{t-1}) - v_{t-1} \big). \label{SRG_update_tangent}
\end{equation}

\begin{algorithm}[!t]
 \caption{Perturbed Riemannian stochastic recursive gradient (\textsf{PRSRG})}
 \label{PRSRG_algorithm}
 \begin{algorithmic}[1]
  \STATE \textbf{Input:} Initialization $x_0$, step size $\eta$, inner loop size $m$, mini-batch size $b$, large batch size $B$, perturbation radius $r$, perturbation interval $\mathscr{T}$, diameter bound $D$, tolerance $\epsilon$
  \FOR{$t = 0,1,2, ...$}
  \IF[{$\mathcal{B}$ contains $B$ samples drawn randomly without replacement}]{$\| \text{grad}f_{\mathcal{B}}(x_{t}) \| \leq \epsilon$} \label{line_check_smallgrad}
      \STATE Draw $u_0 \sim \text{Unif}(\mathbb{B}_{x_{t}}(r))$.
      \STATE $x_{t + \mathscr{T}} = \textsf{TSSRG}(x_t, u_0, \eta, m, b, B, D, \mathscr{T})$.
      \STATE $t \xleftarrow[]{} t + \mathscr{T}$.
  \ELSE
  \STATE $x_{t+m} = \textsf{TSSRG}(x_t, 0, \eta, m, b, B, D, m)$.
  \STATE $t \xleftarrow[]{} t+ m$. 
  \ENDIF
  \ENDFOR
 \end{algorithmic} 
\end{algorithm}

\begin{algorithm}[!t]
 \caption{Tangent space stochastic recursive gradient \big(\textsf{TSSRG}($x, u_0, \eta, m,  b, B, D, \mathscr{T}$)\big)}
 \label{TSSRG_algorithm}
 \begin{algorithmic}[1]
  \STATE \textbf{Input:} Initialization $x$, initial perturbation $u_0$, step size $\eta$, inner loop size $m$, mini-batch size $b$, large batch size $B$, diameter bound $D$, max iteration $\mathscr{T}$
  \STATE If $u_0 = 0$ set $\textsf{perturbed} = 0$; else, set $\textsf{perturbed} = 1$.
  \FOR{$s = 0, 1, ...$} 
  \STATE $v_{sm} = \nabla \hat{f}_{\mathcal{B}, x}(u_{sm})$ 
  \FOR{$k = 1, 2, ..., m$}
  \STATE $t \xleftarrow{} sm + k$
  \STATE $u_{t} \xleftarrow[]{} u_{t-1} - \eta v_{t-1}$
  \IF{$\| u_t\| \geq D$}
  \STATE \textbf{break} with $u_{\mathscr{T}} \xleftarrow[]{} u_{t-1} - \alpha\eta v_{t-1}$, where $\alpha \in (0,1]$ such that $\|u_{t-1} - \alpha\eta v_{t-1} \| = D.$ \label{leave_ball_ensure}
  \ENDIF
  \STATE $v_t \xleftarrow{} \nabla \hat{f}_{\mathcal{I}, x}(u_{t}) - \nabla \hat{f}_{\mathcal{I}, x}(u_{t-1}) + v_{t-1}$ \hfill $\triangleright$ {$\mathcal{I}$ contains i.i.d. $b$ samples}
  \IF{(\textbf{not} \textsf{perturbed}) and $t < \mathscr{T}$} 
  \item \textbf{break} with $u_\mathscr{T} \xleftarrow{} u_t$ with probability $\frac{1}{m-k+1}$ \label{breaking_rules_tssrg}
  \ENDIF
  \IF{$t \geq \mathscr{T}$}
  \STATE \textbf{break} with $u_\mathscr{T} \xleftarrow{} u_t$
  \ENDIF
  \ENDFOR
  \STATE $u_{(s+1)m} \xleftarrow[]{} u_t$. 
  \ENDFOR
  \RETURN $R_x(u_\mathscr{T})$
 \end{algorithmic} 
\end{algorithm}

Note that we do not require any vector transporter since all gradients (of the pullback) are defined on the same tangent space. Hence, \textsf{TSSRG} is very similar to Euclidean SRG with only differences in stopping criteria, discussed as follows. When gradient is large, we perform at most $m$ updates and break the loop with probability $\frac{1}{m-k+1}$ where $k$ is the index of inner iteration (Line \ref{breaking_rules_tssrg}, Algorithm \ref{TSSRG_algorithm}). This stopping rule is equivalent to uniformly selecting from iterates within this epoch at random as the output. This is to ensure that either small gradient condition is triggered or sufficient decrease is achieved. More details can be found in Section \ref{main_proof_idea_section}. When gradient is small (i.e. around saddle point), we only break until max iteration has been reached.

Finally, we note that the Lipschitzness conditions are only guaranteed within a constraint ball of radius $D$ while taking $\mathscr{T}$ updates on tangent space can violate this requirement. Therefore, in Line \ref{leave_ball_ensure} (Algorithm \ref{TSSRG_algorithm}), we explicitly control the deviation of iterates from the origin and break the loop as soon as one iterate leaves the ball. Then we return some points on the boundary of the ball. By carefully balancing the parameters, we can show that whenever iterates escape the constraint ball, function value already has a large decrease. 

To simplify notations,  for the most time, we refer to Algorithm \ref{TSSRG_algorithm} as \textsf{TSSRG}($x, u_0, \mathscr{T}$).

\section{Main results}
In this section, we present the main complexity results of {PRSRG} in finding second-order stationary points. 

Under finite-sum setting, we choose $B = n$ and thus $\text{grad}f_{\mathcal{B}} \equiv \text{grad}F$. We set the parameters as in \eqref{parameter_set_finite} while we also require first-order tolerance $\epsilon \leq \widetilde{\mathcal{O}}(\frac{D\sqrt{L}}{m\sqrt{\eta}})$. This is not a strict condition given step size $\eta$ can be chosen arbitrarily small. The second-order tolerance $\delta$ needs to be smaller than the Lipschitz constant $\ell$ in Assumption \ref{grad_hess_lip_asssump}. Otherwise when $\delta > \ell$, any $x \in \mathcal{M}$ with small gradient $\| \text{grad}F(x)\| \leq \epsilon$ is already a $(\epsilon, \delta)$-second-order stationary point because by $\|\lambda_{\min}(\nabla \hat{F}_x(0)) \| \leq \ell$ from Assumption \ref{grad_hess_lip_asssump}, we have $\lambda_{\min}(\text{Hess}F(x)) = \lambda_{\min}(\nabla \hat{F}_x(0)) \geq - \ell \geq -\delta$.

\begin{theorem}[Finite-sum complexity]
\label{theorem_finite_sum}
Under Assumptions \ref{bounded_function_assump} to \ref{diff_vector_diff_assump}, consider finite-sum optimization setting. For any starting point $x_0 \in \mathcal{M}$ with the choice of parameters as
\begin{equation}
    \mathscr{T} = \widetilde{\mathcal{O}}\Big(\frac{1}{\delta} \Big), \quad \eta \leq \widetilde{\mathcal{O}}\Big(\frac{1}{L} \Big), \quad m = b = \sqrt{n}, \quad B = n, \quad r = \widetilde{\mathcal{O}} \Big(\min \Big\{ \frac{\delta^3}{\rho^2\epsilon}, \sqrt{\frac{\delta^3}{\rho^2L}} \Big\} \Big), \label{parameter_set_finite}
\end{equation}
suppose $\epsilon, \delta$ satisfy $\epsilon \leq \widetilde{\mathcal{O}}(\frac{D \sqrt{L}}{m \sqrt{\eta}}), \, \delta \leq \ell$ where $\ell \leq L$. Then with high probability, \textsf{PRSRG}$(x_0, \eta, m,$ $b, B, r, \mathscr{T}, D, \epsilon)$ will at least once visit an $(\epsilon, \delta)$-second-order critical point with
\begin{equation*}
    \widetilde{\mathcal{O}}\Big( \frac{\Delta L \sqrt{n}}{\epsilon^2} + \frac{\Delta \rho^2 \sqrt{n} }{\delta^4} + \frac{\Delta \rho^2 n}{\delta^3}\Big)
\end{equation*}
stochastic gradient queries, where $\Delta := F(x_0) - F^*$.
\end{theorem}
Up to some constants and poly-log factors, the complexity in Theorem \ref{theorem_finite_sum} is exactly the same as that when optimization domain is a vector space (i.e. $\mathcal{M} \equiv \mathbb{R}^d$). Indeed, our result is a direct generalization of the Euclidean counterpart where retraction $R_x(u) = x + u$ and the Lipschitzness conditions are made with respect to $F: \mathbb{R}^d \xrightarrow[]{} \mathbb{R}$. Set the same parameters as in \eqref{parameter_set_finite} except that we do not require $\epsilon$ to be small since $D = + \infty$ and require $\delta \leq L$ given $\ell = L$. Then we can recover the Euclidean perturbed stochastic recursive gradient algorithm \citep{LiSSRGD2019}.

Under online setting, the complexities of stochastic gradient queries are presented below.

\begin{theorem}[Online complexity]
\label{online_complexity_theorem}
Under Assumptions \ref{bounded_function_assump} to \ref{bounded_variance_assump}, consider online optimization setting. For any starting point $x_0 \in \mathcal{M}$ with the choice of parameters 
\begin{equation*}
    \mathscr{T} = \widetilde{\mathcal{O}}\Big(  \frac{1}{\delta} \Big), \quad \eta \leq \widetilde{\mathcal{O}} \Big( \frac{1}{L} \Big), \quad m = b = \widetilde{\mathcal{O}}\Big( \frac{\sigma}{\epsilon} \Big), \quad B = \widetilde{\mathcal{O}}\Big( \frac{\sigma^2}{\epsilon^2} \Big), \quad r = \widetilde{\mathcal{O}} \Big(\min \Big\{ \frac{\delta^3}{\rho^2\epsilon}, \sqrt{\frac{\delta^3}{\rho^2L}} \Big\} \Big),
\end{equation*}
suppose $\epsilon, \delta$ satisfy $\epsilon \leq \min \big\{ \frac{\delta^2}{\rho}, \widetilde{\mathcal{O}}( \frac{D\sqrt{L}}{m\sqrt{\eta}} ) \big\}$, $\delta \leq \ell$ where $\ell \leq L$. Then with high probability, $\textsf{PRSRG}(x_0, \eta, m, b, B, r, \mathscr{T}, D, \epsilon)$ will at least once visit an $(\epsilon, \delta)$-second-order critical point with 
\begin{equation*}
    \widetilde{\mathcal{O}}\Big( \frac{\Delta L \sigma}{\epsilon^3} + \frac{\Delta \rho^2 \sigma^2}{\delta^3 \epsilon^2} + \frac{\Delta \rho^2 \sigma}{\delta^4 \epsilon}\Big)
\end{equation*}
stochastic gradient oracles, where $\Delta := F(x_0) - F^*$. 
\end{theorem}

Note the complexities in this paper are analysed in terms of achieving $(\epsilon, \delta)$-second-order stationarity. Some literature prefers to choose $\delta = \sqrt{\rho\epsilon}$ to match the units of gradient and Hessian, following the work in \citep{NesterovCR2006}. In this case, our complexities reduce to $\widetilde{\mathcal{O}}(\frac{n}{\epsilon^{3/2}} + \frac{\sqrt{n}}{\epsilon^2})$ for finite-sum problems and $\widetilde{\mathcal{O}}(\frac{1}{\epsilon^{7/2}})$ for online problems. Compared to the optimal rate of $\mathcal{O}(\frac{\sqrt{n}}{\epsilon^2})$ and $\mathcal{O}(\frac{1}{\epsilon^3})$ in finding first-order stationary points, our rates are sub-optimal even if we ignore the poly-log factors. This nevertheless appears to be a problem for all stochastic variance reduction methods \citep{GePSVRG2019,LiSSRGD2019}.

\section{High-level proof ideas}
\label{main_proof_idea_section}
Now we provide a high-level proof road map of our proposed method in achieving second-order convergence guarantees. The main proof strategies are similar as in \citep{LiSSRGD2019}. However, we need to carefully handle the particularity of manifold geometry as well as manage the unique regularity conditions. We focus on finite-sum problems and only highlight the key differences under online setting. 

\subsection{Finite-sum setting}
We first start by showing how stochastic recursive gradient updates can achieve sublinear convergence in expectation by periodically computing a large-batch gradient. That is, from smoothness of the pullback function, we have 
\begin{align}
    \hat{F}_x(u_t) \leq \hat{F}_x(u_{t-1}) - \frac{\eta}{2} \| \nabla \hat{F}_x(u_{t-1})\|^2 + \frac{\eta}{2} \| v_{t-1} - \nabla \hat{F}_x(u_{t-1}) \|^2 - \big( \frac{1}{2\eta} - \frac{L}{2} \big) \| u_t - u_{t-1} \|^2. \label{finte_sum_key_smmothness}
\end{align}
To achieve first-order stationarity, it is sufficient to bound the variance term in expectation. That is $\mathbb{E}\| v_{t-1} -\nabla \hat{F}_x(u_{t-1}) \|^2 \leq \mathcal{O}(\mathbb{E}\| u_t - u_{t-1}\|)$ (see \cite{NguyenSRG2017}). This suggests the variance is gradually reduced when approaching optimality where gradient is small. Then by carefully choosing the parameters, we can show that for a single epoch, it satisfies that
\begin{equation}
    \mathbb{E}[\hat{F}_x(u_{sm+m}) - \hat{F}_x(u_{sm})] \leq - \frac{\eta}{2} \sum_{j= sm+1}^{sm+m} \mathbb{E}\| \nabla \hat{F}_x(u_{j-1}) \|^2. \label{expectation_convergence}
\end{equation}
Telescoping this result for all epochs and choosing the output uniformly from all iterates at random, we can guarantee that the output is an approximate first-order stationary point. This gives the optimal stochastic gradient complexity of $\mathcal{O}(\frac{\sqrt{n}}{\epsilon^2})$ by choosing $m = n = \sqrt{n}$. 

To achieve second-order stationarity, the algorithm will go through two phases: \textit{large gradients} and \textit{around saddle points}. We present two informal Lemmas corresponding to the two phases respectively, with parameter settings omitted. See Appendix for more details.  

\begin{lemma}\label{large_descent_lemma_informal}
When current iterate has large gradient, i.e. $\| \emph{grad}F(x) \| \geq \epsilon$, running $\textsf{TSSRG}(x, 0, m)$ gives three possible cases:  
    \begin{enumerate}
        \item When the iterates $\{ u_j \}_{j=1}^m$ do not leave the constraint ball $\mathbb{B}_x(D)$:
        \begin{enumerate}
            \item If at least half of the iterates in the epoch satisfy $\| \nabla \hat{F}_x (u_j) \| \leq \epsilon/2$ for $j = 1,...,m$, then with probability at least $1/2$, we have $ \| \emph{grad}F( \textsf{TSSRG}(x, 0, m) )\| \leq \epsilon$. 
            \item Otherwise, with probability at least $1/5$, we have $ F(\textsf{TSSRG}(x, 0, m)) - F(x) \leq - \frac{\eta m\epsilon^2}{32}$. 
        \end{enumerate}
        \item When one of the iterates $\{ u_j \}_{j=1}^m$ leaves the constraint ball $\mathbb{B}_x{(D)}$, with probability at least $1 - \vartheta$, we have $ F(\textsf{TSSRG}(x, 0, m)) - F(x) \leq - \frac{\eta m\epsilon^2}{32}$, where $\vartheta \in (0,1)$ can be made arbitrarily small.
    \end{enumerate}
    No matter which case occurs, $F(\textsf{TSSRG}(x, 0, m)) \leq F(x)$ with high probability.
\end{lemma}

\begin{lemma} \label{escape_stuck_region_lemma_informal} 
When current iterate is around a saddle point, i.e. $\| \emph{grad}F(x) \| \leq \epsilon$ and $\lambda_{\min} (\emph{Hess}F(x))$ $\leq - \delta$, running $\textsf{TSSRG}(x,u_0, \mathscr{T})$ with perturbation $u_0 \in \mathbb{B}_x(r)$ gives sufficient decrease of function value with high probability. That is,
    \begin{equation}
        {F}(\textsf{TSSRG}(x,u_0, \mathscr{T}) ) - F(x) \leq -\mathscr{F}, \nonumber
    \end{equation}
where $\mathscr{F} = \widetilde{\mathcal{O}}(\frac{\delta^3}{\rho^2})$.
\end{lemma}

Lemma \ref{large_descent_lemma_informal} claims that when gradient is large (phase 1), the output after running \textsf{TSSRG} with a single epoch either has a small gradient (Case 1a) or reduces function value by a sufficient amount (Case 1b) with high probability. Note that we need to explicitly address the case when iterates leave the constraint ball (Case 2). In this case, we show that function value already decreases by the same desired amount given that first-order tolerance $\epsilon$ is chosen reasonably small as in Theorem \ref{theorem_finite_sum}. Note for Case 1a, the output satisfies the small gradient condition and hence is immediately followed by perturbation and the follow-up updates in \textsf{TSSRG}$(x,u_0, \mathscr{T})$. Notice that we can only show $\| \nabla \hat{F}_x(u_m) \|$ is small. To connect to Riemannian gradient $\| \text{grad}F(R_x(u_m))\|$, we need the singular value bound of the differentiated retraction in Lemma \ref{singular_value_bound}. In other cases, function value decreases by $\mathcal{O}(\eta m \epsilon^2)$ with high probability. As a result, given that function $F$ is uniformly bounded by $F^*$, we can bound the number of such descent epochs $N_1$ by $\mathcal{O}(\frac{\Delta}{\eta m \epsilon^2})$, where $\Delta := F(x_0) - F^*$. Since we choose $\eta = \widetilde{\mathcal{O}}(\frac{1}{L})$, $m = b = \sqrt{n}$, $B = n$, the stochastic gradient complexity is computed as $N_1 (B + mb ) = $ $\widetilde{\mathcal{O}}(\frac{\Delta L \sqrt{n}}{\epsilon^2})$. 

In phase 2 where gradient is small, current iterate is either already a second-order stationary point or around a saddle point. Lemma \ref{escape_stuck_region_lemma_informal} states that running $\textsf{TSSRG}$ from any perturbation within the ball $\mathbb{B}_x(r)$ decreases function value by at least $\mathscr{F} = \widetilde{\mathcal{O}}(\frac{\delta^3}{\rho^2})$ with high probability. Again, since the optimality gap is bounded, the number of such escape epochs $N_2$ is bounded by $\widetilde{\mathcal{O}}( \frac{\rho^2 \Delta}{\delta^3})$. And similarly, the number of stochastic gradient queries is $N_2( \lceil \mathscr{T}/m \rceil n + \mathscr{T} b) = \widetilde{\mathcal{O}}(\frac{\Delta \rho^2 \sqrt{n}}{\delta^4} + \frac{\Delta \rho^2 n}{\delta^3})$, where we choose $\mathscr{T} = \widetilde{\mathcal{O}}(\frac{1}{\delta})$. Combining the complexities under phase 1 and phase 2 gives the result. For more rigorous analysis, we need to consider the number of wasted epochs where neither function decreases sufficiently nor gradient of the output is small. The complexity of such epochs however turns out to not exceed the complexities established before. Detailed proofs are included in Appendix \ref{subsec_theorem_finite}. 

Next we will briefly explain how Lemma \ref{large_descent_lemma_informal} (large gradient phase) and Lemma \ref{escape_stuck_region_lemma_informal} (around saddle point phase) are derived.

\vspace*{9pt}
\noindent\textbf{Large gradient phase.} The key result underlying Lemma \ref{large_descent_lemma_informal} is a high-probability version of \eqref{expectation_convergence}. To this end, we first need a high-probability bound for the variance term $\| v_{t} - \nabla \hat{F}_x(u_{t}) \|^2$. It is not difficult to verify that $y_t := v_t - \nabla \hat{F}_x(u_t), t = sm, ..., (s+1)m$ is a martingale sequence. As required by Azuma-Hoeffing inequality (Lemma \ref{azuma_hoeffding_lemma}, Appendix \ref{app_concentration_bound}), in order to bound $\| y_t \|$, we need to bound its difference sequence $z_t := y_t - y_{t-1}$. This difference sequence can be bounded by applying vector Bernstein inequality (Lemma \ref{bernsteinlemma}, Appendix \ref{app_concentration_bound}). After bounding $\| y_t\|$, we can substitute this result into \eqref{finte_sum_key_smmothness} to obtain 
\begin{equation}
    \hat{F}_x(u_t) - \hat{F}_x(u_{sm}) \leq - \frac{\eta}{2} \sum_{j=sm+1}^t \| \nabla \hat{F}_x(u_{j-1}) \|^2 \quad \text{ with high probability, } \label{inequality_main_descent}
\end{equation}
for $1 \leq t \leq m$. Note that we always call \textsf{TSSRG} for only one epoch each time. Therefore, it is sufficient to consider the first epoch in \textsf{TSSRG}. Next, the analysis is divided into whether iterates leave the constraint ball or not. When all iterates stay within the boundary of the ball, inequality \eqref{inequality_main_descent} suggests that if at least half of iterates in this epoch have large gradient, then we obtain a sufficient decrease. Otherwise, the output uniformly selected from the iterates in this epoch will have a small gradient with high probability. On the other hand, when one of the iterates escape the constraint ball, we still can show a sufficient decrease by a localization Lemma (Lemma \ref{improve_localize_lemma}, Appendix \ref{subsec_lemma_finite}). Specifically, we have $\| u_t\|^2 \leq \widetilde{\mathcal{O}}\big(\hat{F}_x(0) - \hat{F}_x(u_t) \big)$, which is derived from \eqref{finte_sum_key_smmothness} and the high-probability bound of the variance term. This bound implies that if iterates are far from the origin, function value already decreases a lot.

\vspace*{9pt}
\noindent \textbf{Around saddle point phase.} When the current iterate is around a saddle point, we need to show that the objective can still decrease at a reasonable rate with high probability. At high-level, we adopt the same coupling-sequence analysis originally introduced in \citep{JinPGD2017}. Define the stuck region as $\mathcal{X}_{\text{stuck}} := \{ u\in \mathbb{B}_x(r) : F(\textsf{TSSRG}(x, u, \mathscr{T}) - \hat{F}_x(u) \geq -2\mathscr{F}\}$ such that running \textsf{TSSRG} from points initialized in this region will not give sufficient function value decrease. Consider two initialization $u_0, u_0'$ that only differs in the direction of the smallest eigenvector $e_1$ of the Hessian $\text{Hess}F(x)$. That is, $u_0 - u_0' = r_0 e_1$ where $r_0 = \frac{\nu r}{\sqrt{d}}$ ($\nu \in (0,1)$ can be chosen arbitrarily small). Then we can prove that at least one of the sequences $\{ u_t\}, \{ u_t'\}$ generated by running \textsf{TSSRG} from perturbation $u_0$, $u_0'$ achieves large deviation from the initialized points within $\mathscr{T}$ steps (see Lemma \ref{small_stuck_region_lemma}). That is, with high probability,
\begin{equation}
    \exists \, t \leq \mathscr{T}, \, \max\{ \| u_t - u_0 \|, \| u_t' - u_0' \|\} \geq \widetilde{\mathcal{O}}\Big(\frac{\delta}{\rho}\Big).  \nonumber
\end{equation}
This result together with the localization Lemma indicates that at least one of the sequences also achieves high function value decrease. Particularly, we obtain $\max \{ \hat{F}_x (u_{0}) - F(\textsf{TSSRG}(x, u_0, \mathscr{T}),$ $\hat{F}_x (u_0') -F(\textsf{TSSRG}(x, u_0', \mathscr{T})\} \geq 2 \mathscr{F}$ where we note $\hat{F}_x(u_\mathscr{T}) = F(\textsf{TSSRG}(x, u_0, \mathscr{T}))$. This directly suggests that the width of stuck region $\mathcal{X}_{\text{stuck}}$ is at most $r_0$. Based on some geometric results, we know that the probability of any perturbation $u_0 \in \mathbb{B}_x(r)$ falling in the stuck region is at most $\nu$. In other words, with high probability, an arbitrarily chosen $u_0$ falls outside of stuck region. In this case, we achieve sufficient decrease of function value as $\hat{F}_x (u_{0}) - F(\textsf{TSSRG}(x, u_0, \mathscr{T}) \geq 2\mathscr{F}$. By carefully selecting the radius $r$ of the perturbation ball, we can bound the difference between $\hat{F}_x(0)$ and $\hat{F}_x(u_0)$ by $\mathscr{F}$. Finally, combining these two results yields Lemma \ref{escape_stuck_region_lemma_informal}:
\begin{equation}
   F(x) - F(\textsf{TSSRG}(x, u_0, \mathscr{T})) = \hat{F}_x(0) - \hat{F}_x(u_0) + \hat{F}_x(u_0) -  \hat{F}_x(u_\mathscr{T}) \geq - \mathscr{F}+ 2\mathscr{F} = \mathscr{F}. \nonumber
\end{equation}

\subsection{Online setting}
Consider online problems where full gradient is inaccessible. The proof roadmap is the same as in finite-sum setting. But now we have $v_{sm} = \nabla \hat{f}_{\mathcal{B},x}(u_{sm}) \neq \nabla \hat{F}_x(u_{sm})$. Most key results are relaxed with an additional term that relates to the variance of stochastic gradient (Assumption \ref{bounded_variance_assump}). 

\vspace*{9pt}
\noindent\textbf{Large gradient phase.} Under phase 1, we can show that \eqref{inequality_main_descent} holds with an additional term. That is, 
\begin{equation}
    \hat{F}_x (u_t) - \hat{F}_x(u_0) \leq -\frac{\eta}{2} \sum_{j=1}^t \| \nabla \hat{F}_x(u_{j-1}) \|^2 + \widetilde{\mathcal{O}} \Big(\frac{\eta t \sigma^2}{B} \Big) \quad \text{ with high probability. } \nonumber
\end{equation}
Note that under online setting, the small gradient condition is checked against the large-batch gradient $\text{grad}f_{\mathcal{B}}(x)$, rather than the full gradient (Line \ref{line_check_smallgrad}, Algorithm \ref{PRSRG_algorithm}). Therefore, compared to finite-sum cases, we require an extra bound on the difference between full gradient and large-batch gradient $\| \nabla \hat{f}_{\mathcal{B},x}(u_m) - \nabla \hat{F}_x(u_m) \|$. This can be obtained by Bernstein inequality. By choosing $B = \widetilde{\mathcal{O}}(\frac{\sigma^2}{\epsilon^2})$, similar results to Lemma \ref{large_descent_lemma} can be derived.

\vspace*{9pt}
\noindent \textbf{Around saddle point phase.} Under phase 2, we can obtain the same inequality as in Lemma \ref{escape_stuck_region_lemma}, with differences only in terms of parameter settings. 

\vspace*{9pt}
\noindent These results guarantee that the number of phase-1 and phase-2 epochs match those of finite-sum setting, up to some constants and poly-log factors. That is, $N_1 \leq \mathcal{O}(\frac{\Delta}{\eta m \epsilon^2})$ and $N_2 \leq \widetilde{\mathcal{O}}(\frac{\Delta \rho^2}{\delta^3})$. Following similar logic and choosing parameters as $m = b = \sqrt{B} = \widetilde{\mathcal{O}}(\frac{\sigma}{\epsilon})$, we can obtain the complexity in Theorem \ref{online_complexity_theorem}.

\section{Conclusion}
In this paper, we generalize perturbed stochastic recursive gradient method to Riemannian manifold by adopting the idea of tangent space steps introduced in \citep{CriscitielloEESPManifold2019}. This avoids using any complicated geometric result as in \citep{SunPRGD2019} and thus largely simplifies the analysis. We show that up to some constants and poly-log factors, our generalization achieves the same stochastic gradient complexities as the Euclidean version \citep{LiSSRGD2019}. Under finite-sum setting, our result is strictly superior to PRGD by \cite{CriscitielloEESPManifold2019} and to PRAGD by \cite{CriscitielloAFOM2020} for large-scale problems. We also prove an online complexity, which is, to the best of our knowledge, the first result in finding second-order stationary points only using first-order information.

\newpage
\appendix

\section{Useful concentration bound}
This section presents some useful concentration bounds on vector space, which is used to derive high probability bounds for sequences defined on tangent space of the manifold. 

\label{app_concentration_bound}
\begin{lemma}[Vector Bernstein inequality, \cite{TroppVBbounds2012}]
\label{bernsteinlemma}
    Given a sequence of independent vector random variables $\{ v_k \}$ in $\mathbb{R}^d$, which satisfies $\| v_k - \mathbb{E}[v_k] \| \leq R$ almost surely, then for $\varsigma \geq 0$
    \begin{align*}
        \emph{Pr} \{\| \sum_{k} (v_k - \mathbb{E}[v_k]) \| \geq \varsigma \} \leq (d+1) \exp \big( \frac{-\varsigma^2/2}{\sigma^2 + R\varsigma/3} \big) 
    \end{align*}
    where $\sigma^2 := \sum_{k} \mathbb{E} \| v_k - \mathbb{E}[v_k] \|^2.$
\end{lemma}

\begin{lemma}[Azuma-Hoeffding inequality, \cite{HoeffdingBound1994,ChungBound2006}]
\label{azuma_hoeffding_lemma}
    Consider a vector-valued martingale sequence $\{y_k \}$ and its corresponding martingale difference sequence $\{ z_k \}$ in $\mathbb{R}^d$. If $\{z_k \}$ satisfies $\| z_k\| = \| y_k - y_{k-1} \| \leq c_k$ almost surely, then for $\varsigma \geq 0$, 
    \begin{equation}
        \emph{Pr} \{ \| y_k - y_0 \| \geq \varsigma \} \leq (d+1) \exp \Big( \frac{-\varsigma^2}{8\sum_{i=1}^k c_i^2} \Big). \label{prob_1_AH_inequality}
    \end{equation}
    If $\| z_k \| = \|y_k - y_{k-1}\| \leq c_k$ with probability $1 - \zeta_k$, then for $\varsigma \geq 0$, 
    \begin{equation}
        \emph{Pr} \{ \| y_k - y_0 \| \geq \varsigma \} \leq (d+1) \exp \Big( \frac{-\varsigma^2}{8\sum_{i=1}^k c_i^2} \Big) + \sum_{i=1}^k \zeta_i \nonumber
    \end{equation}
\end{lemma}

\section{Regularity conditions on Riemannian manifold}
\label{app_regularity_condition}
In this section, we prove some regularity conditions on manifolds, which are fundamental for Riemannian optimization, as seen in several literature \citep{CriscitielloAFOM2020,BoumalIntroductionMO2020}.

\vspace*{9pt}
\noindent \textbf{Lemma \ref{L_lipschit_assump} ($L$-Lipschitz continuous)} \textit{  Under Assumption \ref{grad_hess_lip_asssump}, for all $x \in \mathcal{M}$, there exists $L = \ell + \rho D$ such that $\nabla \hat{f}_{i, x}$ is $L$-Lipschitz continuous inside the ball $\mathbb{B}_x(D)$. This also implies that $\nabla \hat{F}_{x}$ is $L$-Lipschitz continuous. That is, for any $u, v \in \mathbb{B}_x(D)$, we have}
    \begin{equation*}
        \| \nabla \hat{f}_{i, x}(u) - \nabla \hat{f}_{i, x} (v)  \| \leq L \| u - v \| \quad \text{ and } \quad \|\nabla \hat{F}_{x}(u) - \nabla \hat{F}_{ x} (v) \| \leq L \| u - v \|.
    \end{equation*}
\begin{proof}
    The proof of $\nabla \hat{f}_{i, x}$ being Lipschitz continuous is the same as in \cite{CriscitielloEESPManifold2019}. We include it here for completeness. From Assumption \ref{grad_hess_lip_asssump}, we have $\| \nabla^2 \hat{f}_{i,x}(0) \| \leq \ell$. Thus, 
    \begin{equation}
        \| \nabla^2 \hat{f}_{i,x}(u) \| \leq \| \nabla^2 \hat{f}_{i,x} (u) - \nabla^2 \hat{f}_{i,x}(0) \| + \| \nabla^2 \hat{f}_{i,x}(0) \| \leq \rho \| u\| + \ell \leq \ell + \rho D = L. \nonumber
    \end{equation}
    Hence, for any $u, v \in \mathbb{B}_x(D)$, we obtain
    \begin{equation}
        \| \nabla \hat{f}_{i, x}(u) - \nabla \hat{f}_{i, x} (v) \| = \| \int_{0}^1 \nabla^2 \hat{f}_{i,x} (v + (u-v)\tau)[u - v] d\tau \| \leq L \| u - v\|.   \nonumber
    \end{equation}
    This implies that the full gradient $\nabla \hat{F}_{x}$ is Lipschitz continuous because for any $u, v \in \mathbb{B}_x(D)$
    \begin{equation}
        \| \nabla \hat{F}_x(u) - \nabla \hat{F}_x(v)\| = \| \frac{1}{n}\sum_{i=1}^n \big( \nabla \hat{f}_{i, x}(u) - \nabla \hat{f}_{i,x}(v) \big)\| = \| \nabla \hat{f}_{i, x}(u) - \nabla \hat{f}_{i,x}(v) \| \leq L \| u - v \|. \nonumber
    \end{equation}
    This completes the proof.
\end{proof}

\noindent \textbf{Lemma \ref{singular_value_bound} (Singular value bound of differentiated retraction)}
\textit{ For all $x, y = R_x(u) \in \mathcal{X}$ where $ u \in \mathbb{B}_x(D)$ with $D \leq \frac{1}{2 c_0}$, we have $\sigma_{\min}(T_u) \geq \frac{1}{2}$. }

\vspace*{9pt}
\begin{proof}
    From Assumption \ref{diff_vector_diff_assump}, we have $\| T_u - \mathcal{T}_u \| \leq c_0 \| u\| \leq c_0 D \leq \frac{1}{2}$. Therefore $\sigma_{\min}(T_u) = \min_{\| \xi \| = 1} \| T_u\xi \| \geq \min_{\xi} \| \mathcal{T}_u \xi \| - \| (T_u - \mathcal{T}_u) \xi \| \geq 1-\frac{1}{2} = \frac{1}{2}$, where the last inequality uses the fact that all singular values of an isometric operator are $1$.
\end{proof}

\begin{lemma}[Gradient and Hessian of the pullback under second-order retraction]
\label{grad_hess_pullback_second_order}
Given a second order retraction $R_x : T_x\mathcal{M} \xrightarrow[]{} \mathcal{M}$, both gradient and Hessian of the pullback function $\hat{f}_x := f \circ R_x: T_x\mathcal{M} \xrightarrow{} \mathbb{R}$ evaluated at the origin of $T_x\mathcal{M}$ match the Riemannian gradient and Hessian of $f$. That is, for all $x \in \mathcal{M}$,
\begin{equation*}
    \nabla \hat{f}_x(0) = \emph{grad} f(x), \, \emph{ and } \, \nabla^2 \hat{f}_x(0) = \emph{Hess} f(x)
\end{equation*}
\end{lemma}
\begin{proof}
    The proof is mainly based on \citep{BoumalIntroductionMO2020} and we include it here for completeness. First note that for any retraction (not necessarily second-order), the gradient is matched. That is, for any $u \in T_x\mathcal{M}$, by chain rule, 
    \begin{equation*}
        \text{D} \hat{f}_x(0)[u] = \text{D} (f \circ R_x)(0)[u] = \text{D} f(R_x(0))[\text{D} R_x(0)[u]] = Df(x)[u],
    \end{equation*}
    where we use the definition of retraction where $R_x(0) = x$ and $\text{D} R_x(0)[u] = u$. Then we can use the definition of Riemannian gradient and its uniqueness property to show the result. Next we prove the second result. Consider a second-order Taylor expansion of $f$ from $x$ to $R_x(u)$ along the retraction curve as 
    \begin{align}
        \hat{f}_x(u) = f(R_x(u)) &= f(x) + \langle \text{grad}f(x), u\rangle + \frac{1}{2} \langle \text{Hess}f(x)[u],u \rangle + \frac{1}{2} \langle \text{grad}f(x), c''(0) \rangle + \mathcal{O}(\| u\|^3) \nonumber\\
        &=  f(x) + \langle \text{grad}f(x), u\rangle + \frac{1}{2} \langle \text{Hess}f(x)[u],u \rangle + \mathcal{O}(\| u\|^3), \label{qwewewrere}
    \end{align}
    due to $c''(0) = 0$ for second-order retraction. Also since $\hat{f}_x$ is a `classic' function from vector space to real number, we can use a classic Taylor expansion of this function as 
    \begin{align}
        \hat{f}_x(u) = \hat{f}_x(0) + \langle\nabla \hat{f}_x(0), u \rangle + \frac{1}{2} \langle  \nabla^2 \hat{f}_x(0)[u], u \rangle + \mathcal{O}(\| u\|^3). \label{pmuykujk}
    \end{align}
    Given that we already have $\nabla \hat{f}_x(0) = \text{grad}f(x)$, we have by comparing \eqref{pmuykujk} with \eqref{qwewewrere}, $\nabla^2 \hat{f}_x(0) = \text{Hess}f(x)$. 
\end{proof}

\section{Proof for finite-sum setting}
In this section, we prove the main complexity results under finite-sum setting. In this case, we choose $B = n$ and hence from Algorithm \ref{TSSRG_algorithm}, we have access to the full gradient and $\text{grad}f_{\mathcal{B}} \equiv \text{grad}F$. We start by showing the proof for the main complexity Theorem in subsection \ref{subsec_theorem_finite}. Then we prove some key lemmas necessary to derive the Theorem in \ref{subsec_lemma_finite}.

\subsection{Proof for main Theorem}
\label{subsec_theorem_finite}

\textbf{Theorem \ref{theorem_finite_sum} (Finite-sum complexity)}
\textit{Under Assumptions \ref{bounded_function_assump} to \ref{diff_vector_diff_assump}, consider finite-sum optimization setting. For any starting point $x_0 \in \mathcal{M}$ with the choice of parameters as
\begin{equation*}
    \mathscr{T} = \widetilde{\mathcal{O}}\Big(\frac{1}{\delta} \Big), \quad \eta \leq \widetilde{\mathcal{O}}\Big(\frac{1}{L} \Big), \quad m = b = \sqrt{n}, \quad B = n, \quad r = \widetilde{\mathcal{O}} \Big(\min \Big\{ \frac{\delta^3}{\rho^2\epsilon}, \sqrt{\frac{\delta^3}{\rho^2L}} \Big\} \Big),
\end{equation*}
suppose $\epsilon, \delta$ satisfy $\epsilon \leq \widetilde{\mathcal{O}}(\frac{D \sqrt{L}}{m \sqrt{\eta}}), \, \delta \leq L$. Then with high probability, \textsf{PRSRG}$(x_0, \eta, m, b, B, r, \mathscr{T}, D, \epsilon)$ will at least once visit an $(\epsilon, \delta)$-second-order critical point with
\begin{equation*}
    \widetilde{\mathcal{O}}\Big( \frac{\Delta L \sqrt{n}}{\epsilon^2} + \frac{\Delta \rho^2 \sqrt{n} }{\delta^4} + \frac{\Delta \rho^2 n}{\delta^3}\Big)
\end{equation*}
stochastic gradient queries, where $\Delta := F(x_0) - F^*$.}

\vspace*{9pt}
\begin{proof}
Here are all possible cases when running the main algorithm $\textsf{PRSRG}$. Notice that we need to explicitly discuss the case where iterates escape the constraint ball $\mathbb{B}_x(D)$ within an epoch. Under large gradient situation, when iterates leave the constraint ball, it achieves a large function decrease with high probability (hence labelled as descent epoch). Around saddle point, when iterates leave the constraint ball, it already decreases function value with probability 1 (hence merged with the case when iterates fall inside the ball).  
\begin{itemize}
    \item {Large gradients where ${\| \text{grad}F(x) \| \geq \epsilon}$}.
    \begin{enumerate}
        \item \textit{Type-1 descent epoch}: Iterates escape the constraint ball.
        \item Iterates do not escape the constraint ball. 
        \begin{enumerate}
            \item \textit{Type-2 descent epoch}: At least half of iterates in current epoch have pullback gradient larger than $\epsilon/2$. 
            \item \textit{Useful epoch}: At least half of iterates in current epoch have pullback gradient no larger than $\epsilon/2$ and output $\Tilde{x}$ from current epoch has gradient no larger than $\epsilon$. (Since output satisfies small gradient condition, the next epoch will run \textsf{TSSRG}$(\Tilde{x}, u_0, \mathscr{T})$ to escape saddle points).
            \item \textit{Wasted epoch}: At least half of iterates in current epoch have pullback gradient no larger than $\epsilon/2$ and output $\Tilde{x}$ from current epoch has gradient larger than $\epsilon$.
        \end{enumerate}
    \end{enumerate}
    \item {Around saddle points where ${\| \text{grad}F(x) \| \leq \epsilon}$ and ${\lambda_{\min} (\text{Hess}(x)) \leq -\delta }$}
    \begin{enumerate}[start=3]
        \item \textit{Type-3 descent epoch}: Current iterate is around saddle point.
    \end{enumerate}
\end{itemize}
First, because output $\Tilde{x}$ for current epoch is randomly selected from the iterates, the probability of a \textit{wasted epoch} is at most $1/2$. Also due to the independence of each \textit{wasted epoch}, with high probability, \textit{wasted epoch} occurs consecutively at most $\widetilde{\mathcal{O}}(1)$ times before either a \textit{descent epoch} (either type 1 or 2) or a \textit{useful epoch}. \footnote{That is, suppose the probability of at least half of iterates not escaping the constraint ball have large small gradient is $\theta$. Therefore, the probability of $X$ consecutive occurrences of \textit{wasted epoch} is $(\frac{\theta}{2})^X$. Then with high probability of $1 - \iota$, there exists at least one \textit{useful epoch} or \textit{type-2 descent epoch} in $X = \mathcal{O}(-\log(\iota)) = \widetilde{\mathcal{O}}(1)$.} Hereafter, we use $N_1$, $N_2$ and $N_3$ to respectively represent three types of descent epoch. 
    
    Consider \textit{Type-1 descent epoch}. From Case 2 in Lemma \ref{large_descent_lemma}, with probability $1 - \vartheta$, function value decreases by at least $\frac{\eta m \epsilon^2}{32}$ and with high probability the function value decreases. Hence by the standard concentration, after $N_1$ such epochs, function value is reduced by $\mathcal{O}(\eta m \epsilon^2 N_1)$ with high probability. Given that $F(x)$ is bounded by $F^*$, the decrease cannot exceed $\Delta := F(x_0) - F^*$. Therefore, $N_1 \leq \mathcal{O}(\frac{\Delta}{\eta m \epsilon^2})$. Similarly, for \textit{Type-2 descent epoch}, $N_2 \leq \mathcal{O}(\frac{\Delta}{\eta  m \epsilon^2})$. 
    
    Consider \textit{Useful epoch} where output $\| \text{grad}F(\Tilde{x}) \| \leq \epsilon$. If further $\lambda_{\min}(\text{Hess}F(\Tilde{x})) \geq - \delta$, then $\Tilde{x}$ is already an $(\epsilon, \delta)$-second-order critical point. Otherwise, a \textit{Useful epoch} is followed by \textit{Type-3 descent epoch} around saddle points. From Lemma \ref{escape_stuck_region_lemma}, we know that function value decreases by $\mathscr{F} = \frac{\delta^3}{2c_3 \rho^2}$ with high probability. Similar to argument for other types of descent epoch, $N_3 \leq \widetilde{ \mathcal{O}}(\frac{ \rho^2 \Delta}{\delta^3})$, where we omit $c_3 = \widetilde{\mathcal{O}}(1)$. 
    
    Hence, we have the following stochastic gradient complexity: 
    \begin{align}
        &(N_1 + N_2)\big(\widetilde{\mathcal{O}}(1) \big( n + mb\big) + n + mb  \big) + N_3 \big( \widetilde{\mathcal{O}}(1) \big( n + mb \big) + \lceil \mathscr{T}/m \rceil n + \mathscr{T} b \big) \nonumber\\
        &\leq \widetilde{\mathcal{O}}\Big( \frac{\Delta}{\eta m \epsilon^2} n +  \frac{\rho^2 \Delta}{\delta^3} n + \frac{\rho^2 \Delta}{\delta^3} (n + \frac{\sqrt{n}}{\delta}) \Big) \nonumber\\
        &\leq \widetilde{\mathcal{O}}\Big( \frac{\Delta L \sqrt{n}}{\epsilon^2} + \frac{\Delta \rho^2 \sqrt{n} }{\delta^4} + \frac{\Delta \rho^2 n}{\delta^3}\Big), \nonumber
    \end{align}
    where $\mathscr{T} = \widetilde{\mathcal{O}}(\frac{1}{\delta})$, $m = b = \sqrt{n}$, $\eta = \widetilde{\mathcal{O}}(\frac{1}{L})$.
\end{proof}

\subsection{Proof for key Lemmas}
\label{subsec_lemma_finite}
Organization of these Lemmas are as follows. In Lemma \ref{hpBound_estimation_error}, we first prove a high probability bound for the estimation error of the modified gradient $v_t$. This is to replace the expectation bound commonly used in deriving first-order guarantees. Lemma \ref{improve_localize_lemma} is to show that the iterates that deviate a lot from the initialized point also achieve large function value decrease. These two results are subsequently used to derive Lemma \ref{large_descent_lemma}, a descent Lemma for large gradient phase. 

Under saddle point phase, we first show the proof that at least one of the coupled sequences achieve large deviation from the initialization (Lemma \ref{small_stuck_region_lemma}). This then translates to a sufficient function value decrease in Lemma \ref{descent_around_saddle_lemma}. Finally, it can be shown in Lemma \ref{escape_stuck_region_lemma} that with high probability, the iterates can escape saddle point and decrease function value by a desired amount.

\begin{lemma}[High probability bound on estimation error] 
\label{hpBound_estimation_error}
Under Assumption \ref{grad_hess_lip_asssump}, we have the following high probability bound for estimation error of the modified gradient under finite-sum setting. That is, for $sm+1 \leq t \leq (s+1)m$, 
    \begin{equation}
        \| v_t - \nabla \hat{F}_x(u_t) \| \leq  \frac{\mathcal{O}\big(\log^{\frac{3}{2}}(d/\vartheta) \big) L}{\sqrt{b}} \sqrt{\sum_{j=sm+1}^t \| u_j - u_{j-1} \|^2}  \quad \text{ with probability } 1 - \vartheta, \nonumber
    \end{equation}
\end{lemma}
\begin{proof}
For simplicity of notation, consider a single epoch from $k = 1,...,m, \forall{s}$. Consider two sequences on $T_x\mathcal{M} \subseteq \mathbb{R}^d$, defined as $y_k := v_k - \nabla \hat{F}_{x}(u_k)$ and $z_k := y_k - y_{k-1}$. It is easily verified that $y_k$ is a martingale and $z_k$ is its difference sequence. That is, denote $\mathcal{F}_{t}$ is a filtration at time $t$. $\mathbb{E}[y_t | \mathcal{F}_{t-1}] = \mathbb{E}_{\mathcal{F}_{t-1}}[\nabla \hat{f}_{\mathcal{I}, x}(u_{t}) - \nabla \hat{f}_{\mathcal{I}, x}(u_{t-1}) + v_{t-1} -  \nabla \hat{F}_x(u_t)] = v_{t-1} - \nabla \hat{F}_x(u_{t-1}) = y_{t-1}$ where we use unbiasedness of i.i.d sampling. Hence, to bound $\| y_k \|$, we need to first bound $\| z_k \|$ according to Azuma-Hoeffding Lemma. We can use vector Bernstein inequality to bound $z_k$ as follows. First note that
\begin{align}
    z_k &= v_k -  \nabla \hat{F}_{x}(u_k) - v_{k-1} +  \nabla \hat{F}_{x}(u_{k-1}) \nonumber\\
    &= \nabla \hat{f}_{\mathcal{I}, x} (u_k) - \nabla \hat{f}_{\mathcal{I}, x} (u_{k-1}) + v_{k-1} - \nabla \hat{F}_{x}(u_k) - v_{k-1} +  \nabla \hat{F}_{x}(u_{k-1})\nonumber\\
    &= \nabla \hat{f}_{\mathcal{I}, x} (u_k) - \nabla \hat{f}_{\mathcal{I}, x} (u_{k-1}) - \nabla \hat{F}_{x}(u_k) + \nabla \hat{F}_{x}(u_{k-1}). \nonumber
\end{align}
Denote $w_i := \nabla \hat{f}_{i, x} (u_k) - \nabla \hat{f}_{i, x} (u_{k-1}) - \nabla \hat{F}_{x}(u_k) + \nabla \hat{F}_{x}(u_{k-1})$ and therefore $z_k = \frac{1}{b}\sum_{i \in \mathcal{I}} w_i$ with $|\mathcal{I}| =b$. In order to apply Bernstein inequality, we need to show that $\| w_i\|$ is bounded. This is achieved by Lemma \ref{L_lipschit_assump}. That is, 
\begin{align}
    \| w_i\| &= \| \nabla \hat{f}_{i, x} (u_k) - \nabla \hat{f}_{i, x} (u_{k-1}) - \nabla \hat{F}_{x}(u_k) + \nabla \hat{F}_{x}(u_{k-1})\| \nonumber\\
    &\leq \| \nabla \hat{f}_{i, x} (u_k) - \nabla \hat{f}_{i, x} (u_{k-1}) \| + \| \nabla \hat{F}_x(u_k) - \nabla \hat{F}_x(u_{k-1}) \| \leq 2L \| u_k - u_{k-1} \| =: R. \nonumber
\end{align}
Also, the total variance $\sigma^2$ is computed as 
\begin{align}
    \sum_{i \in \mathcal{I}} \mathbb{E} \|w_i - \mathbb{E}[w_i] \|^2 &= \sum_{i \in \mathcal{I}} \mathbb{E} \| w_i\|^2 \nonumber\\
    &= \sum_{i \in \mathcal{I}} \mathbb{E} \| \nabla \hat{f}_{i, x} (u_k) - \nabla \hat{f}_{i, x} (u_{k-1}) - \nabla \hat{F}_{x}(u_k) + \nabla \hat{F}_{x}(u_{k-1}) \|^2  \nonumber\\
    &\leq  \sum_{i \in \mathcal{I}}\mathbb{E} \|\nabla \hat{f}_{i, x} (u_k) - \nabla \hat{f}_{i, x} (u_{k-1})  \|^2 \leq b L^2 \| u_k - u_{k-1} \|^2 =: \sigma^2, \nonumber
\end{align}
where the first inequality holds due to $\mathbb{E}[w_i]=0$ and the second last inequality applies $\mathbb{E}\| x - \mathbb{E}[x] \|^2 \leq \mathbb{E}\| x\|^2$ and the last inequality uses the gradient Lipschitzness result in Lemma \ref{L_lipschit_assump}. Finally we can apply Bernstein inequality (Lemma \ref{bernsteinlemma}) to bound $\|z_k \|$. That is, 
\begin{align}
    \text{Pr} \{ \| \sum_{i \in \mathcal{I}} w_i \| \geq \varsigma \} &= \text{Pr} \{ \| \frac{1}{b} \sum_{i \in \mathcal{I}} w_i  \| \geq \varsigma/b \} = \text{Pr}\{ \| z_k \| \geq {\varsigma}/b \}   \nonumber\\
    &\leq (d+1)\exp \big( \frac{-\varsigma^2/2}{\sigma^2 + R\varsigma/3} \big) \nonumber\\
    &\leq (d+1) \exp \big( \frac{-\varsigma^2/2}{bL^2 \| u_k - u_{k-1}\|^2 + 2\sqrt{b}L \| u_k - u_{k-1}\| \varsigma/3}  \big) \nonumber\\
    &\leq \vartheta_k, \nonumber
\end{align}
where the second inequality substitutes $b L^2 \| u_k - u_{k-1} \|^2$ as $\sigma^2$ and $2L\| u_k - u_{k-1} \|$ as $R$ in Lemma \ref{bernsteinlemma}. It also uses the fact that $\sqrt{b} \geq 1$. The last inequality holds by the choice $\varsigma = \varsigma_k = \mathcal{O} \big( \log(d/\vartheta_k) \sqrt{b} L \big) \| u_k - u_{k-1} \|$ (for example, $\frac{10}{3} \log((d+1)/\vartheta_k)\sqrt{b} L \| u_k - u_{k-1} \|$).  This gives a probability bound for $\{z_k\}$, which is
\begin{equation}
    \|z_k \| \leq \mathcal{O}\Big(\frac{{\log(d/\vartheta_k)} L }{\sqrt{b}}\Big) \|u_k - u_{k-1} \| \quad \text{ with probability } 1 - \vartheta_k. \nonumber
\end{equation}
Now given the bound on $\|z_k\|$, we can bound $\|y_k \|$ using the Azuma-Hoeffding inequality. Suppose we set $\vartheta_k = \vartheta/m$, where $m$ is the epoch length. Then by union bound, for $k \leq m$
\begin{equation}
    \text{Pr}\{ \bigcup\limits_{j=1}^{k} \big(\|z_j \| \geq \varsigma\big) \} \leq \sum_{j=1}^k \vartheta_j \leq \sum_{j=1}^m \vartheta_j = \vartheta. \label{union_bound_z_k}
\end{equation}
Therefore, the probability that $\|y_k - y_{k-1} \| =\| z_k\| \leq \varsigma_k$ for all $k = 1,...,m$ is at least $1-\vartheta$. Hence by Lemma \ref{azuma_hoeffding_lemma}, we have 
\begin{equation}
    \text{Pr} \{ \|y_k - y_0 \| \geq \beta \} \leq (d+1)\exp \big( \frac{-\beta^2}{8 \sum_{j=1}^k \varsigma_j^2} \big) + k \vartheta_k \leq 2\vartheta, \label{hpBound_yk}
\end{equation}
where the last inequality holds due to $k \leq m$ and the choice that 
\begin{equation}
    \beta = \sqrt{8 \sum_{j=1}^k \varsigma^2_j \log((d+1)/\vartheta)} = \mathcal{O}\Big( \frac{\log^{\frac{3}{2}}(d/\vartheta) L}{\sqrt{b}}\Big) \sqrt{\sum_{j=1}^k \| u_j - u_{j-1} \|^2}, \nonumber
\end{equation}
where we denote $\log^{\frac{3}{2}}(a) = (\log(a))^{\frac{3}{2}}$. Note under finite-sum setting, $y_0 = v_0 - \nabla \hat{F}_x(u_0) = 0$. Thus \eqref{hpBound_yk} implies for $k \in [1, m]$,
\begin{equation}
    \| v_k - \nabla \hat{F}_x(u_k) \| = \| y_k\| \leq \mathcal{O}\Big( \frac{\log^{\frac{3}{2}}(d/\vartheta) L}{\sqrt{b}}\Big) \sqrt{\sum_{j=1}^k \| u_j - u_{j-1} \|^2}  \label{mmnnkhkh}
\end{equation}
holds with probability $1 - 2\vartheta$. Note that we can always set $\vartheta/2$ while the result still holds because $\mathcal{O}(\log^{\frac{3}{2}}(2d/\vartheta)) = \mathcal{O}(\log^{\frac{3}{2}}(d/\vartheta))$. Hence the probability reduces to $1 - \vartheta$. 
\end{proof}

\begin{lemma}[Improve or localize]
\label{improve_localize_lemma}
Consider $\{ u_t\}_{t=1}^{\mathscr{T}}$ is the sequence generated by running $\textsf{TSSRG}(x,$ $u_0, \mathscr{T})$. Suppose we choose $b \geq m$, $\eta \leq \frac{1}{2c_1 L}$ where $c_1 = \mathcal{O}(\log^{\frac{3}{2}}({dt}/{\vartheta}))$. Then we have
    \begin{equation*}
        \| u_t - u_0 \| \leq \sqrt{\frac{4t}{c_1L} (\hat{F}_x(u_0) - \hat{F}_x(u_t))} \quad \quad \text{ with probability } 1 - \vartheta,
    \end{equation*}
\end{lemma}
\begin{proof}
    First by generalizing \eqref{uopppp} to any epoch (i.e. $1 \leq t \leq \mathscr{T}$), we have 
    \begin{align}
        \hat{F}_x(u_t) - \hat{F}_x(u_{sm}) &\leq - \frac{\eta}{2} \sum_{j=sm+1}^{t} \| \nabla \hat F_x (u_{j-1}) \|^2 - \big( \frac{1}{2\eta} - \frac{L}{2} - \frac{\eta c_1^2 L^2 }{2} \big) \sum_{j = sm+1}^t \| u_j - u_{j-1} \|^2 \nonumber\\
        &\leq - \big( \frac{1}{2\eta} - \frac{L}{2} - \frac{\eta c_1^2 L^2 }{2} \big) \sum_{j = sm+1}^t \| u_j - u_{j-1} \|^2 \nonumber\\
        &\leq -\frac{c_1 L}{4} \sum_{j=sm+1}^t \| u_j - u_{j-1} \|^2, \label{yerent}
    \end{align}
    where the last inequality holds due to the choice $\eta \leq \frac{1}{2c_1 L}$ and the assumption that $c_1 \geq 1$. Summing \eqref{yerent} for all epoch up to $t$ gives 
    \begin{equation}
        \hat{F}_x(u_t) - \hat{F}_x(u_0) \leq -\frac{c_1 L}{4} \sum_{j = 1}^t \| u_j - u_{j-1}\|^2. \nonumber
    \end{equation}
    Also by Cauchy-Schwarz inequality and triangle inequality, 
    \begin{equation}
        \sqrt{t \sum_{j=1}^t \| u_j - u_{j-1} \|^2} \geq \sum_{j=1}^{t} \| u_j - u_{j-1} \| \geq \| u_t - u_0 \|. \nonumber
    \end{equation}
    The proof is complete by noting $\sqrt{t \sum_{j=1}^t \| u_j - u_{j-1} \|^2} \leq \sqrt{\frac{4t}{c_1L} (\hat{F}_x(u_0) - \hat{F}_x(u_t))}$. 
\end{proof}

\begin{lemma}[Large gradient descent lemma] \emph{(Lemma \ref{large_descent_lemma_informal} in the maix text).}
\label{large_descent_lemma}
    Under Assumptions 2, 3, and 4, suppose we choose $\eta \leq \frac{1}{2c_1L}$, $b \geq m$ where $c_1 = \mathcal{O}(\log^{\frac{3}{2}}(dm/\vartheta) )$. Consider $x \in \mathcal{M}$ where $\| \emph{grad}F(x) \| \geq \epsilon$ with $\epsilon \leq \frac{D}{m} \sqrt{\frac{8c_1 L}{\eta}}$. Then by running $\textsf{TSSRG}(x, 0, m)$, we have the following three cases:  
    \begin{enumerate}
        \item When the iterates $\{ u_j \}_{j=1}^m$ do not leave the constraint ball $\mathbb{B}_x(D)$:
        \begin{enumerate}
            \item If at least half of the iterates in the epoch satisfy $\| \nabla \hat{F}_x (u_j) \| \leq \epsilon/2$ for $j = 1,...,m$, then with probability at least $1/2$, we have $ \| \emph{grad}F( \textsf{TSSRG}(x, 0, m) )\| \leq \epsilon$. 
            \item Otherwise, with probability at least $1/5$, we have $ F(\textsf{TSSRG}(x, 0, m)) - F(x) \leq - \frac{\eta m\epsilon^2}{32}$. 
        \end{enumerate}
        \item When one of the iterates $\{ u_j \}_{j=1}^m$ leaves the constraint ball $\mathbb{B}_x{(D)}$, with probability at least $1 - \vartheta$, we have $ F(\textsf{TSSRG}(x, 0, m)) - F(x) \leq - \frac{\eta m\epsilon^2}{32}$. 
    \end{enumerate}
    Regardless which case occurs, $F(\textsf{TSSRG}(x, 0, m)) \leq F(x)$ with high probability.
\end{lemma}
\begin{proof}
First note that when gradient is large, we always call \textsf{TSSRG}($x, 0, m$) with total iterations set to be $m$ (i.e. a single epoch). Hence, we consider $t = 1, ..., m$ (in \textsf{TSSRG}). Compared to the proof in \citep{LiSSRGD2019}, we further need to address the case where the iterates fall outside the prescribed ball $\mathbb{B}_x(D)$. So we divide the proof into two parts. 

\textbf{1. Iterates do not leave the constraint ball}. By $L$-Lipschitzness in Lemma \ref{L_lipschit_assump}, we have
\begin{align}
    \hat{F}_x(u_t) &\leq \hat{F}_x(u_{t-1}) + \langle  \nabla \hat{F}_{x}(u_{t-1}), u_t - u_{t-1} \rangle + \frac{L}{2} \|u_t - u_{t-1} \|^2 \nonumber\\
    &= \hat{F}_x(u_{t-1}) - \eta \langle \nabla \hat{F}_x(u_{t-1}) , v_{t-1} \rangle + \frac{\eta^2 L}{2} \| v_{t-1} \|^2  \nonumber\\
    &= \hat{F}_x(u_{t-1}) -\frac{\eta}{2} \| \nabla \hat{F}_x(u_{t-1})\|^2 - \frac{\eta}{2} \| v_{t-1} \|^2 + \frac{\eta}{2} \| v_{t-1} - \nabla \hat{F}_x(u_{t-1}) \|^2 + \frac{\eta^2 L}{2} \| v_{t-1} \|^2 \nonumber\\
    &= \hat{F}_x(u_{t-1}) - \frac{\eta}{2} \| \nabla \hat{F}_x(u_{t-1})\|^2 + \frac{\eta}{2} \| v_{t-1} - \nabla \hat{F}_x(u_{t-1}) \|^2 - \big( \frac{1}{2\eta} - \frac{L}{2} \big) \| u_t - u_{t-1} \|^2. \label{tuioioo}
\end{align}
From Lemma \ref{hpBound_estimation_error}, we know that 
\begin{equation}
    \| v_{t-1} - \nabla \hat{F}_x(u_{t-1}) \|^2 \leq \frac{\mathcal{O}(\log^3(d/\vartheta)) L^2}{b} \sum_{j = sm+1}^{t-1} \| u_j - u_{j-1} \|^2 \label{qmqmmqmq}
\end{equation}
holds with high probability $1 - \vartheta$. By a union bound, $\text{Pr} \big\{ \bigcup\limits_{\tau=1}^{{t}} \big( \text{\eqref{qmqmmqmq} does not hold} \big) \big\} \leq t \vartheta$ and therefore for all $1 \leq \tau \leq t$, \eqref{qmqmmqmq} holds with probability $1 - t\vartheta$. Setting $\vartheta/t$ as $\vartheta$ we have for all $1 \leq \tau \leq t$,
\begin{equation}
    \| v_{\tau-1} - \nabla \hat{F}_x(u_{\tau-1}) \|^2 \leq \frac{\mathcal{O}(\log^3(d t/\vartheta)) L^2}{b} \sum_{j = 1}^{\tau-1} \| u_j - u_{j-1} \|^2 \quad \text{ with probability } 1 - \vartheta. \nonumber
\end{equation}
Substituting this result into \eqref{tuioioo} and summing over this epoch from $1$ to $t$ gives
\begin{align}
   &\hat{F}_x(u_t) - \hat{F}_x(u_{0}) \nonumber\\
   &\leq  - \frac{\eta}{2} \sum_{j=1}^{t} \| \nabla \hat F_x (u_{j-1}) \|^2 + \frac{\eta c_1^2 L^2}{2b} \sum_{k = 1}^{t-1} \sum_{j = 1}^{k} \| u_{j} - u_{j-1} \|^2 - \big( \frac{1}{2\eta} - \frac{L}{2} \big) \sum_{j = 1}^t \| u_j - u_{j-1} \|^2  \nonumber\\
   &\leq  - \frac{\eta}{2} \sum_{j=1}^{t} \| \nabla \hat F_x (u_{j-1}) \|^2 + \frac{\eta c_1^2 L^2 m}{2b} \sum_{j = 1}^t \| u_j - u_{j-1} \|^2 - \big( \frac{1}{2\eta} - \frac{L}{2} \big) \sum_{j = 1}^t \| u_j - u_{j-1} \|^2 \nonumber\\
   &\leq - \frac{\eta}{2} \sum_{j=1}^{t} \| \nabla \hat F_x (u_{j-1}) \|^2 - \big( \frac{1}{2\eta} - \frac{L}{2} - \frac{\eta c_1^2 L^2 }{2} \big) \sum_{j = 1}^t \| u_j - u_{j-1} \|^2 \label{uopppp}\\
   &\leq -\frac{\eta}{2} \sum_{j=1}^{t} \| \nabla \hat F_x (u_{j-1}) \|^2, \label{descent_equation}
\end{align}
where $c_1 := \mathcal{O}(\log^{\frac{3}{2}}(dt/\vartheta))$. The second inequality uses the fact that $t \leq m$ and the third inequality holds due to the choice that $b \geq m$. The last inequality holds by noticing $\eta \leq \frac{1}{2c_1L} \leq \frac{\sqrt{4c_1^2+1}-1}{2c_1^2 L}$ by assuming $c_1 \geq 1$. Note that we require \eqref{descent_equation} to hold for all $t \leq m$ and thus we change $c_1 = \mathcal{O}(\log^{\frac{3}{2}}(dm/\vartheta))$. Then we have the following two cases.
\begin{itemize}
    \item \textbf{(Case 1a)} Suppose at least half of the iterates in the epoch satisfy $\| \nabla \hat{F}_x (u_j) \| \leq \epsilon/2$ for $j = 1,...,m$. Then by uniformly sampling $\{u_j\}_{j=1}^m$ (i.e. uniformly breaking by setting $u_t$ as $u_m$ as in Algorithm \ref{TSSRG_algorithm}, Line \ref{breaking_rules_tssrg}), the output $u_m$ has gradient norm $\|\nabla \hat{F}_x (u_m) \|$ no larger than $\epsilon/2$ with probability $1/2$. Recall the definition of the pullback gradient in Lemma \ref{define_grad_hess_pullback}, i.e. $\nabla \hat{F}_x(u_m) = T_{u_m}^* \text{grad} F(R_x(u_m))$. Then the output of \textsf{TSSRG}($x,0,m$) satisfies
    \begin{align}
        \| \text{grad}F( \textsf{TSSRG}(x, 0, m) )\| = \| \text{grad} F(R_x(u_m))\| &= \| (T^*_{u_m})^{-1} \nabla \hat{F}_x(u_m)\| \nonumber\\
        &\leq \| (T^*_{u_m})^{-1}\| \|\nabla \hat{F}_x(u_m) \| \nonumber\\
        &\leq 2 \cdot \frac{\epsilon}{2} = \epsilon, \nonumber
    \end{align}
    where we use $\sigma_{\min} (T_u) \geq 1/2$ in Lemma \ref{singular_value_bound}.
    \item \textbf{(Case 1b)} Suppose at least half of the points in the epoch satisfy $\| \nabla \hat{F}_x (u_j) \| \geq \epsilon/2$ for $j = 1,...,m$. With probability $1/4$, the output $u_t$ falls within the last quarter of $\{ u_j \}_{j=1}^m$ by uniform sampling. In this case, $\sum_{j = 1}^t \| \nabla \hat{F}_x(u_{j-1}) \|^2 \geq \frac{m}{4} \cdot \frac{\epsilon^2}{4} = \frac{m \epsilon^2}{16}$ because at least a quarter of points with large gradient appear before $u_t$. Thus, by \eqref{descent_equation}, we have 
    \begin{align}
        F(\textsf{TSSRG}(x, 0, m)) - F(x) = F(R_x(u_t)) - F(x) &= \hat{F}_x(u_t) - \hat{F}_x(0) \nonumber\\
        &\leq - \frac{\eta}{2} \sum_{j=1}^t \| \nabla \hat{F}_x(u_{j-1}) \|^2 \nonumber\\
        &\leq - \frac{\eta m \epsilon^2}{32}. \label{euitp}
    \end{align}
    Note that \eqref{descent_equation} holds with probability $ 1- \vartheta$. Then by a union bound, \eqref{euitp} holds with probability at least $\frac{1}{4}- \vartheta$. Without loss of generality, we can choose $\vartheta \leq \frac{1}{20}$ and thus \eqref{euitp} holds with probability at least $\frac{1}{5}$. 
\end{itemize}

\textbf{2. Iterates leave the constraint ball.} Suppose at $\tau \leq m$, we have $\| u_\tau\| > D$. Therefore by Lemma \ref{improve_localize_lemma}, we know that the function value already decreases a lot. That is, with probability $1 - \vartheta$, running \textsf{TSSRG}$(x,0,m)$ gives
\begin{align}
    \hat{F}_x(0) - \hat{F}_x(u_\tau) \geq \frac{c_1 L}{4\tau} \| u_\tau\|^2 \geq \frac{c_1 L D^2}{4 m} \geq \frac{\eta m \epsilon^2}{32}, \nonumber
\end{align}
where the last inequality follows from the choice that $\epsilon \leq \frac{D}{m} \sqrt{\frac{8c_1 L}{\eta}}$ and $\eta$ can be sufficiently small. Note this requirement is not difficult to satisfy since $c_1$ can be sufficiently large. Hence by returning $u_\tau$ as $u_{m}$, we have 
\begin{equation}
    F(\textsf{TSSRG}(x,0,m)) - F(x) \leq - \frac{\eta m \epsilon^2}{32}. \nonumber
\end{equation}
In summary, regardless of whether the iterates stay within the ball $\mathbb{B}_x(D)$, with high probability, either the gradient norm of the output is small, or the function value decreases a lot. Note for Case 1a, the function still decreases by \eqref{descent_equation}.
\end{proof}

\begin{lemma}[Small stuck region]
\label{small_stuck_region_lemma}
    Consider $x \in \mathcal{M}$ with $\|\emph{grad}F(x)\| \leq \epsilon$ and $-\gamma := \lambda_{\min}(\nabla^2 \hat{F}_{x}(0))$ $= \lambda_{\min}(\emph{Hess} F(x)) \leq - \delta$ and $L \geq \delta$. Let $u_0, u_0' \in T_x\mathcal{M}$ be two random perturbation, satisfying $\| u_0 \|, \| u_0'\| \leq r$ and $u_0 - u_0' =  r_0 e_1$, where $e_1$ denotes the smallest eigenvector of $\nabla^2 \hat{f}_x(0)$ and $r_0 = \frac{\nu r}{\sqrt{d}}, r \leq \frac{\delta}{c_2\rho}$. Also set parameters $\mathscr{T} = \frac{2\log_{\alpha}(\frac{8 \delta \sqrt{d}}{c_2 \rho \nu r})}{\delta} = \widetilde{\mathcal{O}}(\frac{1}{\delta})$, $\eta \leq \min \{ \frac{1}{8C_1\log(\mathscr{T}) L}, \frac{1}{8\log_{\alpha}(\frac{8 \delta \sqrt{d}}{c_2 \rho \nu r}) L} \} = \widetilde{\mathcal{O}}(\frac{1}{L})$, where $\alpha \geq 1$ is chosen sufficiently small such that $\log_\alpha(1 + \eta \gamma) > \gamma$. Then for $\{ u_t\}, \{u_t' \}$ generated by running \textsf{TSSRG}$(x, u_0, \mathscr{T})$ and \textsf{TSSRG}$(x, u_0', \mathscr{T})$ with same sets of mini-batches, with probability $1 - \zeta$, we have
    \begin{equation*}
        \exists \, t \leq \mathscr{T}, \, \max\{ \| u_t - u_0 \|, \| u_t' - u_0' \|\} \geq \frac{\delta}{c_2 \rho},
    \end{equation*}
    where $c_2 \geq \max\{\frac{12C_1}{L}, \frac{2\delta}{D\rho}\}$ and $C_1 = \mathcal{O}(\log^{\frac{3}{2}} (d\mathscr{T}/\zeta) ) = \widetilde{\mathcal{O}}(1)$.
\end{lemma}

\begin{proof}
The proof is by contradiction. So we assume 
\begin{equation}
    \forall \, t \leq \mathscr{T}, \, \max\{ \| u_t - u_0\|, \| u_t' - u_0 \|\} \leq \frac{\delta}{c_2 \rho}. \label{qwewq}
\end{equation}

First we should note none of the two sequences $\{ u_j \}_{j=0}^\mathscr{T}$, $\{ u_j' \}_{j=0}^\mathscr{T}$ escape the ball $\mathbb{B}_x(D)$ within $\mathscr{T}$ steps under condition \eqref{qwewq}. This is because (take $\{ u_j\}$ for example), 
\begin{equation*}
    \| u_t \| \leq \| u_t - u_0 \| + \| u_0\| \leq \frac{\delta}{c_2 \rho} + r \leq \frac{2\delta}{c_2 \rho} \leq D,
\end{equation*}
where we apply {$r \leq \frac{\delta}{c_2 \rho}$ and $c_2 \geq \frac{2\delta}{D\rho}$}. Hence, if condition \eqref{qwewq} is satisfied, $\{ u_j\}, \{ u_j'\}$ must remain inside the ball $\mathbb{B}_x(D)$. As a result, we can proceed the proof by following the similar idea as in \citep{LiSSRGD2019}, which is to show an exponential growth in the distance between two coupled sequences and ultimately will exceed the bound in \eqref{qwewq} using triangle inequality. This as a result gives rise to a contradiction. 

Denote $\hat{u}_t := u_t - u_t'$ and $\mathcal{H} := \nabla^2 \hat{F}_x(0)$. With $u_t = u_{t-1} - \eta v_{t-1}$, we can express $\hat{u}_t$ as 
\begin{align}
    \hat{u}_t &= \hat{u}_{t-1} - \eta (v_{t-1} - v_{t-1}') \nonumber\\
    &= \hat{u}_{t-1} - \eta ( \nabla \hat{F}_x(u_{t-1}) - \nabla \hat{F}_x(u_{t-1}') + v_{t-1} - \nabla \hat{F}_x(u_{t-1}) - v_{t-1}' + \nabla \hat{F}_x (u_{t-1}') ) \nonumber\\
    &= \hat{u}_{t-1} - \eta \big( \int_0^1 \nabla^2 \hat{F}_x(u_{t-1}' + \theta (u_{t-1} - u_{t-1}')) [u_{t-1} - u_{t-1}'] d\theta + v_{t-1} - \nabla \hat{F}_x(u_{t-1}) \nonumber\\ 
    &\quad - v_{t-1}' + \nabla \hat{F}_x (u_{t-1}') \big) \nonumber\\
    &= \hat{u}_{t-1} - \eta \big( (\mathcal{H} + \Delta_{t-1}) \hat{u}_{t-1} + \hat{y}_{t-1} \big) \label{temp_u_t_u_t1}\\
    &= (I - \eta \mathcal{H}) \hat{u}_{t-1} - \eta (\Delta_{t-1} \hat{u}_{t-1}  + \hat{y}_{t-1}). \label{urtmroqroer}
\end{align}
where $\Delta_t := \int_0^1 [\nabla^2 \hat{F}_x (u_t' + \theta (u_t - u_t')) - \mathcal{H}]  d\theta$ and $\hat{y}_t := y_t - y_{t}' := v_{t} - \nabla \hat{F}_x(u_{t}) - v_{t}' + \nabla \hat{F}_x (u_{t}')$. Note we can bound $\| \Delta_t\|$ as 
\begin{align}
    \| \Delta_t \| = \| \int_0^1 \big( \nabla^2 \hat{F}_x (u_t' + \theta (u_t - u_t'))  - \nabla^2 \hat{F}_x(0) \big) d\theta \| &\leq \int_0^1 \rho \| u_t' + \theta (u_t - u_t') \| d\theta \nonumber\\
    &\leq \int_0^1 \rho \big( \theta \|u_t \| + (1- \theta) \| u_t'\| \big) d\theta \nonumber\\
    &\leq \rho \max \{ \|u_t \|, \|u_t' \| \} \nonumber\\
    &= \rho \mathscr{D}_t \leq \rho \big( \frac{\delta}{c_2 \rho} + r\big), \label{delta_t_bound}
\end{align}
where $\mathscr{D}_t := \max\{\|u_t \|, \|u_t' \|\}$. The first inequality uses Assumption \ref{grad_hess_lip_asssump} and the last inequality follows from $\|u_t \| \leq \| u_t - u_0 \| + \| u_0\| \leq \frac{\delta}{c_2 \rho} + r$. Denote for $t \geq 1$,
\begin{align}
    &p(t) := (I - \eta \mathcal{H})^{t} \hat{u}_0 = (I - \eta \mathcal{H})^t r_0e_1 = (1 + \eta \gamma)^t r_0 e_1, \nonumber\\
    &q(t) := \eta \sum_{j = 0}^{t-1} (I -\eta \mathcal{H})^{t -1 - j} (\Delta_j \hat{u}_j + \hat{y}_j), \quad \text{ and } \label{q_t_definition}\\
    &p(0) = \hat{u}_0, \, q(0) = 0.
\end{align}
By induction, we can show that $\hat{u}_t = p(t) - q(t)$. It is noticed that $\| p(t)\| = (1+\eta \gamma)^t r_0$ grows exponentially and the only thing left to determine is that $q(t)$ is dominated by $p(t)$ when $t$ increases. To this end, we inductively show that (1) $\| \hat{y}_t\| \leq  \gamma L (1 + \eta \gamma)^t r_0$ and (2) $\| q(t) \| \leq \| p(t) \|/2$. First note that when $t=0$, these two conditions clearly hold. Now suppose for $j \leq t-1$, these claims have been proved to be true. Hence we immediately have, by triangle inequality,
\begin{equation}
    \| \hat{u}_j \| \leq \| p(j) \| + \| q(j) \| \leq \frac{3}{2} \| p(j) \| \leq \frac{3}{2} (1 + \eta \gamma)^j r_0 \label{u_j_inductive_bound}
\end{equation}
holds for $j \leq t-1$. We first prove that second condition holds for $j = t$.

\textbf{Proof that $\boldsymbol{\| q(t)\|}$ is bounded by $\boldsymbol{\|p(t) \|/2}$.} Note that from the definition in \eqref{q_t_definition},
\begin{equation*}
    \| q(t) \| \leq \| \eta \sum_{j = 0}^{t-1} (I -\eta \mathcal{H})^{t -1 - j} \Delta_j \hat{u}_j \| + \| \eta \sum_{j = 0}^{t-1} (I -\eta \mathcal{H})^{t -1 - j} \hat{y}_j \|. 
\end{equation*}
Then we can bound $\| q(t) \|$ by respectively bounding the two terms. The first term is bounded as follows,
\begin{align}
    \| \eta \sum_{j = 0}^{t-1} (I -\eta \mathcal{H})^{t -1 - j} \Delta_j \hat{u}_j \| &\leq \eta \sum_{j=0}^{t-1} \| (I - \eta \mathcal{H})^{t-1-j} \| \| \Delta_j \| \| \hat{u}_j\| \nonumber\\
    &\leq \eta \rho (\frac{\delta}{c_2 \rho} + r) \sum_{j=0}^{t-1} (1+\eta \gamma)^{t-1-j} \|\hat{u}_j \| \nonumber\\
    &\leq \frac{3}{2} \eta \rho (\frac{\delta}{c_2 \rho} + r) t (1 + \eta \gamma)^{t-1}r_0 \leq \frac{3\eta \delta}{c_2} \mathscr{T} \| p(t)\| \leq \frac{1}{4} \| p(t)\|, \label{q_t_bound_former}
\end{align}
where the second inequality holds because $\| I - \eta \mathcal{H} \| = \lambda_{\max} (I - \eta \mathcal{H}) = 1 + \eta \gamma$. The third inequality applies the first condition for $j \leq t-1$. The fourth inequality is by $r \leq \frac{\delta}{c_2 \rho}$ and $t \leq \mathscr{T}$. The last inequality follows by the choice of parameters $\mathscr{T} \leq 2 \log_{\alpha}(\frac{8\delta \sqrt{d}}{c_2 \rho \nu r})/ (\eta \gamma), \delta \leq \gamma$ and $c_2 \geq 24 \log_{\alpha}(\frac{8\delta \sqrt{d}}{c_2 \rho \nu r})$ where $\alpha \geq 1$ is a constant defined later. The second term can also be similarly bounded as
\begin{align}
    \| \eta \sum_{j = 0}^{t-1} (I -\eta \mathcal{H})^{t -1 - j} \hat{y}_j \| &\leq \eta \sum_{j = 0}^{t-1} (1+\eta \gamma)^{t-1-j} \| \hat{y}_j\| \leq \eta \sum_{j = 0}^{t-1} (1+\eta \gamma)^{t-1-j} \gamma L (1+ \eta \gamma)^j r_0 \nonumber\\
    &\leq \eta \gamma L \mathscr{T} (1 + \eta \gamma)^{t-1} r_0 \leq \frac{1}{4} (1 + \eta \gamma)^t r_0 = \frac{1}{4} \|p(t) \|, \label{q_t_bound_latter}
\end{align}
where we apply the bound on $\| \hat{y}_t \|$ in the second inequality. The last inequality uses $\mathscr{T} \leq 2\log_{\alpha}(\frac{8 \delta \sqrt{d}}{c_2 \rho \nu r})/\gamma$ and $\eta \leq \frac{1}{8 \log_{\alpha}(\frac{8 \delta \sqrt{d}}{c_2 \rho \nu r}) L}$. From \eqref{q_t_bound_former} and \eqref{q_t_bound_latter}, we prove $\| q(t)\| \leq \| p(t)\|/2$. Note we can always assume $\eta \leq 1$ given its requirement. Therefore $1/\gamma \leq 1/(\eta \gamma)$ and it is sufficient to require $\mathscr{T} = 2\log_\alpha (\frac{8 \delta \sqrt{d}}{c_2 \rho \nu r})/\gamma$ to guarantee the result. Now we can proceed to prove the first condition, which is an intermediate result that has been used in proving the second condition.

\textbf{Proof that $\boldsymbol{\| \hat{y}_t\| \text{ is bounded by }  \gamma L (1 + \eta \gamma)^t r_0}$.}
We first re-write $\hat{y}_t$ into a recursive form as
    \begin{align}
        \hat{y}_t &= v_t - \nabla \hat{F}_x(u_t) - v_t' + \nabla \hat{F}_x(u_t') \nonumber\\
        &= \nabla \hat{f}_{\mathcal{I}, x} (u_t) - \nabla \hat{f}_{\mathcal{I}, x} (u_{t-1}) + v_{t-1} - \nabla \hat{F}_x(u_t) - \nabla \hat{f}_{\mathcal{I}, x} (u_t') + \nabla \hat{f}_{\mathcal{I}, x} (u_{t-1}') - v_{t-1}' + \nabla \hat{F}_x(u_t') \nonumber\\
        &= \nabla \hat{f}_{\mathcal{I}, x} (u_t) - \nabla \hat{f}_{\mathcal{I}, x} (u_{t-1}) - \nabla \hat{F}_x(u_t) + \nabla \hat{F}_x(u_{t-1}) - \big(\nabla \hat{f}_{\mathcal{I}, x} (u_t') - \nabla \hat{f}_{\mathcal{I}, x} (u_{t-1}') \nonumber\\
        &\quad - \nabla \hat{F}_x(u_t') + \nabla \hat{F}_x(u_{t-1}')  \big) + v_{t-1} - \nabla \hat{F}_x(u_{t-1}) - v_{t-1}' + \hat{F}_x(u_{t-1}') \nonumber\\
        &= z_t + \hat{y}_{t-1}, \nonumber
    \end{align}
    where we denote $z_t := \nabla \hat{f}_{\mathcal{I}, x} (u_t) - \nabla \hat{f}_{\mathcal{I}, x} (u_{t-1}) - \nabla \hat{F}_x(u_t) + \nabla \hat{F}_x(u_{t-1}) - \big(\nabla \hat{f}_{\mathcal{I}, x} (u_t') - \nabla \hat{f}_{\mathcal{I}, x} (u_{t-1}') - \nabla \hat{F}_x(u_t') + \nabla \hat{F}_x(u_{t-1}')  \big) = \hat{y}_t - \hat{y}_{t-1}$. It is easy to verify that $\{ \hat{y}_t\}$ is a martingale sequence and $\{ z_t\}$ is its difference sequence. Similar to the proof strategy as for Lemma \ref{hpBound_estimation_error}, we first need to derive bound for $\| z_t\|$ by Bernstein inequality. Denote $w_i$ as the component of $z_t$. That is 
    \begin{align}
        z_t = \frac{1}{b}\sum_{i \in \mathcal{I}} w_i = \frac{1}{b}\sum_{i \in \mathcal{I}} \Big(  &\nabla \hat{f}_{i, x} (u_t) - \nabla \hat{f}_{i, x} (u_{t-1}) - \nabla \hat{F}_x(u_t) + \nabla \hat{F}_x(u_{t-1}) \nonumber\\
        &- \big(\nabla \hat{f}_{i, x} (u_t') - \nabla \hat{f}_{i, x} (u_{t-1}') - \nabla \hat{F}_x(u_t') + \nabla \hat{F}_x(u_{t-1}')  \big) \Big). \nonumber
    \end{align}
    Then $\| w_i\|$ is bounded as 
    \begin{align}
        &\| w_i \| \nonumber\\
        &= \| \nabla \hat{f}_{i, x} (u_t) - \nabla  \hat{f}_{i, x} (u_{t-1}) - \nabla \hat{F}_x(u_t) + \nabla \hat{F}_x(u_{t-1}) - \nabla \hat{f}_{i, x} (u_t') + \nabla \hat{f}_{i, x} (u_{t-1}') \nonumber\\
        &+ \nabla \hat{F}_x(u_t') - \nabla \hat{F}_x(u_{t-1}')   \| \nonumber\\
        &= \| (\nabla \hat{f}_{i, x} (u_t) - \nabla \hat{f}_{i, x} (u_t')) - (\nabla  \hat{f}_{i, x} (u_{t-1}) - \nabla  \hat{f}_{i, x} (u_{t-1}')) - ( \nabla \hat{F}_x(u_t) -\nabla \hat{F}_x(u_t')  ) \nonumber\\
        &+ (  \nabla \hat{F}_x(u_{t-1}) -\nabla \hat{F}_x(u_{t-1}') )  \| \nonumber\\
        &= \| \int_0^1 \nabla^2 \hat{f}_{i,x} (u_t' + \theta(u_t - u_t') ) [u_t - u_t'] d\theta -\int_0^1 \nabla^2 \hat{f}_{i,x} (u_{t-1}' + \theta(u_{t-1} - u_{t-1}'))[u_{t-1} - u_{t-1}'] d\theta \nonumber\\
        &- \int_0^1 \nabla^2 \hat{F}_x (u_t' + \theta (u_t - u_t'))[u_t - u_t'] d\theta + \int_0^1 \nabla^2 \hat{F}_x (u_{t-1}' + \theta (u_{t-1} - u_{t-1}'))[u_{t-1} - u_{t-1}'] d\theta \| \nonumber\\
        &= \| (\Delta_t^i + \mathcal{H}_i) \hat{u}_t - (\Delta_{t-1}^i + \mathcal{H}_i) \hat{u}_{t-1} - (\Delta_t + \mathcal{H}) \hat{u}_{t} + (\Delta_{t-1} + \mathcal{H} ) \hat{u}_{t-1} \| \nonumber\\
        &\leq \| \mathcal{H}_i  (\hat{u}_t - \hat{u}_{t-1}) \| + \| \mathcal{H} (\hat{u}_t - \hat{u}_{t-1})  \| + \| \Delta_t^i \hat{u}_t \| + \| \Delta_t \hat{u}_t \| + \| \Delta_{t-1}^i \hat{u}_{t-1} \|  + \| \Delta_{t-1} \hat{u}_{t-1} \| \nonumber\\
        &\leq 2L \| \hat{u}_t - \hat{u}_{t-1}\| + 2 \rho \mathscr{D}_t \| \hat{u}_t \| + 2 \rho \mathscr{D}_{t-1} \| \hat{u}_{t-1}\| =: R, \label{bound_on_wi}
    \end{align}
    where $\hat{u}_t := u_t - u_t'$, $\mathcal{H}_i := \nabla^2 \hat{f}_{i,x}(0)$, $\Delta_{t}^i := \int_0^1 (\nabla^2 \hat{f}_i (u_t' + \theta(u_t - u_{t}')) - \mathcal{H}_i ) d\theta$, $\mathscr{D}_t := \max \{ \| u_t\|, \| u_t'\| \}$. The second last inequality uses triangle inequality and the last inequality considers Assumption \ref{L_lipschit_assump} and also the constraint on $\| \Delta_t\|$ (similarly $\| \Delta_t^i\|$) in \eqref{delta_t_bound}. The total variance $\sigma^2$ is derived as 
    \begin{align}
        &\sum_{i\in \mathcal{I}} \mathbb{E} \| w_i \|^2 \nonumber\\
        &= \sum_{i\in \mathcal{I}} \mathbb{E}\|\nabla \hat{f}_{i, x} (u_t) - \nabla  \hat{f}_{i, x} (u_{t-1}) - \nabla \hat{F}_x(u_t) + \nabla \hat{F}_x(u_{t-1}) - \nabla \hat{f}_{i, x} (u_t') + \nabla \hat{f}_{i, x} (u_{t-1}') \nonumber\\
        &+ \nabla \hat{F}_x(u_t') - \nabla \hat{F}_x(u_{t-1}') \|^2 \nonumber\\
        &\leq \sum_{i \in \mathcal{I}} \mathbb{E}\| \nabla \hat{f}_{i, x} (u_t) - \nabla \hat{f}_{i, x} (u_t') - \nabla  \hat{f}_{i, x} (u_{t-1}) +  \nabla \hat{f}_{i, x} (u_{t-1}') \|^2 \nonumber\\
        &= \sum_{i \in \mathcal{I}} \mathbb{E}\| (\Delta_t^i + \mathcal{H}_i) \hat{u}_t - (\Delta_{t-1}^i + \mathcal{H}_i) \hat{u}_{t-1} \|^2 \nonumber\\
        &\leq b (L\| \hat{u}_t - \hat{u}_{t-1}\| + \rho \mathscr{D}_t \| \hat{u}_t \| + \rho \mathscr{D}_{t-1} \| \hat{u}_{t-1}\|)^2 =: \sigma^2 \nonumber
    \end{align}
    where the first inequality uses $\mathbb{E} \|x - \mathbb{E}[x]\|^2 \leq \mathbb{E}\| x\|^2$ and the last inequality follows similar logic as \eqref{bound_on_wi}. Hence by vector Bernstein Lemma \ref{bernsteinlemma}, we obtain 
    \begin{align}
        \text{Pr}\{ \| z_t \| \geq \varsigma/b \} = \text{Pr} \{ \| \sum_{i \in \mathcal{I}} w_i \| \geq \varsigma \} &\leq (d+1) \exp \big( \frac{ -\varsigma^2/2}{\sigma^2 + R\varsigma/3}  \big) \nonumber\\
        &\leq \zeta_t, \nonumber
    \end{align}
    where we choose $\varsigma = \mathcal{O}\big( \log(d/\zeta_t) \sqrt{b} \big) (L \| \hat{u}_t - \hat{u}_{t-1}\| + \rho \mathscr{D}_t \| \hat{u}_t \| + \rho \mathscr{D}_{t-1} \| \hat{u}_{t-1}\|)$ based on similar logic as in Lemma \ref{hpBound_estimation_error}. Therefore, we can bound $\| z_t\|$ with high probability. That is, with probability $1 - \zeta_t$, 
    \begin{align}
        \| z_t\| \leq \mathcal{O} \Big( {\frac{\log(d/\zeta_t)}{\sqrt{b}}} \Big) (L \| \hat{u}_t - \hat{u}_{t-1}\| + \rho \mathscr{D}_t \| \hat{u}_t \| + \rho \mathscr{D}_{t-1} \| \hat{u}_{t-1}\|) =: \varsigma_t. \nonumber
    \end{align}
    We now can obtain a bound on $\|\hat{y}_t \|$. We only need to consider a single epoch because $\hat{y}_{sm} = 0$ due to the full gradient evaluation at the start of each epoch. Note that similar to \eqref{union_bound_z_k} and union bound, we know that with probability at least $1 - \zeta$, where $\zeta_t := \zeta/m$, $\|z_t\| \leq \varsigma_t$ holds for all $t = sm+1, ..., (s+1)m$. Then applying Azuma-Hoeffding inequality (Lemma \ref{azuma_hoeffding_lemma}) to the martingale sequence $\{\hat{y}_t\}$ yields
    \begin{align}
        \text{Pr}\{ \|\hat{y}_t\| \geq \beta \} \leq (d+1) \exp\Big( \frac{-\beta^2}{8 \sum_{j = sm+1}^t \varsigma_j^2}  \Big)+ \zeta \leq 2 \zeta, \quad \forall \, t \in [sm+1, sm+m], \nonumber
    \end{align}
    where we choose 
    \begin{align}
        \beta = \beta_t &= \sqrt{8 \log((d+1)/\zeta) \sum_{j=sm+1}^t \varsigma_j^2} \nonumber\\
        &= \mathcal{O} \Big( \frac{\log^{\frac{3}{2}}(d/\zeta)}{\sqrt{b}} \Big) \sqrt{\sum_{j=sm+1}^t (L \| \hat{u}_t - \hat{u}_{t-1}\| + \rho \mathscr{D}_t \| \hat{u}_t \| + \rho \mathscr{D}_{t-1} \| \hat{u}_{t-1}\|)^2}. \nonumber
    \end{align}
    By a union bound, $\text{Pr} \{ \bigcup\limits_{t=1}^{\mathscr{T}} \big(\| \hat{y}_t \| \geq \beta_t \big) \} \leq 2\mathscr{T} \zeta$ and therefore for all $t \leq \mathscr{T}$, 
    \begin{equation*}
        \| \hat{y}_t \| \leq \beta_t \quad \text{ with probability } 1 - 2 \mathscr{T} \zeta. \nonumber
    \end{equation*}
    Note by simply setting $\zeta/(2\mathscr{T})$ as $\zeta$, we obtain with probability $1 - \zeta$, for all $t \leq \mathscr{T}$,
    \begin{equation}
        \| \hat{y}_t \| \leq \mathcal{O}\Big( \frac{\log^{\frac{3}{2}}(d \mathscr{T}/\zeta)}{\sqrt{b}} \Big) \sqrt{\sum_{j=sm+1}^t (L \| \hat{u}_t - \hat{u}_{t-1}\| + \rho \mathscr{D}_t \| \hat{u}_t \| + \rho \mathscr{D}_{t-1} \| \hat{u}_{t-1}\|)^2}. \label{y_hat_t_bound}
    \end{equation}
    What remains to be shown is that the right hand side of \eqref{y_hat_t_bound} is bounded. From \eqref{temp_u_t_u_t1}, we first have 
    \begin{equation}
        \|\hat{u}_t - \hat{u}_{t-1} \| = \|\eta \big( (\mathcal{H} + \Delta_{t-1}) \hat{u}_{t-1} + \hat{y}_{t-1} \big)  \| \leq \| \eta \mathcal{H} \hat{u}_{t-1} \| + \|\eta (\Delta_{t-1} \hat{u}_{t-1} + \hat{y}_{t-1}) \|. \label{ppmppklkk}
    \end{equation}

    The first term $\| \eta \mathcal{H} \hat{u}_{t-1}\|$ can be bounded as follows. First note that the Hessian satisfies $-\gamma \leq \lambda(\mathcal{H}) \leq L$ by construction, where $\lambda(\mathcal{H})$ represents eigenvalues of $\mathcal{H}$. Although $L \geq \delta$, it is difficult to compare $L$ and $\gamma$ since $\gamma \geq \delta$. Thus, we bound the term by projecting it to the following two subspaces:
    \begin{itemize}
        \item $S_{-}$: subspace spanned by eigenvectors of $\mathcal{H}$ where eigenvalues are within $[-\gamma, 0]$.
        \item $S_{+}$: subspace spanned by eigenvectors of $\mathcal{H}$ where eigenvalues are within $(0, L]$.
    \end{itemize}
    That is, for the first case
    \begin{align}
        \| \text{Proj}_{S_{-}} \big(\eta \mathcal{H} \hat{u}_{t-1} \big)\| \leq \eta \| \text{Proj}_{S_{-}}(\mathcal{H}) \| \| \hat{u}_{t-1} \| \leq \eta \gamma \frac{3}{2}(1+\eta \gamma)^{t-1}r_0, \label{qtrjegjkn}
    \end{align}
    where we use the bound on $\| \hat{u}_{t-1} \|$ and the fact that $\|\text{Proj}_{S_{-}}(\mathcal{H})  \| \leq \lambda$. 
    For the second case, we have from \eqref{urtmroqroer},
    \begin{align}
        \| \text{Proj}_{S_{+}} (\eta \mathcal{H} \hat{u}_{t-1}) \| &= \| \text{Proj}_{S_{+}}\big( \eta \mathcal{H} (1 + \eta \gamma)^{t-1} r_0 e_1 - \eta \sum_{j=0}^{t-2} \eta \mathcal{H} (I - \eta \mathcal{H})^{t-2-j} (\Delta_j \hat{u}_j + \hat{y}_j) \big) \| \nonumber\\
        &\leq \eta \gamma (1 + \eta \gamma)^{t-1} r_0 + \eta \| \sum_{j = 0}^{t-2} \text{Proj}_{S_{+}} \big(\eta \mathcal{H} (I - \eta \mathcal{H})^{t-2-j} \big) \| \max_{0 \leq j \leq t-2} \| \Delta_j \hat{u}_j + \hat{y}_{j} \| \nonumber\\
        &\leq \eta \gamma (1 + \eta \gamma)^{t-1} r_0 + \sum_{j=0}^{t-2} \frac{\eta}{t-1-j} \max_{0 \leq j \leq t-2} \| \Delta_j \hat{u}_j + \hat{y}_{j} \| \nonumber\\
        &\leq \eta \gamma (1 + \eta \gamma)^{t-1} r_0 + \eta \log(t) \big( \max_{0 \leq j \leq t-2} (\| \Delta_j \|\| \hat{u}_j\| + \| \hat{y}_j \|) \big) \nonumber\\
        &\leq \eta \gamma (1 + \eta \gamma)^{t-1} r_0 + \eta \log(t) \big( \rho (\frac{\delta}{c_2 \rho} + r) \frac{3}{2} (1+\eta \gamma)^{t-2} r_0 + \gamma L (1 + \eta \gamma)^{t-2} r_0\big) \label{inbfghh}\\
        &\leq \eta \gamma (1 + \eta \gamma)^{t-1} r_0 + \eta \log(t) \big(  \frac{3\delta}{c_2}   (1+\eta \gamma)^{t-2} r_0 + \gamma L (1 + \eta \gamma)^{t-2} r_0\big) \nonumber\\
        &\leq \eta \gamma (1 + \eta \gamma)^{t-1} r_0 + \eta \log(\mathscr{T}) \big(  \frac{3\delta}{c_2}   (1+\eta \gamma)^{t-1} r_0 + \gamma L (1 + \eta \gamma)^{t-1} r_0\big) \label{utytjnhj}
    \end{align}
    The second inequality is due to $(1 - \lambda)^{t-1}\lambda \leq \frac{1}{t}$ for $0 < \lambda \leq 1$. That is, given the choice $\eta \leq \widetilde{\mathcal{O}}(\frac{1}{L})$ and $\boldsymbol{0} \preceq \text{Proj}_{S_{+}} (\mathcal{H}) \preceq L I$, we obtain $0 \leq \lambda(\text{Proj}_{S_{+}}(\eta \mathcal{H})) \leq 1$. The third inequality uses the finite-sum bound on Harmonic series. Inequality \eqref{inbfghh} applies (i) the bound on $\| \Delta_t \|$ as in \eqref{delta_t_bound} (ii) the inductive bound on $\| \hat{u}_j \|$ as in \eqref{u_j_inductive_bound} and (iii) the inductive assumption on $\| y_j\|$. The second last inequality is by the choice $r \leq \frac{\delta}{c_2\rho}$. It is easy to verify that the right-hand-side of \eqref{utytjnhj} is larger than right-hand-side of \eqref{qtrjegjkn}. Hence combining the two cases gives
    \begin{equation}
        \| \eta \mathcal{H} \hat{u}_{t-1}\| \leq  \eta \gamma (1 + \eta \gamma)^{t-1} r_0 + \eta \log(\mathscr{T}) \big(  \frac{3\delta}{c_2}   (1+\eta \gamma)^{t-1} r_0 + \gamma L (1 + \eta \gamma)^{t-1} r_0\big). \label{iekmktnhtr}
    \end{equation}
    
    The second term in \eqref{ppmppklkk} is bounded as
    \begin{align}
        \| \eta (\Delta_{t-1} \hat{u}_{t-1} + \hat{y}_{t-1}) \| &\leq \eta \| \Delta_{t-1} \| \| \hat{u}_{t-1} \| + \eta \| \hat{y}_{t-1} \| \nonumber\\
        &\leq \eta \rho \big( \frac{\delta}{c_2 \rho} + r \big) \frac{3}{2}(1 + \eta \gamma)^{t-1}r_0 + \eta \gamma L(1+ \eta\gamma)^{t-1} r_0 \nonumber\\
        &\leq \frac{3\eta \delta}{c_2} (1 + \eta\gamma)^{t-1} r_0 + \eta \gamma L(1+ \eta\gamma)^{t-1} r_0, \label{mbnnb}
    \end{align}
    where the second inequality is derived similarly as \eqref{inbfghh} and the last inequality is again due to the assumption that $r \leq \frac{\delta}{c_2 \rho}$. Combining \eqref{iekmktnhtr} and \eqref{mbnnb} gives a bound,
    \begin{align}
        L \| \hat{u}_t - \hat{u}_{t-1} \| &\leq L \Big( \eta \gamma (1 + \eta \gamma)^{t-1} r_0 + \eta \log(\mathscr{T}) \big(  \frac{3\delta}{c_2}   (1+\eta \gamma)^{t-1} r_0 + \gamma L (1 + \eta \gamma)^{t-1} r_0\big) + \nonumber\\
        &+ \frac{3\eta \delta}{c_2} (1 + \eta\gamma)^{t-1} r_0 + \eta \gamma L(1+ \eta\gamma)^{t-1} r_0 \Big) \nonumber\\
        &\leq \Big( {\eta} + \frac{3\eta\log(\mathscr{T})}{c_2} + \eta L \log(\mathscr{T}) + \frac{3\eta}{c_2} + \eta L \Big) \gamma L (1 + \eta \gamma)^t r_0 \nonumber\\
        &\leq \Big( \frac{3 \eta \log(\mathscr{T})}{c_2} + 2\eta L \log(\mathscr{T}) \Big) \gamma L ( 1+\eta\gamma)^t r_0, \label{former_bound_yhat}
    \end{align}
    where the second inequality uses the definition that $\delta \leq \gamma$ and $1 + \eta \gamma > 1$. The last inequality by considering $1 + \frac{3}{c_2} + L \leq L \log(\mathscr{T} + a)$, for some $a \in \mathbb{N}^+$. This is because, $1/c_2$ is a constant that is sufficiently small while $\log(\mathscr{T}+a)$ can be made sufficiently large to ensure the validity of this result. Here, for simplicity, we still write $\log(\mathscr{T})$ since this will not affect the result. 
    
    Similarly, we can bound $\rho \mathscr{D}_t \| \hat{u}_t \| + \rho \mathscr{D}_{t-1} \| \hat{u}_{t-1}\|$ as
    \begin{align}
        \rho \mathscr{D}_t \| \hat{u}_t \| + \rho \mathscr{D}_{t-1} \| \hat{u}_{t-1}\| &\leq \rho \big( \frac{\delta}{c_2 \rho} +r \big) \Big( \frac{3}{2} (1 + \eta \gamma)^t r_0 + \frac{3}{2}(1+ \eta \gamma)^{t-1} r_0 \Big) \nonumber\\
        &\leq \frac{6\delta}{c_2} (1 + \eta \gamma)^t r_0, \label{latter_bound_yhat}
    \end{align}
    where the first inequality applies the second condition for $j = t \text{ and } t-1$ and the second inequality again uses $r \leq \frac{\delta}{c_2 \rho}$. By combining results in \eqref{former_bound_yhat} and \eqref{latter_bound_yhat}, we have 
    \begin{align}
        \| \hat{y}_t\| &\leq \mathcal{O}\Big( \frac{\log^{\frac{3}{2}} (d \mathscr{T}/ \zeta)}{\sqrt{b}} \Big) \sqrt{m} \Big( \big( \frac{3\eta \log(\mathscr{T})}{c_2} + 2 \eta L \log(\mathscr{T}) \big) \gamma L (1+\eta \gamma)^tr_0 + \frac{6\delta}{c_2} (1 + \eta \gamma)^t r_0 \Big) \nonumber\\
        &\leq C_1 \Big( \frac{3\eta \log(\mathscr{T})}{c_2} + 2 \eta L \log(\mathscr{T}) + \frac{6 }{c_2 L}\Big) \gamma L (1 + \eta \gamma)^t r_0 \nonumber\\
        &\leq \gamma L ( 1+ \eta \gamma)^t r_0 \qquad \qquad \qquad \text{ with probability } 1 - \zeta, \nonumber
    \end{align}
    where we denote $C_1 := \mathcal{O}(\log^{\frac{3}{2}}(d\mathscr{T}/\zeta))$ (note $c_1 = \mathcal{O}(\log^{\frac{3}{2}}(dm/\vartheta)$ in Lemma \ref{large_descent_lemma}). The second inequality considers $b \geq m$ and $ \delta \leq \gamma$. The last inequality follows by the parameter setting that $\eta \leq \frac{1}{8 L \log(\mathscr{T}) C_1}$ and $c_2 \geq \frac{12C_1}{L}$, $C_1 \geq 1$. This completes the proof of the bound on $\| \hat{y}_t\|$. 
    
    Finally, we can proceed to raise a contradiction. First given the bound on $\| q(t)\|$, we have
\begin{align}
    \| \hat{u}_\mathscr{T} \| = \| p(\mathscr{T}) - q(\mathscr{T}) \| \geq \| p(\mathscr{T})\| - \| q(\mathscr{T})\| \geq \| p(\mathscr{T}) \|/2 &= \frac{1}{2} (1 + \eta \gamma)^\mathscr{T} r_0 \nonumber\\
    &= \frac{1}{2} (1 + \eta \gamma)^\mathscr{T} \frac{\nu r}{\sqrt{d}} > \frac{4\delta}{c_2\rho}, \nonumber
\end{align}
where the last inequality follows by the choice that $\mathscr{T} = {2 \log_\alpha( \frac{8 \delta \sqrt{d}}{c_2 \rho \nu r})}/\gamma$, where $\alpha \geq 1$ is chosen such that $\log_\alpha (1 + \eta \gamma) > \gamma$. This requirement is reasonable since when $\alpha \xrightarrow[]{} 1$, $\log_\alpha ( 1 + \eta \gamma)$ increases while $\gamma$ is a constant. In this case, 
\begin{equation}
    \log_\alpha \big((1 + \eta \gamma)^\mathscr{T}\big) > \gamma \cdot 2 \log_\alpha ( \frac{8 \delta \sqrt{d}}{c_2 \rho \nu r})/\gamma > \log_\alpha ( \frac{8 \delta \sqrt{d}}{c_2 \rho \nu r}). \nonumber
\end{equation}
However, we have for $t \leq \mathscr{T}$, $\| \hat{u}_t \| = \| u_t - u_t'\| \leq \| u_t\| + \| u_t'\| \leq 2(r + \frac{\delta}{c_2 \rho}) \leq \frac{4\delta}{c_2 \rho}$, which gives a contradiction. Hence $\exists\, t \leq \mathscr{T} = {2 \log_\alpha( \frac{8 \delta \sqrt{d}}{c_2 \rho \nu r})}/\gamma$, such that $\max\{ \| u_t - u_0 \|, \| u_t' - u_0 \|\} \geq \frac{\delta}{c_2 \rho}$. Given $\gamma$ changes throughout optimization process, we may choose {$\mathscr{T} = {2 \log_\alpha( \frac{8 \delta \sqrt{d}}{c_2 \rho \nu r})}/\delta$. Since $\delta \leq \gamma$, the result still holds}. 
\end{proof}

\begin{lemma}[Descent around saddle points]
\label{descent_around_saddle_lemma}
    Let $x \in \mathcal{M}$ satisfy $\| \emph{grad}F(x) \| \leq \epsilon$ and $\lambda_{\min} (\nabla^2 \hat{F}_x(0))$ $\leq - \delta$. With the same setting as in Lemma \ref{small_stuck_region_lemma}, the two coupled sequences satisfy, with probability $1 - \zeta$, 
    \begin{equation}
        \exists \, t \leq \mathscr{T}, \, \max \{ \hat{F}_x (u_{0}) - \hat{F}_x(u_t) , \hat{F}_x (u_0') - \hat{F}_x (u_t') \} \geq 2 \mathscr{F}, \nonumber
    \end{equation}
    where $\mathscr{F} = \frac{\delta^3}{2c_3 \rho^2}$, $c_3 = \frac{8 \log_\alpha( \frac{8 \delta \sqrt{d}}{c_2 \rho \nu r}) c_2^2}{c_1 L}$, and $c_1$ and $c_2$ are defined in Lemma \ref{large_descent_lemma} and Lemma \ref{small_stuck_region_lemma}, respectively.
\end{lemma}
\begin{proof}
    We will again prove this result by contradiction. Suppose 
    \begin{equation}
        \forall \, t \leq \mathscr{T}, \, \max \{ \hat{F}_x (u_{0}) - \hat{F}_x(u_t) , \hat{F}_x (u_0') - \hat{F}_x (u_t') \} \leq 2 \mathscr{F}. \label{condition_contradiction_}
    \end{equation}
    Then we first claim that both $\{ u_t \}$ and $\{ u_t'\}$ fall within the prescribed ball $\mathbb{B}_x(D)$. This is verified by contradiction. Assume that, without loss of generality, at iteration $j \leq \mathscr{T}$, $\| u_j \| \geq D$. In this case, $\textsf{TSSRG}(x,u_0, \mathscr{T})$ returns $R_x( u_{j-1} - \alpha \eta v_{t-1} )$ with $\alpha \in (0,1)$ such that $\| u_{j-1} - \alpha \eta v_{t-1}\| = D$. Hence
    \begin{align}
        D = \| u_{j-1} - \alpha \eta v_{t-1} \| &\leq \| u_{j-1} - \alpha \eta v_{t-1} - u_0\| + \| u_0\| \leq \sqrt{\frac{4j}{c_1 L} (\hat{F}_x(u_0) - \hat{F}_x(u_j))} + r \nonumber\\
        &\leq \sqrt{ \frac{8 \mathscr{T} \mathscr{F}}{c_1L}} + r = \sqrt{\frac{{ 8 \log_\alpha( \frac{8 \delta \sqrt{d}}{c_2 \rho \nu r})}}{\delta c_1 L} \cdot \frac{\delta^3}{c_3 \rho^2} } + r = \frac{\delta}{c_2 \rho} + r \label{imimiinihi}
    \end{align}
    where we note $\mathscr{F} = \frac{\delta^3}{2c_3 \rho^2}$ where $c_3 = \frac{8 \log_\alpha( \frac{8 \delta \sqrt{d}}{c_2 \rho \nu r}) c_2^2}{c_1 L}$ and $\mathscr{T} = 2 \log_\alpha( \frac{8 \delta \sqrt{d}}{c_2 \rho \nu r})/\delta$. The second inequality uses Lemma \ref{improve_localize_lemma} where we can replace $\eta$ with $\alpha \eta$ for iteration $j$. However, by parameter choice, we have $\frac{\delta}{c_2\rho} + r \leq D$. This gives a contradiction. Therefore, under condition \eqref{condition_contradiction_}, the two coupled sequences do not escape the ball. Hence we can proceed similarly as in \citep{LiSSRGD2019}.
    
    First Lemma \ref{small_stuck_region_lemma} claims that, for some $t \leq \mathscr{T}$, $\max \{ \| u_t - u_0 \|, \|u_t' - u_0 \| \} \geq \frac{\delta}{c_2 \rho}$. Without loss of generality, suppose $\| u_T - u_0 \| \geq \frac{\delta}{c_2 \rho}$, for $T \leq \mathscr{T}$. Then by Lemma \ref{improve_localize_lemma}, we have $\| u_T - u_0 \| \leq \sqrt{\frac{4T}{c_1 L} (\hat{F}_x(u_0) - \hat{F}_x(u_T))}$. This implies 
    \begin{align}
        \hat{F}_x(u_0) - \hat{F}_x(u_T) \geq \frac{c_1 L}{4T} \| u_T - u_0\|^2 \geq \frac{c_1 L \delta^2}{4 \mathscr{T} c_2^2 \rho^2} \geq \frac{c_1 L \delta^3}{8 \log_\alpha( \frac{8 \delta \sqrt{d}}{c_2 \rho \nu r}) c_2^2 \rho^2} \geq \frac{\delta^3}{c_3 \rho^2} = 2 \mathscr{F} \nonumber
    \end{align}
    where we use the choice of $c_3$. This contradicts \eqref{condition_contradiction_}. Therefore the proof is complete. 
\end{proof}

Without loss of generality, the result in Lemma \ref{descent_around_saddle_lemma} can be written as $\max \{ \hat{F}_x (u_{0}) - \hat{F}_x(u_\mathscr{T}) , \hat{F}_x (u_0') - \hat{F}_x (u_\mathscr{T}') \} \geq 2 \mathscr{F}$. This represents the worst case scenario where we require $\mathscr{T}$ iterations of \textsf{TSSRG} to reach a large function decrease.\footnote{We can also add a stopping criterion that breaks with $u_\mathscr{T}$ ($u_\mathscr{T}'$) set to be $u_T$ (resp. $u_T'$) where $T$ is the earliest iteration where a large function value decrease is reached.} 

\begin{lemma}[Escape stuck region] \emph{(Lemma \ref{escape_stuck_region_lemma_informal} in the main text).}
\label{escape_stuck_region_lemma}
Let $x \in \mathcal{M}$ satisfying $\| \emph{grad}F(x)\| \leq \epsilon$ and $-\gamma = \lambda_{\min} (\nabla^2 \hat{F}_x(0) )$ $\leq - \delta$. Given that result in Lemma \ref{descent_around_saddle_lemma} holds and choosing perturbation radius $r \leq \min\{ \frac{\delta^3}{4c_3 \rho^2 \epsilon}, \sqrt{\frac{\delta^3}{2 c_3 \rho^2 L}} \}$, we have a sufficient decrease of function value with high probability: 
    \begin{equation}
        {F}(\textsf{TSSRG}(x,u_0, \mathscr{T}) ) - F(x) \leq -\mathscr{F} \quad \text{ with probability } 1- \nu \nonumber.
    \end{equation}
\end{lemma}
\begin{proof}
    First define the stuck region formally as
    \begin{equation}
        \mathcal{X}_{\text{stuck}} = \{ u \in \mathbb{B}_x(r) : F( \textsf{TSSRG}(x, u, \mathscr{T}) ) - \hat{F}_x(u) \geq - 2\mathscr{F} \}. \nonumber
    \end{equation}
    This suggests that running \textsf{TSSRG} from any point initialized in this region will not give sufficient decrease of function value. Similar to \citep{LiSSRGD2019}, the aim is to show this region is small in volume. First note that iterates with initialization within $\mathcal{X}_{\text{stuck}}$ do not escape the constraint ball from arguments in \eqref{imimiinihi}. Hence, if iterates leave the ball $\mathbb{B}_x(D)$, the output already escapes the stuck region with large function decrease. Given that tangent space $T_x\mathcal{M}$ is a $d$-dimensional vector space, we can perform similar analysis as in \citep{JinPGD2017,LiSSRGD2019}.
    
    Consider the two coupled sequences with starting points $u_0$ and $u_0'$ satisfying $u_0 - u_0' = r_0 e_1$, where $e_1$ is the smallest eigenvector of $\nabla^2 \hat{F}_x(0)$ and $r_0 = \frac{\nu r}{\sqrt{d}}$. Therefore from Lemma \ref{descent_around_saddle_lemma}, at least one of the two sequences finally leaves $\mathcal{X}_{\text{stuck}}$ after $\mathscr{T}$ steps with probability $1 - \zeta$. Therefore, under this result, the width of the stuck region along direction $e_1$ is at most $r_0$. Based on similar argument as in \citep{CriscitielloEESPManifold2019}, $\text{Vol}(\mathcal{X}_{\text{stuck}}) \leq r_0 \text{Vol}(\mathcal{B}_r^{d-1})$, where $\mathcal{B}_r^d$ represents $d$-dimensional sphere with radius $r$. Therefore, 
    \begin{equation}
        \text{Pr} \{ u_0 \in \mathcal{X}_{\text{stuck}} \} = \frac{\text{Vol}(\mathcal{X}_{\text{stuck}})}{\text{Vol}(\mathcal{B}_r^d)} \leq \frac{r_0 \text{Vol}(\mathcal{B}_r^{d-1})}{\text{Vol}(\mathcal{B}_r^d)} = \frac{r_0 \Gamma (\frac{d}{2}+1)}{\sqrt{\pi} r \Gamma(\frac{d}{2} + \frac{1}{2})} \leq \frac{r_0}{\sqrt{\pi} r} \sqrt{\frac{d}{2} + 1} \leq \frac{r_0\sqrt{d}}{r} = \nu, \nonumber
    \end{equation}
    where we have used Gautschi's inequality for the Gamma function. This claims the probability of $u_0 \in \mathcal{X}_{\text{stuck}}$ is at most $\nu$. Therefore with high probability at least $1 - \nu$, $u_0 \notin \mathcal{X}_{\text{stuck}}$. In this case, function value decreases a lot. This is verified as follows. First note that $F(x) = \hat{F}_x(0)$ and $F(\textsf{TSSRG}(x, u_0, \mathscr{T})) = \hat{F}_x(u_\mathscr{T})$. Then, by $L$-Lipschitzness of the gradient $\nabla \hat{F}_x$,
    \begin{align}
        \hat{F}_x(u_0) - \hat{F}_{x}(0) \leq \langle \nabla \hat{F}_x(0), u_0 \rangle + \frac{L}{2} \| u_0 \|^2 &\leq \|\nabla \hat{F}_x(0) \| \|u_0\| + \frac{L}{2} \| u_0\|^2 \leq \epsilon r + \frac{L r^2}{2} \nonumber\\
        &\leq \frac{\delta^3}{2 c_3 \rho^2} =  \mathscr{F}, \label{unrutr}
    \end{align}
    using $\| u_0\| \leq r, \| \nabla \hat{F}_x(0)\|\leq \epsilon$ and also $r \leq \min\{ \frac{\delta^3}{4c_3 \rho^2 \epsilon}, \sqrt{\frac{\delta^3}{2 c_3 \rho^2 L}} \}$. Therefore
    \begin{align}
        F(x) - F(\textsf{TSSRG}(x, u_0, \mathscr{T})) = \hat{F}_x(0) - \hat{F}_x(u_0) + \hat{F}_x(u_0) -  \hat{F}_x(u_\mathscr{T}) \geq - \mathscr{F}+ 2\mathscr{F} = \mathscr{F}, \nonumber
    \end{align}
    where we apply \eqref{unrutr} and the fact that $u_0 \notin \mathcal{X}_{\text{stuck}}$. 
\end{proof}

\section{Proof for online setting}
Here we provide proof for online problems where full gradient $\text{grad}F$ needs to be replaced by large-batch gradient $\text{grad}f_{\mathcal{B}}$. The proof is similar to that of finite-sum problems and thus we mainly present the key steps in the following proof.

\subsection{Proof for main Theorem}
\textbf{Theorem \ref{online_complexity_theorem} (Online complexity)}
\textit{ Under Assumptions \ref{bounded_function_assump} to \ref{bounded_variance_assump}, consider online optimization setting. For any starting point $x_0 \in \mathcal{M}$ with the choice of parameters 
\begin{equation*}
    \mathscr{T} = \widetilde{\mathcal{O}}\Big(  \frac{1}{\delta} \Big), \quad \eta \leq \widetilde{\mathcal{O}} \Big( \frac{1}{L} \Big), \quad m = b = \widetilde{\mathcal{O}}\Big( \frac{\sigma}{\epsilon} \Big), \quad B = \widetilde{\mathcal{O}}\Big( \frac{\sigma^2}{\epsilon^2} \Big), \quad r = \widetilde{\mathcal{O}} \Big(\min \Big\{ \frac{\delta^3}{\rho^2\epsilon}, \sqrt{\frac{\delta^3}{\rho^2L}} \Big\} \Big),
\end{equation*}
suppose $\epsilon, \delta$ satisfy $\epsilon \leq \min \big\{ \frac{\delta^2}{\rho}, \widetilde{\mathcal{O}}( \frac{D\sqrt{L}}{m\sqrt{\eta}} ) \big\}$, $\delta \leq \ell$ where $\ell \leq L$. Then with high probability, $\textsf{PRSRG}(x_0, \eta, m, b, B, r, \mathscr{T}, D, \epsilon)$ will at least once visit an $(\epsilon, \delta)$-second-order critical point with 
\begin{equation*}
    \widetilde{\mathcal{O}}\Big( \frac{\Delta L \sigma}{\epsilon^3} + \frac{\Delta \rho^2 \sigma^2}{\delta^3 \epsilon^2} + \frac{\Delta \rho^2 \sigma}{\delta^4 \epsilon}\Big)
\end{equation*}
stochastic gradient oracles, where $\Delta := F(x_0) - F^*$. }

\vspace*{9pt}
\begin{proof}
    Similar to Theorem \ref{theorem_finite_sum}, we have the following possible cases when running the main algorithm \textsf{PRSRG}:
\begin{itemize}
    \item {Large gradients where ${\| \text{grad}f_{\mathcal{B}}(x) \| \geq \epsilon}$}.
    \begin{enumerate}
        \item \textit{Type-1 descent epoch}: Iterates escape the constraint ball.
        \item Iterates do not escape the constraint ball. 
        \begin{enumerate}
            \item \textit{Type-2 descent epoch}: At least half of iterates in current epoch have pullback gradient larger than $\epsilon/4$. 
            \item \textit{Useful epoch}: At least half of iterates in the current epoch have pullback gradient no larger than $\epsilon/4$ and the output $\Tilde{x}$ from the current epoch has batch gradient $\| \text{grad}f_{\mathcal{B}}(x)\|$ no larger than $\epsilon$. (Since the output satisfies small gradient condition, the next epoch will run \textsf{TSSRG}$(\Tilde{x}, u_0, \mathscr{T})$ to escape saddle points).
            \item \textit{Wasted epoch}: At least half of iterates in the current epoch have pullback gradient no larger than $\epsilon/4$ and the output $\Tilde{x}$ from the current epoch has batch gradient $\| \text{grad}f_{\mathcal{B}}(x)\|$ larger than $\epsilon$.
        \end{enumerate}
    \end{enumerate}
    \item {Around saddle points where ${\| \text{grad} f_{\mathcal{B}}(x) \| \leq \epsilon}$ and ${\lambda_{\min} (\text{Hess}(x)) \leq -\delta }$}
    \begin{enumerate}[start=3]
        \item \textit{Type-3 descent epoch}: The current iterate is around saddle point.
    \end{enumerate}
\end{itemize}
The following proof is very similar to that of Lemma \ref{theorem_finite_sum}. First from Lemma \ref{large_gradien_descent_lemma_online}, we know that the probability of wasted epoch occurring is at most $2/3$. Given independence of different wasted epoch, with high probability, wasted epoch happens consecutively at most $\widetilde{\mathcal{O}}(1)$ times before either a descent epoch (either Type 1 or 2) or a useful epoch. We use $N_1, N_2, N_3$ to respectively represent three types of descent epoch. 

For \textit{Type-1 descent epoch}, with probability $1-\vartheta$, function value decrease by at least $\frac{3\eta m \epsilon^2}{512}$. Hence by standard concentration, after $N_1$ such epochs, function value is decreased by ${\mathcal{O}}(\eta m \epsilon^2 N_1)$ with high probability. Given that $F(x)$ is bounded by $F^*$, the decrease cannot exceed $\Delta := F(x_0) - F^*$. Thus, $N_1 \leq \mathcal{O}(\frac{\Delta}{\eta m \epsilon^2})$. Similarly, for \textit{Type-2 descent epoch}, $N_2 \leq \mathcal{O}(\frac{\Delta}{\eta m \epsilon^2})$. For \textit{Type-3 useful epoch}, the output is either an $(\epsilon,\delta)$-second-order critical point or the epoch is immediately followed by a \textit{Type-3 descent epoch} around saddle points. From Lemma \ref{escape_stuck_region_lemma_online}, we know that function value decreases by $\mathscr{F}= \frac{\delta^3}{2c_3 \rho^2}$ with high probability. Therefore by similar arguments, $N_3 \leq \widetilde{\mathcal{O}}(\frac{\rho^2\Delta}{\delta^3})$. 

Therefore to combine them, we have the following stochastic gradient complexity: 
\begin{align}
        &(N_1 + N_2)\big(\widetilde{\mathcal{O}}(1) \big( B + mb\big) + B + mb  \big) + N_3 \big( \widetilde{\mathcal{O}}(1) \big( B + mb \big) + \lceil \mathscr{T}/m \rceil B + \mathscr{T} b \big) \nonumber\\
        &\leq \widetilde{\mathcal{O}} \Big( \frac{\Delta L \sigma}{\epsilon^3} + \frac{\Delta \rho^2 \sigma^2}{\delta^3 \epsilon^2} + \frac{\rho^2 \Delta}{\delta^3} \big( \frac{\sigma^2}{\epsilon^2} + \frac{\sigma}{\delta \epsilon} \big) \Big) \nonumber\\
        &\leq \widetilde{\mathcal{O}}\Big( \frac{\Delta L \sigma}{\epsilon^3} + \frac{\Delta \rho^2 \sigma^2}{\delta^3 \epsilon^2} + \frac{\Delta \rho^2 \sigma}{\delta^4 \epsilon}\Big), \nonumber
    \end{align}
where $\mathscr{T} = \widetilde{\mathcal{O}}(\frac{1}{\delta})$, $m = b = \sqrt{B} = \widetilde{\mathcal{O}}(\frac{\sigma}{\epsilon})$ and $\eta = \widetilde{\mathcal{O}}(\frac{1}{L})$.
\end{proof}

\subsection{Proof for key Lemmas}
\begin{lemma}[High probability bound on estimation error]
\label{hp_estimation_error_online}
Under Assumptions \ref{L_lipschit_assump} and \ref{bounded_variance_assump}, we have the following high probability bound for estimation error of modified gradient under online setting. That is, for $sm+1 \leq t \leq (s+1)m$,
\begin{equation*}
    \| v_t - \nabla \hat{F}_x(u_t) \| \leq  \frac{\mathcal{O}(\log^{\frac{3}{2}}(d/\vartheta)) L}{\sqrt{b}}  \sqrt{\sum_{j=sm+1}^t \| u_j - u_{j-1}\|^2} + \frac{\mathcal{O}(\log(d/\vartheta)) \sigma}{\sqrt{B}} \, \, \text{ with probability } 1 - \vartheta.
\end{equation*}
\end{lemma}

\begin{proof}
For simplicity of notation, consider a single epoch $k = 1,...,m, \forall \, s$. Because under online setting, $v_{0} = \nabla \hat{f}_{\mathcal{B}, x}(u_{0}) \neq \nabla \hat{F}_x(u_{0})$, we first need to bound $\| v_{0} - \nabla \hat{F}_x(u_{0})\|$ by Bernstein inequality. By assumption of bounded variance, we have 
\begin{align}
    \| \nabla \hat{f}_{i,x} (u_{0}) - \nabla \hat{F}_x(u_{0}) \| \leq \sigma \text{ and } \sum_{i \in \mathcal{B}} \| \nabla \hat{f}_{i,x} (u_{0}) - \nabla \hat{F}_x(u_{0}) \|^2 \leq B\sigma^2. \nonumber
\end{align}
By Lemma \ref{bernsteinlemma}, 
\begin{align}
    \text{Pr} \{ \| \sum_{i \in \mathcal{B}} \big( \nabla \hat{f}_{i,x} (u_{0}) - \nabla \hat{F}_x(u_{0}) \big) \| \geq \varsigma \} &= \text{Pr} \{ \| \frac{1}{B} \sum_{i \in \mathcal{B}} \big( \nabla \hat{f}_{i,x} (u_{0}) - \nabla \hat{F}_x(u_{0}) \big) \| \geq \varsigma/B \} \nonumber\\
    &= \text{Pr} \{ \| v_{0} - \hat{F}_x(u_{0}) \| \geq \varsigma/B \} \nonumber\\
    &\leq (d+1) \exp \big( \frac{-\varsigma^2/2}{B\sigma^2 + \sqrt{B}\sigma \varsigma/3} \big) \leq \vartheta, \nonumber
\end{align}
where the first inequality also considers $\sqrt{B} \geq 1$. The last inequality holds by setting $\varsigma = \mathcal{O}(\log(d/\vartheta))\sqrt{B} \sigma$. Thus, we obtain 
\begin{equation}
    \| v_{0} - \hat{F}_x(u_{0}) \| \leq \frac{\mathcal{O}(\log(d/\vartheta)) \sigma}{\sqrt{B}} \, \text{ with probability } 1- \vartheta. \label{v_0_bound_temp}
\end{equation}
Next, denote $y_k := v_k - \nabla \hat{F}_x(u_k)$ and $z_k := y_k - y_{k-1}$. Then we follow the same steps as in Lemma \ref{hpBound_estimation_error} to obtain the bound on $\| y_k\|$ except that $y_{0} \neq 0$. From \eqref{mmnnkhkh}, we have 
\begin{equation}
    \| y_k - y_0\| \leq \mathcal{O}\Big( \frac{\log^{\frac{3}{2}}(d/\vartheta) L}{\sqrt{b}} \Big) \sqrt{\sum_{j=1}^k \| u_j - u_{j-1}\|^2} \, \text{ with probability } 1 - 2\vartheta. \nonumber
\end{equation}
By union bound, for $k \in [1, m]$, with probability at least $1-3\vartheta$,
\begin{equation}
    \| v_k - \nabla \hat{F}_x(u_k) \| = \| y_k \| \leq \| y_k - y_0 \| + \| y_0\| \leq  \frac{\mathcal{O}(\log^{\frac{3}{2}}(d/\vartheta)) L}{\sqrt{b}}  \sqrt{\sum_{j=1}^k \| u_j - u_{j-1}\|^2} + \frac{\mathcal{O}(\log(d/\vartheta)) \sigma}{\sqrt{B}}. \nonumber
\end{equation}
The proof is complete by setting $\vartheta/3$ as $\vartheta$.
\end{proof}

\begin{lemma}[Improve or localize]
\label{improve_localize_online}
    Consider $\{ u_t \}_{t=1}^{\mathscr{T}}$ as the sequence generated by running $\textsf{TSSRG}(x,$ $u_0, \mathscr{T})$. Suppose we choose $b \geq m$, $\eta \leq \frac{1}{2c_1 L}$, where $c_1 = \mathcal{O}(\log^{\frac{3}{2}}(\frac{dt}{\vartheta}))$. Then we have
    \begin{equation*}
         \| u_t - u_0\| \leq \sqrt{\frac{4t(\hat{F}_x(u_{0}) - \hat{F}_x(u_{t}))}{c_1 L} + \frac{2 \eta t^2 c_1'^2 \sigma^2}{c_1 L B} } \quad \text{ with probability } 1 - \vartheta,
    \end{equation*}
    with $c_1' = \mathcal{O}(\log(\frac{dt}{\vartheta m}))$.
\end{lemma}
\begin{proof}
First we generalize \eqref{ytniemuow} to any epoch (i.e. $1 \leq t \leq \mathscr{T}$) as  
\begin{align}
    &\hat{F}_x(u_t) - \hat{F}_x(u_{sm}) \nonumber\\
    &\leq -\frac{\eta}{2} \sum_{j=sm+1}^t \| \nabla \hat{F}_x(u_{j-1}) \|^2 - (\frac{1}{2 \eta} - \frac{L}{2} - \frac{\eta c_1^2 L^2}{2}) \sum_{j=sm+1}^t \| u_j - u_{j-1} \|^2 + \frac{\eta (t -sm) c_1'^2 \sigma^2}{2B} \nonumber\\
    &\leq - (\frac{1}{2 \eta} - \frac{L}{2} - \frac{\eta c_1^2 L^2}{2}) \sum_{j=sm+1}^t \| u_j - u_{j-1} \|^2 + \frac{\eta (t -sm) c_1'^2 \sigma^2}{2B} \nonumber\\
    &\leq -\frac{c_1 L}{4} \sum_{j=sm+1}^t \| u_j - u_{j-1} \|^2 +  \frac{\eta (t -sm) c_1'^2 \sigma^2}{2B} \quad \text{with probability } 1- \vartheta, \label{iemqpier}
\end{align}
where we choose $\eta \leq \frac{1}{2c_1 L}$ and assume $c_1 \geq 1$. Summing \eqref{iemqpier} for all epochs up to $t$ gives
\begin{equation*}
    \hat{F}_x(u_t) - \hat{F}_x(u_{0}) \leq -\frac{c_1 L}{4} \sum_{j=1}^t \| u_j - u_{j-1}\|^2 + \frac{\eta t c_1'^2 \sigma^2}{2B}.
\end{equation*}
Lastly by Cauchy-Schwarz inequality and triangle inequality, $\sqrt{t \sum_{j=1}^t \| u_j - u_{j-1} \|^2} \geq \| u_t - u_0\|$. Hence,
\begin{equation}
    \| u_t - u_0\| \leq \sqrt{\frac{4t(\hat{F}_x(u_{0}) - \hat{F}_x(u_{t}))}{c_1 L} + \frac{2 \eta t^2 c_1'^2 \sigma^2}{c_1 L B} }. \nonumber
\end{equation}
The proof is now complete.
\end{proof}

\begin{lemma}[Large gradient descent lemma]
\label{large_gradien_descent_lemma_online}
    Under Assumptions \ref{L_lipschit_assump} to \ref{bounded_variance_assump}, suppose we set $\eta \leq \frac{1}{2c_1 L}, b \geq m, B \geq \frac{256 c_1'^2\sigma^2}{\epsilon^2}$, where $c_1 = \mathcal{O}(\log^{\frac{3}{2}}(\frac{dm}{\vartheta})), c_1' = \mathcal{O}(\log(\frac{d}{\vartheta }))$. Consider $x \in \mathcal{M}$ where $\| \emph{grad}f_{\mathcal{B}}(x) \| \geq \epsilon$, with $\epsilon \leq \frac{D}{m} \sqrt{\frac{32c_1 L}{\eta}}$. Then by running $\textsf{TSSRG}(x,0,m)$, we have the following three cases:
    \begin{enumerate}
        \item When the iterates $\{ u_j \}_{j=1}^m$ do not leave the constraint ball $\mathbb{B}_x(D)$:
        \begin{enumerate}
            \item If at least half of the iterates in the epoch satisfy $\| \nabla \hat{F}_x(u_j) \| \leq \epsilon/4$ for $j = 1,...,m$, then with probability at least $1/3$, we have $\| \emph{grad}f_{\mathcal{B}} (\textsf{TSSRG}(x,0,m))\| \leq \epsilon$.
            \item Otherwise, with probability at least $1/5$, we have $F(\textsf{TSSRG}(x,0,m)) - F(x) \leq -\frac{3\eta m \epsilon^2}{512}$.
        \end{enumerate}
        \item When one of the iterates $\{ u_j \}_{j=1}^m$ leaves the constraint ball $\mathbb{B}_x(D)$, with probability at least $1- \vartheta$, we have $F(\textsf{TSSRG}(x,0,m)) - F(x) \leq -\frac{3\eta m \epsilon^2}{512}$.
    \end{enumerate}
    Regardless which case occurs, $F(\textsf{TSSRG}(x,0,m)) - F(x) \leq \frac{\eta t c_1'^2 \sigma^2}{2B}$ with high probability.
\end{lemma}
\begin{proof}
Similar to proof of Lemma \ref{large_descent_lemma}, we only need to consider a single epoch in \textsf{TSSRG}$(x,0,m)$, i.e. $t = 1,...,m$. We also divide the proof into two parts. 

\textbf{1. Iterates do not leave the constraint ball.} From Lemma \ref{hp_estimation_error_online} and union bound, we have for all $1 \leq \tau \leq t$,
\begin{equation}
    \| v_\tau - \nabla \hat{F}_x(u_\tau) \|^2 \leq \frac{c_1^2 L^2}{{b}} \sum_{j=1}^{\tau-1} \| u_j - u_{j-1}\|^2 + \frac{c_1'^2 \sigma^2}{B} \, \, \text{ with probability } 1- \vartheta. \label{turunvgd}
\end{equation}
where we denote $c_1 := \mathcal{O}(\log^{\frac{3}{2}}(\frac{dt}{\vartheta})), c_1' = \mathcal{O}(\log(\frac{dt}{\vartheta m}))$. \footnote{More precisely, $c_1' = \mathcal{O}(\log(\frac{d}{\vartheta} \lceil \frac{t}{m} \rceil )) = \mathcal{O}(\log(\frac{dt}{\vartheta m}))$.} We also use $(A + B)^2 \leq 2A^2 + 2B^2$. Summing up \eqref{tuioioo} from $j=1,...,t$, we have 
\begin{align}
    &\hat{F}_x (u_t) - \hat{F}_x(u_0) \nonumber\\
    &\leq -\frac{\eta}{2} \sum_{j=1}^t \| \nabla \hat{F}_x(u_{j-1}) \|^2 + \frac{\eta}{2} \sum_{j=1}^t \| v_{t-1} - \nabla \hat{F}_x(u_{t-1}) \|^2 - (\frac{1}{2\eta} - \frac{L}{2}) \sum_{j=1}^t \| u_t - u_{t-1} \|^2  \nonumber\\
    &\leq -\frac{\eta}{2} \sum_{j=1}^t \| \nabla \hat{F}_x(u_{j-1}) \|^2 - (\frac{1}{2\eta} - \frac{L}{2}) \sum_{j=1}^t \| u_t - u_{t-1} \|^2 + \frac{\eta c_1^2 L^2}{2b} \sum_{k=1}^{t-1} \sum_{j=1}^k \| u_j - u_{j-1} \|^2 + \frac{\eta t c_1'^2 \sigma^2}{2 B} \nonumber\\
    &\leq -\frac{\eta}{2} \sum_{j=1}^t \| \nabla \hat{F}_x(u_{j-1}) \|^2 - (\frac{1}{2\eta} - \frac{L}{2}) \sum_{j=1}^t \| u_t - u_{t-1} \|^2 + \frac{\eta c_1^2 L^2 m}{2b} \sum_{j=1}^t \| u_j - u_{j-1} \|^2 + \frac{\eta t c_1'^2 \sigma^2}{2 B} \nonumber\\
    &\leq -\frac{\eta}{2} \sum_{j=1}^t \| \nabla \hat{F}_x(u_{j-1}) \|^2 - (\frac{1}{2 \eta} - \frac{L}{2} - \frac{\eta c_1^2 L^2}{2}) \sum_{j=1}^t \| u_j - u_{j-1} \|^2 + \frac{\eta t c_1'^2 \sigma^2}{2B} \label{ytniemuow}\\
    &\leq -\frac{\eta}{2} \sum_{j=1}^t \| \nabla \hat{F}_x(u_{j-1}) \|^2 + \frac{\eta t c_1'^2 \sigma^2}{2B}, \label{uqmqmrbbvty}
\end{align}
where the last inequality holds due to the choice $\eta \leq \frac{1}{2c_1 L} \leq \frac{\sqrt{4c_1^2 + 1} - 1}{2c_1^2 L}$ by assuming $c_1 \geq 1$. Note that we need to ensure \eqref{uqmqmrbbvty} to hold for all $t \leq m$. Hence we change $c_1 = \mathcal{O}(\log^{\frac{3}{2}}(\frac{dm}{\vartheta})), c_1' = \mathcal{O}(\log(\frac{d}{\vartheta}))$. Here are two cases.
\begin{itemize}
    \item \textbf{(Case 1a)} If at least half of the iterates in the epoch satisfy $\| \nabla \hat{F}_x(u_j) \| \leq \epsilon/4$ for $j = 1,...,m$, then by uniformly sampling $\{ u_j\}_{j=1}^m$ (i.e. uniformly breaking by setting $u_t$ as $u_m$ as in Algorithm \ref{TSSRG_algorithm}, Line \ref{breaking_rules_tssrg}), the output $u_m$ satisfies $\| \nabla \hat{F}_x(u_m)\| \leq \epsilon/4$ with probability at least $\frac{1}{2}$. Under online setting, full gradient $\text{grad}F$ is inaccessible and we need to use approximated batch gradient $\text{grad}f_{\mathcal{B}}$ to check the small-gradient condition in Line \ref{line_check_smallgrad}, Algorithm \ref{PRSRG_algorithm}. We know based on \eqref{v_0_bound_temp}, $\| \nabla \hat{f}_{\mathcal{B},x}(u_m) - \nabla \hat{F}_x (u_m) \| \leq \frac{c_1' \sigma}{\sqrt{B}}$ with probability $1 - \vartheta$. By a union bound, with probability at least $\frac{1}{2} - \vartheta$, 
    \begin{equation*}
        \| \nabla \hat{f}_{\mathcal{B},x}(u_{m}) \| \leq \| \nabla \hat{f}_{\mathcal{B},x}(u_{m}) - \nabla \hat{F}_x (u_{m}) \| + \|  \nabla \hat{F}_x(u_{m}) \| \leq \frac{c_1' \sigma}{\sqrt{B}} + \frac{\epsilon}{4} \leq \frac{\epsilon}{2},
    \end{equation*}
    where the last inequality holds by $B \geq \frac{256c_1'^2\sigma^2}{\epsilon^2}$. Without loss of generality, we set $\vartheta \leq \frac{1}{6}$ and therefore the probability reduces to $\frac{1}{3}$. 
    Lastly, from the definition of the pullback gradient, we have 
    \begin{equation*}
        \| \text{grad}f_{\mathcal{B}}(\textsf{TSSRG}(x,0,m) )\| = \| \text{grad}f_{\mathcal{B}}(R_x(u_m)) \| \leq \| (T^*_{u_m})^{-1}\| \| \nabla \hat{f}_{\mathcal{B},x}(u_m)\| \leq \epsilon,
    \end{equation*}
    holds with probability at least $\frac{1}{3}$. The last inequality is due to Lemma \ref{singular_value_bound}. Note that in this case, we also have $\| \text{grad}F(\textsf{TSSRG}(x,0,m)) \| \leq \| (T^*_{u_m})^{-1}\| \| \nabla \hat{F}_{x}(u_m)\| \leq \epsilon/2$. 
    \item \textbf{(Case 1b)} If at least half of the points in the epoch satisfy $\| \nabla \hat{F}_x(u_j) \| \geq \epsilon/4$ for $j = 1,...,m$, with probability at least $\frac{1}{4}$, the selected output $u_m$ falls within the last quarter of $\{ u_j \}_{j=1}^m$. In this case, $\sum_{j=1}^t \| \nabla \hat{F}_x(u_{j-1}) \|^2 \geq \frac{m}{4} \cdot \frac{\epsilon^2}{16} = \frac{m \epsilon^2}{64}$. Note that \eqref{uqmqmrbbvty} holds with probability at least $1 - \vartheta$. By a union bound, we have with probability at least $\frac{1}{5}$ (by letting $\vartheta \leq \frac{1}{20}$), 
    \begin{align}
        F(\textsf{TSSRG}(x,0,m)) - F(x) = \hat{F}_x(u_t) - \hat{F}_x(0) &\leq -\frac{\eta}{2} \sum_{j=1}^t \| \nabla \hat{F}_x(u_{j-1}) \|^2 + \frac{\eta t c_1'^2 \sigma^2}{2B} \nonumber\\
        &\leq -\frac{\eta m \epsilon^2}{128} + \frac{\eta m \epsilon^2}{512} = -\frac{3\eta m \epsilon^2}{512}, \nonumber
    \end{align}
    where we use $B \geq \frac{256c_1'^2 \sigma^2}{\epsilon^2}$ and $t \leq m$. 
\end{itemize}

\textbf{2. Iterates leave the constraint ball.}
Suppose at $\tau \leq m$, we have $\| u_\tau \| > D$. Then by Lemma \ref{improve_localize_online}, we know the function value already decreases a lot. That is with probability $1 - \vartheta$,
\begin{align}
    \hat{F}_x(0) - \hat{F}_x(u_\tau) \geq \frac{c_1 L}{4\tau} \| u_\tau \|^2 - \frac{\eta \tau c_1'^2 \sigma^2}{2B} \geq \frac{c_1 L D^2}{4 m} - \frac{\eta m c_1'^2 \sigma^2}{2B} \geq \frac{3 \eta m \epsilon^2}{512}, \nonumber
\end{align}
where the last inequality follows from the choice that $\epsilon \leq \frac{D}{m}\sqrt{\frac{32 c_1 L}{\eta}}$ and $B \geq \frac{256c_1'^2 \sigma^2}{\epsilon^2}$. Hence by returning $u_\tau$ as $u_m$, we have with high probability,
\begin{equation}
    F(\textsf{TSSRG}(x, 0, m)) - F(x) \leq -\frac{3\eta m \epsilon^2}{512}. \nonumber
\end{equation}
In summary, with high probability, either gradient norm of the output is small or the function value decreases a lot. Even under Case 1a, $\hat{F}_x (u_t) - \hat{F}_x(u_0) \leq  \frac{\eta t c_1'^2 \sigma^2}{2B}$ with high probability by \eqref{uqmqmrbbvty}. 
\end{proof}

\begin{lemma}[Small stuck region]
\label{small_stuck_region_online}
    Consider $x \in \mathcal{M}$ with $\| \emph{grad}f_{\mathcal{B}}(x) \| \leq \epsilon$, $- \gamma := \lambda_{\min} (\nabla^2 \hat{F}_x(0))$ $= \lambda_{\min}(\emph{Hess} F(x))\leq - \delta$ and $L \geq \delta$. Let $u_0, u_0' \in T_x\mathcal{M}$ be two random perturbations, satisfying $\| u_0\|, \| u_0'\| \leq r$ and $u_0 - u_0' = r_0 e_1$, where $e_1$ denotes the smallest eigenvector of $\nabla^2 \hat{f}_x(0)$ and $r_0 = \frac{\nu r}{\sqrt{d}}, r \leq \frac{\delta}{c_2\rho}$. Also set parameters $\mathscr{T} = \frac{2 \log_{\alpha} ( \frac{8\delta\sqrt{d}}{c_2 \rho \nu r})}{\delta} = \widetilde{\mathcal{O}}(\frac{1}{\delta})$, $\eta \leq \min \{\frac{1}{8 \log(\mathcal{T}) C_1 L}, \frac{1}{16 \log_{\alpha} ( \frac{8\delta\sqrt{d}}{c_2 \rho \nu r})L } \} = \widetilde{\mathcal{O}}(\frac{1}{L})$ where $\alpha \geq 1$ is chosen sufficiently small such that $\log_\alpha(1+\eta \gamma) > \gamma$. Also choose $B \geq\mathcal{O}(\frac{c_2'^2\sigma^2}{\epsilon^2}) = \widetilde{\mathcal{O}}(\frac{\sigma^2}{\epsilon^2})$, $\epsilon \leq \frac{\delta^2}{\rho}$. Then for $\{ u_t\}, \{u_t' \}$ generated by running \textsf{TSSRG}$(x, u_0, \mathscr{T})$ and \textsf{TSSRG}$(x, u_0', \mathscr{T})$ with same sets of mini-batches, with probability $1 - \zeta$, 
    \begin{equation}
        \exists \, t \leq \mathscr{T}, \, \max \{ \| u_t - u_0 \|, \| u_t' - u_0' \| \} \geq \frac{\delta}{c_2 \rho},  \nonumber
    \end{equation}
    where $c_2 \geq \max\{ \frac{12C_1}{L}, \frac{2\delta}{D\rho}\}$, and $C_1 = \mathcal{O}(\log^{\frac{3}{2}}(d\mathscr{T}/\zeta)) =\widetilde{\mathcal{O}}(1)$, $c_2' = \mathcal{O}(\log(\frac{d\mathscr{T}}{\zeta m})) = \widetilde{\mathcal{O}}(1)$.
\end{lemma}
\begin{proof}
The proof is by contradiction. So we assume 
\begin{equation}
    \forall \, t \leq \mathscr{T}, \, \max \{ \| u_t - u_0 \|, \| u_t' - u_0 \|  \} \leq \frac{\delta}{c_2 \rho}. \label{rtuijrym}
\end{equation}
First, we again note that both $\{ u_j \}_{j=0}^\mathscr{T} \text{ and } \{ u_j'\}_{j=0}^\mathscr{T}$ do not escape the constraint ball under condition \eqref{rtuijrym}. This is because $\| u_t\| \leq \| u_t - u_0\| + \| u_0\| \leq \frac{\delta}{c_2 \rho} + r \leq D$ by the choice $r \leq \frac{\delta}{c_2 \rho}$ and $c_2 \geq \frac{2\delta}{D \rho}$. Next, we can show an exponential growth in the distance between two coupled sequences and will eventually exceed the bound in \eqref{rtuijrym}. The proof roadmap is exactly the same as in Lemma \ref{small_stuck_region_lemma} and thus we only highlights the main results under online settings. 

Denote $\hat{u}_t := u_t - u_t'$, $\mathcal{H} := \nabla^2 \hat{F}_x(0), \Delta_t := \int_0^1 [ \nabla^2 \hat{F}_x(u_t' + \theta(u_t - u_t')) - \mathcal{H} ] d\theta$, $\hat{y}_t := y_t - y_t' := v_t - \nabla \hat{F}_x(u_t) - v_t' + \nabla \hat{F}_x(u_t')$. Recall we can bound $\| \Delta_t\|$ as $\| \Delta_t\| \leq \rho \mathscr{D}_t \leq \rho(\frac{\delta}{c_2 \rho} + r)$, where $\mathscr{D}_t := \max\{ \| u_t\|, \| u_t'\| \}$. Thus from proof of Lemma \ref{small_stuck_region_lemma}, we know that we can re-write $\hat{u}_t = p(t) - q(t)$ where 
\begin{align}
    &p(t) := (I - \eta \mathcal{H})^{t} \hat{u}_0 = (1 + \eta \gamma)^t r_0 e_1, \nonumber\\
    &q(t) := \eta \sum_{j = 0}^{t-1} (I -\eta \mathcal{H})^{t -1 - j} (\Delta_j \hat{u}_j + \hat{y}_j), \quad \text{ and } \nonumber\\
    &p(0) = \hat{u}_0, \, q(0) = 0. \nonumber
\end{align}
Next, we can inductively show that $(1) \, \| \hat{y}_t\| \leq 2 \gamma L (1 + \eta \gamma)^t r_0$ and $(2) \,\| q(t)\| \leq \| p(t) \|/2$. When $t = 0$, these two conditions are easily verified. Now suppose for $j \leq t -1$, these two conditions are true, then we can immediately obtain that for all $j \leq t-1$,
\begin{equation}
    \| \hat{u}_j \| \leq \| p(j) \| + \| q(j)\| \leq \frac{3}{2}\| p(j) \| \leq \frac{3}{2} (1+\eta\gamma)^j r_0. \label{iqmtrqt}
\end{equation}
Also, different to finite-sum setting, we need a bound of $\| \hat{y}_{sm}\|$. Given that $v_{sm} = \nabla \hat{f}_{\mathcal{B},x} (u_{sm})$ and $v_{sm}' = \nabla \hat{f}_{\mathcal{B},x} (u_{sm}')$, we have 
\begin{align}
    \hat{y}_{sm} &= v_{sm} - \nabla \hat{F}_x(u_{sm}) - v_{sm}' + \nabla \hat{F}_x(u_{sm}') \nonumber\\
    &= \frac{1}{B} \sum_{i \in \mathcal{B}} \big( \nabla \hat{f}_{i,x} (u_{sm}) - \nabla \hat{f}_{i,x} (u_{sm}') - \nabla \hat{F}_x(u_{sm}) +\nabla \hat{F}_x(u_{sm}') \big). \nonumber
\end{align}
For each component term, we have 
\begin{align}
    \| \nabla\hat{f}_{i,x} (u_{sm}) - \nabla\hat{f}_{i,x} (u_{sm}') - \nabla \hat{F}_x(u_{sm}) +\nabla \hat{F}_x(u_{sm}') \| &\leq 2L \| u_{sm} - u_{sm}' \| \nonumber\\
    &= 2L \| \hat{u}_{sm} \| \leq 3 L (1+\eta \gamma)^{sm}r_0, \nonumber
\end{align}
where the last inequality uses \eqref{iqmtrqt}. Also the variance term is bounded as 
\begin{align}
    &\sum_{i \in \mathcal{B}} \mathbb{E}\|\nabla\hat{f}_{i,x} (u_{sm}) - \nabla\hat{f}_{i,x} (u_{sm}') - \nabla \hat{F}_x(u_{sm}) +\nabla \hat{F}_x(u_{sm}')   \|^2 \nonumber\\
    &\leq \sum_{i \in \mathcal{B}} \mathbb{E}\|  \nabla\hat{f}_{i,x} (u_{sm}) - \nabla\hat{f}_{i,x} (u_{sm}') \|^2 \nonumber\\
    &\leq B L^2 \| u_{sm} - u_{sm}'\|^2 \nonumber\\
    &= BL^2 \| \hat{u}_{sm}\|^2 \leq \frac{9}{4} BL^2 (1 + \eta \gamma)^{2sm} r_0^2. \nonumber 
\end{align}
Therefore, we can substitute these two bounds in Bernstein inequality (Lemma \ref{bernsteinlemma}) to bound $\| \hat{y}_{sm}\|$. That is,
\begin{align}
    \text{Pr} \{ \| \hat{y}_{sm} \| \geq \frac{\varsigma}{B} \} &\leq (d+1) \exp \Big( \frac{-\varsigma^2/2}{\sigma^2 + R\varsigma/3} \Big) \nonumber\\
    &\leq (d+1) \exp \Big( \frac{-\varsigma^2/2}{9BL^2 (1+\eta\gamma)^{2sm}r_0^2/4 + 3\sqrt{B}L(1+ \eta \gamma)^{sm} r_0 \varsigma/3} \Big) \nonumber\\
    &\leq \zeta, \nonumber
\end{align}
where the second inequality also uses $\sqrt{B} \geq 1$ and the last inequality holds due to 
\begin{equation}
    \varsigma = \mathcal{O}(\log(d/\vartheta)) L \sqrt{B} (1 + \eta \gamma)^{sm} r_0.    \nonumber
\end{equation}
By a union bound, for all $t \leq \mathscr{T}, s = \lceil \frac{t}{m} \rceil$, we have 
\begin{equation}
    \| \hat{y}_{sm} \| \leq \frac{c_2' L (1+\eta\gamma)^{sm} r_0}{\sqrt{B}} \leq \gamma L (1+ \eta \gamma)^{sm} r_0, \quad \text{ with probability } 1 - \zeta \label{eirqmqwqq}
\end{equation}
where $c_2' = \mathcal{O}(\log(\frac{d\mathscr{T}}{\zeta m})) = \widetilde{\mathcal{O}}(1)$ and the second inequality is due to $B \geq \mathcal{O}(c_2'^2\sigma^2/\epsilon^2) = \widetilde{\mathcal{O}}(\sigma^2/\epsilon^2)$, $\epsilon \leq \delta^2/\rho$, $\delta \leq \gamma$ and assume without loss of generality $\delta \leq 1$. Now we will first prove that the second condition holds for $j = t$.

\textbf{Proof that $ \boldsymbol{\| q(t)\|}$ is bounded by $\boldsymbol{\| p(t) \|/2}$.} Recall we can decompose $\| q(t)\|$ into two terms:
\begin{equation}
    \| q(t) \| \leq \| \eta \sum_{j = 0}^{t-1} (I -\eta \mathcal{H})^{t -1 - j} \Delta_j \hat{u}_j \| + \| \eta \sum_{j = 0}^{t-1} (I -\eta \mathcal{H})^{t -1 - j} \hat{y}_j \|. \nonumber
\end{equation}
The first term is bounded the same way as in Lemma \ref{small_stuck_region_lemma}. Consider parameter settings, $r \leq \frac{\delta}{c_2 \rho}$, $\mathscr{T} \leq 2 \log_{\alpha} \big( \frac{8\delta\sqrt{d}}{c_2 \rho \nu r}\big)/(\eta \gamma)$, $c_2 \geq 24 \log_{\alpha} \big( \frac{8\delta\sqrt{d}}{c_2 \rho \nu r} \big)$, where $1 \leq \alpha \leq e$ satisfying $\log_\alpha (1+ \eta \gamma) > \gamma$. Then we have 
\begin{equation}
     \| \eta \sum_{j = 0}^{t-1} (I -\eta \mathcal{H})^{t -1 - j} \Delta_j \hat{u}_j \| \leq \frac{1}{4} \| p(t)\|. \nonumber
\end{equation}
The second term can be bounded as 
\begin{align}
    \| \eta \sum_{j = 0}^{t-1} (I -\eta \mathcal{H})^{t -1 - j} \hat{y}_j \| &\leq \eta \sum_{j = 0}^{t-1} (1+\eta \gamma)^{t-1-j} \| \hat{y}_j\| \leq 2\eta \sum_{j = 0}^{t-1} (1+\eta \gamma)^{t-1-j} \gamma L (1+ \eta \gamma)^j r_0 \nonumber\\
    &\leq 2 \eta \gamma L \mathscr{T} (1+\eta \gamma)^{t-1} r_0 \leq \frac{1}{4} \| p(t)\|, \nonumber
\end{align}
where we choose $\mathscr{T} \leq 2 \log_{\alpha} \big( \frac{8\delta\sqrt{d}}{c_2 \rho \nu r}\big)/ \gamma$ and $\eta \leq \frac{1}{16 \log_{\alpha} ( \frac{8\delta\sqrt{d}}{c_2 \rho \nu r})L }$. These two results complete the proof that $\| q(t)\| \leq \| p(t)\|/2$. Now we proceed to prove the first condition.

\textbf{Proof that $\boldsymbol{\| \hat{y}_t\|}$ is bounded by $\boldsymbol{2\gamma L (1 + \eta \gamma)^t r_0}$.} 
Denote $z_t = \hat{y}_t - \hat{y}_{t-1}$. Hence we can verify that $\{ \hat{y}_t \}$ is a martingale sequence and $\{ z_t \}$ is its difference sequence. Then we can first bound $\| z_t\|$ by Bernstein inequality. The following result is exactly the same as in Lemma \ref{small_stuck_region_lemma}. With probability $1 - \zeta_t$, 
\begin{equation}
    \| z_t  \| \leq \mathcal{O}\Big( \frac{\log(d/\zeta_t)}{\sqrt{b}} \Big) (L\|\hat{u}_t -\hat{u}_{t-1} \| + \rho \mathscr{D}_t \| \hat{u}_t \| + \rho \mathscr{D}_{t-1} \| \hat{u}_{t-1}\|) =: \varrho_t. \nonumber
\end{equation}
Now we can bound $\|\hat{y}_t \|$. We only need to consider a single epoch because $\| \hat{y}_{sm}\|$ is bounded as in \eqref{eirqmqwqq}. Similarly, set $\zeta_t = \zeta/m$, we have $\| z_t \| \leq \varrho_t$ for all $t = sm+1,..., (s+1)m$. By Azuma-Hoeffding inequality (Lemma \ref{azuma_hoeffding_lemma}), we have for all $t \leq \mathscr{T}$, with probability $1 - \zeta$, 
\begin{align}
    \|\hat{y}_t - \hat{y}_{sm} \| \leq \mathcal{O} \Big( \frac{\log^{\frac{3}{2}}( d\mathcal{T}/\zeta)}{\sqrt{b}} \Big) \sqrt{\sum_{j=sm+1}^t (L \| \hat{u}_t - \hat{u}_{t-1}\| + \rho \mathscr{D}_t \| \hat{u}_t \| + \rho \mathscr{D}_{t-1} \| \hat{u}_{t-1}\|)^2}. \nonumber
\end{align}
This is the same result from \eqref{y_hat_t_bound} except that $ \hat{y}_{sm} \neq 0$. Combining this result and \eqref{eirqmqwqq}, and using a union bound, we have with probability $1 - 2\zeta$, for all $t \leq \mathscr{T}$, 
\begin{align}
    &\| \hat{y}_t \| \leq \| \hat{y}_t - \hat{y}_{sm} \| + \|\hat{y}_{sm} \| \nonumber\\
    &\leq \mathcal{O} \Big( \frac{\log^{\frac{3}{2}}( d\mathcal{T}/\zeta)}{\sqrt{b}} \Big) \sqrt{\sum_{j=sm+1}^t (L \| \hat{u}_t - \hat{u}_{t-1}\| + \rho \mathscr{D}_t \| \hat{u}_t \| + \rho \mathscr{D}_{t-1} \| \hat{u}_{t-1}\|)^2} + \gamma L (1+ \eta \gamma)^{\lceil t/m\rceil m}. \label{omuyuiyoiup}
\end{align}
What remains to be shown is that the right hand side of \eqref{omuyuiyoiup} is bounded. Recall similarly from \eqref{mbnnb} and \eqref{latter_bound_yhat}, we have 
\begin{align}
    L \| \hat{u}_t - \hat{u}_{t-1} \| &\leq \Big( \frac{3 \eta \log(\mathscr{T})}{c_2} + 3\eta L \log(\mathscr{T}) \Big) \gamma L ( 1+\eta\gamma)^t r_0, \text{ and } \nonumber\\
    \rho \mathscr{D}_t \| \hat{u}_t \| + \rho \mathscr{D}_{t-1} \| \hat{u}_{t-1}\| &\leq \frac{6\delta}{c_2} (1 + \eta \gamma)^t r_0, \nonumber
\end{align}
where there is a slight difference in one of the constants due to now $\| \hat{y}_j\| \leq 2 \eta \gamma (1+\eta \gamma)^t r_0$. Therefore, we obtain
\begin{align}
    \| \hat{y}_t \| &\leq C_1 \Big( \frac{3\eta \log(t)}{c_2} + 3 \eta L \log(t) + \frac{6 }{c_2 L}\Big) \gamma L (1 + \eta \gamma)^t r_0 + \gamma L (1+\eta\gamma)^{\lceil t/m\rceil m} r_0 \nonumber\\
    &\leq 2 \gamma L (1+ \eta \gamma)^{t} r_0 \nonumber
\end{align}
where $C_1 := \mathcal{O}(\log^{\frac{3}{2}}(d\mathscr{T}/\zeta))$. (Note $c_1 = \mathcal{O}(\log^{\frac{3}{2}}(dm/\vartheta))$ in Lemma \ref{large_gradien_descent_lemma_online}. The second inequality uses the choice that $\eta \leq \frac{1}{8L \log(\mathcal{T}) C_1} \leq \frac{1}{8L \log({t}) C_1}$ and $c_2 \geq \frac{12C_1}{L}, C_1 \geq 1$. This completes the proof for the first condition. Now we can prove the second condition. 

The contradiction is the same as in the proof of Lemma \ref{small_stuck_region_lemma} and hence skipped for clarity. Note similar to the argument in Lemma \ref{small_stuck_region_lemma}, we choose $\mathscr{T} = 2 \log_{\alpha} \big( \frac{8\delta\sqrt{d}}{c_2 \rho \nu r}\big)/\delta$ so that the result still holds.
\end{proof}

\begin{lemma}[Descent around saddle points]
\label{descent_around_saddle_point_lemma_onlint}
Let $x \in \mathcal{M}$ satisfy $\| \emph{grad}f_{\mathcal{B}}(x) \| \leq \epsilon$ and $\lambda_{\min}(\nabla^2 \hat{F}_x(0))$ $\leq - \delta$. With the same settings as in Lemma \ref{small_stuck_region_online}, the two coupled sequences satisfy, with probability $1 - \zeta$, 
\begin{equation}
    \exists \, t \leq \mathscr{T}, \, \max\{ \hat{F}_x(u_0) - \hat{F}_x(u_t), \hat{F}_x(u_0') - \hat{F}_x(u_t') \} \geq 2 \mathscr{F}, \nonumber
\end{equation}
where $\mathscr{F} = \frac{\delta^3}{2c_3 \rho^2}$, $c_3 = \frac{16 \log_{\alpha} ( \frac{8\delta\sqrt{d}}{c_2 \rho \nu r} )c_2^2}{c_1 L}$ and $c_1, c_2$ are defined in Lemma \ref{large_gradien_descent_lemma_online} and \ref{small_stuck_region_online} respectively. 
\end{lemma}
\begin{proof}
We will prove this result by contradiction. Suppose the contrary holds. That is, 
\begin{equation}
    \forall \, t \leq \mathscr{T}, \, \max \{ \hat{F}_x(u_0) - \hat{F}_x(u_t), \hat{F}_x(u_0') - \hat{F}_x(u_t') \} \leq 2 \mathscr{F}. \label{qwteywqreu}
\end{equation}
Then we first notice that under this condition, both $\{u_t\}$ and $\{ u_t'\}$ do not escape the constraint ball $\mathbb{B}_x(D)$. This is verified by contradiction. Assume, without loss of generality, at iteration $j \leq \mathscr{T}$, $\| u_j \| \geq D$. In this case, $\textsf{TSSRG}(x, u_0, \mathscr{T})$ returns $R_x(u_{j-1} - \alpha \eta v_{t-1})$ with $\alpha \in (0,1)$ such that $\| v_{j-1} - \alpha \eta v_{t-1} \| = D$. Hence, by the similar logic as in \eqref{imimiinihi},
\begin{align}
    D &= \| u_{j-1} - \alpha \eta v_{t-1} \| \leq \| u_{j-1} - \alpha \eta v_{t-1} - u_0 \| + \| u_0\| \nonumber\\
    &\leq \sqrt{\frac{4j (\hat{F}_x(u_0) - \hat{F}_x(u_t) )}{c_1L} + \frac{2 \eta j^2 c_1'^2 \sigma^2}{c_1 L B}} + r \leq \sqrt{\frac{8\mathscr{T} \mathscr{F}}{c_1 L} + \frac{2\eta \mathscr{T}^2 c_1'^2 \sigma^2}{c_1 L B} } + r \nonumber\\
    &\leq \sqrt{\frac{8 \log_\alpha (\frac{8\delta \sqrt{d}}{c_2 \rho \nu r})}{\delta c_1 L} \cdot \frac{\delta^3}{c_3 \rho^2} + \frac{8\eta \log_\alpha^2(\frac{8\delta \sqrt{d}}{c_2 \rho \nu r} ) c_1'^2\sigma^2}{\delta^2 c_1 L B}} + r \leq \sqrt{\frac{\delta^2}{2c_2^2\rho^2} + \frac{\log_\alpha(\frac{8\delta\sqrt{d}}{c_2 \rho \nu r}) c_1'^2 \sigma^2}{2 L^2 c_1 \delta^2 B}} + r \nonumber\\
    &\leq \sqrt{\frac{\delta^2}{2 c_2^2 \rho^2} + \frac{\delta^2}{2 c_2^2 \rho^2}} + r = \frac{\delta}{c_2 \rho} + r. \nonumber
\end{align}
where we set parameters, $\mathscr{F}= \frac{\delta^3}{2c_3\rho^2}, \mathscr{T} = 2 \log_{\alpha} \big( \frac{8\delta\sqrt{d}}{c_2 \rho \nu r}\big)/\delta$, $c_3 = \frac{16 \log_{\alpha} ( \frac{8\delta\sqrt{d}}{c_2 \rho \nu r} )c_2^2}{c_1 L}$, $\eta \leq \frac{1}{16 \log_\alpha (\frac{8\delta\sqrt{d}}{c_2 \rho \nu r}) L}$ and $B \geq \mathcal{O}(\frac{\log_\alpha(\frac{8\delta\sqrt{d}}{c_2 \rho \nu r}) c_2^2 c_1'^2\sigma^2}{L^2 c_1\epsilon^2}) = \widetilde{\mathcal{O}}(\frac{\sigma^2}{\epsilon^2})$, $\epsilon \leq \delta^2/\rho$. However, we know that $\frac{\delta}{c_2 \rho} + r \leq D$, which gives a contradiction. Therefore, we claim that $\{ u_t\}$ and $\{ u_t'\}$ stay within the constraint ball within $\mathscr{T}$ steps. 

Also based on Lemma \ref{small_stuck_region_online}, assume $\| u_T - u_0\| \geq \frac{\delta}{c_2 \rho}$ for some $T \leq \mathscr{T}$. Then from Lemma \ref{improve_localize_online}, we know that 
\begin{align}
    \hat{F}_x(u_0) - \hat{F}_x(u_T) \geq \frac{c_1 L}{4T}\| u_T - u_0 \|^2 - \frac{\eta T c_1'^2 \sigma^2}{2B} &\geq \frac{c_1 L \delta^3}{8 \log_\alpha (\frac{8\delta\sqrt{d}}{c_2 \rho \nu r}) c_2^2\rho^2} - \frac{\log_\alpha(\frac{8\delta\sqrt{d}}{c_2 \rho \nu r}) c_1'^2 \sigma^2}{B \delta} \nonumber\\
    &\geq \frac{2\delta^3}{c_3 \rho^2} - \frac{\delta^3}{c_3 \rho^2} = \frac{\delta^3}{c_3 \rho^2} = 2 \mathscr{F}, \nonumber
\end{align}
where we consider $T \leq \mathscr{T}, \eta \leq 1$ and $B \geq \mathcal{O}(\frac{\log_\alpha^2(\frac{8\delta\sqrt{d}}{c_2 \rho \nu r}) c_2^2 c_1'^2  \sigma^2}{\epsilon^2}) = \widetilde{\mathcal{O}}(\frac{\sigma^2}{\epsilon^2})$, $\epsilon \leq \delta^2/\rho$. This contradicts \eqref{qwteywqreu} and hence the proof is complete. 
\end{proof}

\begin{lemma}[Escape stuck region]
\label{escape_stuck_region_lemma_online}
    Let $x \in \mathcal{M}$ satisfy $\| \emph{grad}f_{\mathcal{B}}(x) \| \leq \epsilon$ and $\lambda_{\min}(\nabla^2 \hat{F}_x(0)) \leq - \delta$. Given that result in Lemma \ref{descent_around_saddle_point_lemma_onlint} holds and choose perturbation radius $r \leq \min\{ \frac{\delta^3}{4c_3 \rho^2 \epsilon}, \sqrt{\frac{\delta^3}{2 c_3 \rho^2 L}} \}$, we have a sufficient decrease of function value with high probability:
    \begin{equation}
        {F}(\textsf{TSSRG}(x,u_0, \mathscr{T}) ) - F(x) \leq -\mathscr{F} \quad \text{ with probability } 1- \nu \nonumber
    \end{equation}
\end{lemma}
\begin{proof}
    The proof is exactly the same as Lemma \ref{escape_stuck_region_lemma_online} and hence skipped. One remark is that when small gradient condition $\| \text{grad}f_{\mathcal{B}(x)} \| \leq \epsilon$ is triggered, we already have $\| \text{grad}F(x)\| \leq \epsilon/2$ with high probability from Lemma \ref{large_gradien_descent_lemma_online}. 
\end{proof}


\begin{thebibliography}{32}
\providecommand{\natexlab}[1]{#1}
\providecommand{\url}[1]{\texttt{#1}}
\expandafter\ifx\csname urlstyle\endcsname\relax
  \providecommand{\doi}[1]{doi: #1}\else
  \providecommand{\doi}{doi: \begingroup \urlstyle{rm}\Url}\fi

\bibitem[Absil et~al.(2009)Absil, Mahony, and Sepulchre]{AbsilOPTManifold2009}
P-A Absil, Robert Mahony, and Rodolphe Sepulchre.
\newblock \emph{Optimization algorithms on matrix manifolds}.
\newblock Princeton University Press, 2009.

\bibitem[Agarwal et~al.(2018)Agarwal, Boumal, Bullins, and
  Cartis]{AgarwalARC2018}
Naman Agarwal, Nicolas Boumal, Brian Bullins, and Coralia Cartis.
\newblock Adaptive regularization with cubics on manifolds.
\newblock \emph{arXiv preprint arXiv:1806.00065}, 2018.

\bibitem[Ahn and Sra(2020)]{AhnRAGD2020}
Kwangjun Ahn and Suvrit Sra.
\newblock From {N}esterov's estimate sequence to {R}iemannian acceleration.
\newblock \emph{arXiv preprint arXiv:2001.08876}, 2020.

\bibitem[Allen-Zhu and Li(2018)]{AllenNEON22018}
Zeyuan Allen-Zhu and Yuanzhi Li.
\newblock Neon2: Finding local minima via first-order oracles.
\newblock In \emph{{Advances in Neural Information Processing Systems}}, pages
  3716--3726, 2018.

\bibitem[Arjevani et~al.(2019)Arjevani, Carmon, Duchi, Foster, Srebro, and
  Woodworth]{ArjevaniLBOnline2019}
Yossi Arjevani, Yair Carmon, John~C Duchi, Dylan~J Foster, Nathan Srebro, and
  Blake Woodworth.
\newblock Lower bounds for non-convex stochastic optimization.
\newblock \emph{arXiv preprint arXiv:1912.02365}, 2019.

\bibitem[Bonnabel(2013)]{BonnabelSGD2013}
Silvere Bonnabel.
\newblock Stochastic gradient descent on {R}iemannian manifolds.
\newblock \emph{IEEE Transactions on Automatic Control}, 58\penalty0
  (9):\penalty0 2217--2229, 2013.

\bibitem[Boumal(2020)]{BoumalIntroductionMO2020}
Nicolas Boumal.
\newblock An introduction to optimization on smooth manifolds.
\newblock 2020.

\bibitem[Boumal et~al.(2019)Boumal, Absil, and Cartis]{BoumalRGD2019}
Nicolas Boumal, Pierre-Antoine Absil, and Coralia Cartis.
\newblock Global rates of convergence for nonconvex optimization on manifolds.
\newblock \emph{IMA Journal of Numerical Analysis}, 39\penalty0 (1):\penalty0
  1--33, 2019.

\bibitem[Chung and Lu(2006)]{ChungBound2006}
Fan Chung and Linyuan Lu.
\newblock Concentration inequalities and martingale inequalities: a survey.
\newblock \emph{Internet Mathematics}, 3\penalty0 (1):\penalty0 79--127, 2006.

\bibitem[Criscitiello and Boumal(2020)]{CriscitielloAFOM2020}
Chris Criscitiello and Nicolas Boumal.
\newblock An accelerated first-order method for non-convex optimization on
  manifolds.
\newblock \emph{arXiv preprint arXiv:2008.02252}, 2020.

\bibitem[Criscitiello and Boumal(2019)]{CriscitielloEESPManifold2019}
Christopher Criscitiello and Nicolas Boumal.
\newblock Efficiently escaping saddle points on manifolds.
\newblock In \emph{{Advances in Neural Information Processing Systems}}, pages
  5987--5997, 2019.

\bibitem[Du et~al.(2017)Du, Jin, Lee, Jordan, Singh, and
  Poczos]{DuGDEscape2017}
Simon~S Du, Chi Jin, Jason~D Lee, Michael~I Jordan, Aarti Singh, and Barnabas
  Poczos.
\newblock Gradient descent can take exponential time to escape saddle points.
\newblock In \emph{{Advances in Neural Information Processing Systems}}, pages
  1067--1077, 2017.

\bibitem[Fang et~al.(2018)Fang, Li, Lin, and Zhang]{FangSPIDER2018}
Cong Fang, Chris~Junchi Li, Zhouchen Lin, and Tong Zhang.
\newblock Spider: Near-optimal non-convex optimization via stochastic
  path-integrated differential estimator.
\newblock In \emph{{Advances in Neural Information Processing Systems}}, pages
  689--699, 2018.

\bibitem[Ge et~al.(2019)Ge, Li, Wang, and Wang]{GePSVRG2019}
Rong Ge, Zhize Li, Weiyao Wang, and Xiang Wang.
\newblock Stabilized svrg: Simple variance reduction for nonconvex
  optimization.
\newblock \emph{arXiv preprint arXiv:1905.00529}, 2019.

\bibitem[Hoeffding(1994)]{HoeffdingBound1994}
Wassily Hoeffding.
\newblock Probability inequalities for sums of bounded random variables.
\newblock In \emph{{The Collected Works of Wassily Hoeffding}}, pages 409--426.
  Springer, 1994.

\bibitem[Hosseini and Sra(2020)]{HosseiniSGD2020}
Reshad Hosseini and Suvrit Sra.
\newblock An alternative to em for {G}aussian mixture models: Batch and
  stochastic {R}iemannian optimization.
\newblock \emph{Mathematical Programming}, 181\penalty0 (1):\penalty0 187--223,
  2020.

\bibitem[Huang et~al.(2015)Huang, Gallivan, and Absil]{HuangBFGS2015}
Wen Huang, Kyle~A Gallivan, and P-A Absil.
\newblock {A Broyden class of quasi-Newton methods for Riemannian
  optimization}.
\newblock \emph{SIAM Journal on Optimization}, 25\penalty0 (3):\penalty0
  1660--1685, 2015.

\bibitem[Jin et~al.(2017)Jin, Ge, Netrapalli, Kakade, and Jordan]{JinPGD2017}
Chi Jin, Rong Ge, Praneeth Netrapalli, Sham~M Kakade, and Michael~I Jordan.
\newblock How to escape saddle points efficiently.
\newblock \emph{arXiv preprint arXiv:1703.00887}, 2017.

\bibitem[Jin et~al.(2018)Jin, Netrapalli, and Jordan]{JinPAGD2018}
Chi Jin, Praneeth Netrapalli, and Michael~I Jordan.
\newblock Accelerated gradient descent escapes saddle points faster than
  gradient descent.
\newblock In \emph{{Conference On Learning Theory}}, pages 1042--1085. PMLR,
  2018.

\bibitem[Jin et~al.(2019)Jin, Netrapalli, Ge, Kakade, and
  Jordan]{JinPGD_PSGD2019}
Chi Jin, Praneeth Netrapalli, Rong Ge, Sham~M Kakade, and Michael~I Jordan.
\newblock On nonconvex optimization for machine learning: Gradients,
  stochasticity, and saddle points.
\newblock \emph{arXiv preprint arXiv:1902.04811}, 2019.

\bibitem[Kasai et~al.(2018)Kasai, Sato, and Mishra]{KasaiRSRG2018}
Hiroyuki Kasai, Hiroyuki Sato, and Bamdev Mishra.
\newblock Riemannian stochastic recursive gradient algorithm.
\newblock In \emph{International Conference on Machine Learning}, pages
  2516--2524, 2018.

\bibitem[Li(2019)]{LiSSRGD2019}
Zhize Li.
\newblock {SSRGD: Simple stochastic recursive gradient descent for escaping
  saddle points}.
\newblock In \emph{{Advances in Neural Information Processing Systems}}, pages
  1523--1533, 2019.

\bibitem[Nesterov and Polyak(2006)]{NesterovCR2006}
Yurii Nesterov and Boris~T Polyak.
\newblock {Cubic regularization of Newton method and its global performance}.
\newblock \emph{Mathematical Programming}, 108\penalty0 (1):\penalty0 177--205,
  2006.

\bibitem[Nguyen et~al.(2017{\natexlab{a}})Nguyen, Liu, Scheinberg, and
  Tak{\'a}{\v{c}}]{NguyenSARAH2017}
Lam~M Nguyen, Jie Liu, Katya Scheinberg, and Martin Tak{\'a}{\v{c}}.
\newblock Stochastic recursive gradient algorithm for nonconvex optimization.
\newblock \emph{arXiv preprint arXiv:1705.07261}, 2017{\natexlab{a}}.

\bibitem[Nguyen et~al.(2017{\natexlab{b}})Nguyen, Liu, Scheinberg, and
  Tak{\'a}{\v{c}}]{NguyenSRG2017}
Lam~M Nguyen, Jie Liu, Katya Scheinberg, and Martin Tak{\'a}{\v{c}}.
\newblock Stochastic recursive gradient algorithm for nonconvex optimization.
\newblock \emph{arXiv preprint arXiv:1705.07261}, 2017{\natexlab{b}}.

\bibitem[Reddi et~al.(2016)Reddi, Hefny, Sra, Poczos, and Smola]{ReddiSVRG2016}
Sashank~J Reddi, Ahmed Hefny, Suvrit Sra, Barnabas Poczos, and Alex Smola.
\newblock Stochastic variance reduction for nonconvex optimization.
\newblock In \emph{{International Conference on Machine Learning}}, pages
  314--323, 2016.

\bibitem[Sun et~al.(2019)Sun, Flammarion, and Fazel]{SunPRGD2019}
Yue Sun, Nicolas Flammarion, and Maryam Fazel.
\newblock Escaping from saddle points on {R}iemannian manifolds.
\newblock In \emph{{Advances in Neural Information Processing Systems}}, pages
  7276--7286, 2019.

\bibitem[Tropp(2012)]{TroppVBbounds2012}
Joel~A Tropp.
\newblock User-friendly tail bounds for sums of random matrices.
\newblock \emph{Foundations of Computational Mathematics}, 12\penalty0
  (4):\penalty0 389--434, 2012.

\bibitem[Xu et~al.(2018)Xu, Jin, and Yang]{XuNEON2018}
Yi~Xu, Rong Jin, and Tianbao Yang.
\newblock First-order stochastic algorithms for escaping from saddle points in
  almost linear time.
\newblock In \emph{{Advances in Neural Information Processing Systems}}, pages
  5530--5540, 2018.

\bibitem[Zhang and Sra(2018)]{ZhangRAGD2018}
Hongyi Zhang and Suvrit Sra.
\newblock Towards {R}iemannian accelerated gradient methods.
\newblock \emph{arXiv preprint arXiv:1806.02812}, 2018.

\bibitem[Zhang et~al.(2016)Zhang, Reddi, and Sra]{ZhangRSVRG2016}
Hongyi Zhang, Sashank~J Reddi, and Suvrit Sra.
\newblock Riemannian svrg: Fast stochastic optimization on riemannian
  manifolds.
\newblock In \emph{Advances in Neural Information Processing Systems}, pages
  4592--4600, 2016.

\bibitem[Zhou et~al.(2018)Zhou, Xu, and Gu]{ZhouSNVRG_NEON2018}
Dongruo Zhou, Pan Xu, and Quanquan Gu.
\newblock Finding local minima via stochastic nested variance reduction.
\newblock \emph{arXiv preprint arXiv:1806.08782}, 2018.

\end{thebibliography}
\end{document}